\documentclass[11pt]{article}

\usepackage{amsthm}

\usepackage{hslTR}
\usepackage{graphics}
\usepackage{graphicx}
\usepackage{amssymb}
\usepackage{verbatim}     
\usepackage{amsxtra}
\usepackage{epsfig}
\usepackage{subfigure}
\usepackage{url}
\usepackage{amsmath}
\usepackage{latexsym}
\usepackage{ifthen}
\usepackage{psfrag}
\usepackage{color}
\usepackage{adjustbox}
\usepackage[normalem]{ulem}
\usepackage{float}
\usepackage{tikz}
\usepackage{cuted}
\usepackage{multicol}
\usepackage{bm}
\usepackage{changepage}
\usetikzlibrary{shapes,arrows}
\usepackage{mathtools, cuted}
\usepackage{etoolbox}
\usepackage{enumitem}

\usepackage{rgsEnvironments}
\usepackage{rgsMacros}  

\usetikzlibrary{shapes,arrows}

\DeclareFontFamily{U}{mathc}{}
\DeclareFontShape{U}{mathc}{m}{it}{<->s*[1.03] mathc10}{}
\DeclareMathAlphabet{\mathcalow}{U}{mathc}{m}{it}

\newbool{conf}
\setbool{conf}{false}

\newbool{HSLreport}
\setbool{HSLreport}{true}

\ifbool{HSLreport}{}{
 \addtolength{\oddsidemargin}{-.875in}
	\addtolength{\evensidemargin}{-.875in}
	\addtolength{\textwidth}{1.75in}

	\addtolength{\topmargin}{-.875in}
	\addtolength{\textheight}{1.75in} 
}

\newcommand{\nodes}{\mathcal{V}}

\newcommand{\minus}{\scalebox{0.75}[1.0]{$-$}}
\newcommand{\mequals}{\scalebox{0.75}[1.0]{$=$}}
\newcommand{\mplus}{\scalebox{0.75}[0.75]{$+$}}
\newcommand{\mtop}{\scalebox{0.75}[0.75]{$\top$}}
\newcommand{\m}{\ensuremath{\mathcalow{m}}}

\newtheorem{problem}{Problem}[section]

\tikzstyle{block} = [draw, fill=blue!20, rectangle, 
    minimum height=3em, minimum width=6em]
\tikzstyle{sum} = [draw, fill=blue!20, circle, node distance=1cm]
\tikzstyle{input} = [coordinate]
\tikzstyle{output} = [coordinate]
\tikzstyle{pinstyle} = [pin edge={to-,thin,black}]

\begin{document}

\ifbool{HSLreport}{

{
\ititle{An Adaptive Hybrid Control Algorithm for Sender-Receiver Clock Synchronization}
\title{\bf An Adaptive Hybrid Control Algorithm for Sender-Receiver Clock Synchronization}
\iauthor{
  Marcello Guarro \\
  {\normalsize mguarro@ucsc.edu} \\
    Ricardo Sanfelice \\
  {\normalsize ricardo@ucsc.edu}}
\idate{\today{}} 
\iyear{2020}
\irefnr{01}
\makeititle}

\maketitle

\begin{abstract}
This paper presents an innovative hybrid systems approach to the sender-receiver synchronization of timers. Via the hybrid systems framework, we unite the traditional sender-receiver algorithm for clock synchronization with an online, adaptive strategy to achieve synchronization of the clock rates to exponentially synchronize a pair of clocks connected over a network. Following the conventions of the algorithm, clock measurements of the nodes are given at periodic time instants, and each node uses these measurements to achieve synchronization. For this purpose, we introduce a hybrid system model of a network with continuous and impulsive dynamics that captures the sender-receiver algorithm as a state-feedback controller to synchronize the network clocks. Moreover, we provide sufficient design conditions that ensure attractivity of the synchronization set.
\end{abstract}
}{
\begin{frontmatter}



\title{An Adaptive Hybrid Control Algorithm for Sender-Receiver Clock Synchronization} 


\author[First]{Marcello Guarro}  
\author[First]{Ricardo Sanfelice}

\affiliation[First]{organization={Department of Electrical and Computer Engineering University of California},
            addressline={1156 High Street}, 
            city={Santa Cruz},
            postcode={95060}, 
            state={CA},
            country={USA}}


\begin{abstract}                
This paper presents an innovative hybrid systems approach to the sender-receiver synchronization of timers. Via the hybrid systems framework, we unite the traditional sender-receiver algorithm for clock synchronization with an online, adaptive strategy to achieve synchronization of the clock rates to exponentially synchronize a pair of clocks connected over a network. Following the conventions of the algorithm, clock measurements of the nodes are given at periodic time instants, and each node uses these measurements to achieve synchronization. For this purpose, we introduce a hybrid system model of a network with continuous and impulsive dynamics that captures the sender-receiver algorithm as a state-feedback controller to synchronize the network clocks. We provide sufficient design conditions that ensure the synchronization set is globally attractive for the system model. Moreover, we present a model extension that considers the application of the algorithm in a multi-agent setting. Numerical examples are also presented to show feasibility of the algorithm in both nominal and perturbed system settings. \end{abstract}

\tnotetext[mytitlenote]{This research partially supported by NSF Grants no. ECS-1710621 and CNS-1544396, by AFOSR Grants no. FA9550-16-1-0015, FA9550-19-1-0053, and Grant no. FA9550-19-1-0169, and by CITRIS and the Banatao Institute at the University of California.}

\begin{highlights}
\item Introduction of a hybrid systems approach to model the sender-receiver synchronization of timers via the hybrid systems framework. Sufficient design conditions are provided to ensure a synchronization set of interest is globally attractive for the system model. 
\item An extension to the model that considers the application of the algorithm in a multi-agent setting is also presented with numerical examples to demonstrate its feasibility.
\end{highlights}

\begin{keyword}
Hybrid and switched systems modeling \sep Control over networks \sep Sensor networks
\end{keyword}

\end{frontmatter}}

\section{Introduction}

\subsection{Motivation}

The consensus on a common timescale in a distributed system is an essential requisite for any coordinated system whose algorithm or task depends on the event-ordering of its input. Moreover, distributed algorithms that interact with a dynamical system have the additional requirement of having to know the precise moment an event occurs to ensure desired system function, stability, and safety. This has become increasingly apparent in the deployment of large and complex cyber-physical systems such as autonomous vehicles \cite{samii2018level}, aerospace avionics, and distributed manufacturing robotic systems \cite{eidson2006measurement}.

It has been well established in the literature on networked control systems, that the lack of consensus on a shared timescale among distributed agents can result in performance issues that adversely affect system stability, see \cite{graham2004clock} and \cite{zhang2001ncs}. In particular, time delays due to sampling, transmission, and computation result in the loss of concurrency between the dynamic process events of the plant and computational process events of the system planner or controller (see \cite{nilsson1998real}). To address this issue, system events are time-stamped via clocks local to the sensors and actuators that interact with the physical systems as noted in \cite{nilsson1998real} and the survey papers \cite{zhang2001ncs}, \cite{hespanha2007survey}. However, the success of this strategy relies on the existence of a common timescale among the components and agents in the distributed system. To ensure consensus on a common timescale, the system is coupled with a clock synchronization subsystem that periodically synchronizes the clocks to ensure their relative error is within an acceptable tolerance that is sufficient for desired system performance. 

The design of such a subsystem, however, is nontrivial as network delays, nonuniform clock drifts, and the lack of an absolute time reference, present a unique set of challenges to the clock synchronization problem as noted in \cite{wu2010clock}, \cite{sundararaman2005clock}, and \cite{simeone2008distributed}. Network delays, such as those due to computation or transmission processes, are often the key challenge in synchronization due to the inherent difficulty of their estimation. Of particular concern, however, is the influence of the clock synchronization subsystem on ensuring the stability and robustness guarantees of the larger system as observed in \cite{graham2004clock} and \cite{lavalle2007time}. In \cite{guarro2018state}, we demonstrate a sufficient finite-time convergence condition on the clock synchronization subsystem in a time-stamp aware observer to ensure asymptotic convergence of the estimation error.

In this paper, we are interested in designing a clock synchronization algorithm using two-way communication protocols with tractable design conditions and performance metrics to meet the current and projected demands of distributed system design. In particular, we are interested in the design of algorithm that addresses the following challenges:

\begin{itemize}
\item \emph{Communication delays:} the physical nature and operation of communication networks introduces a variety of stochastic and deterministic delays that adversely affect accurate synchronization of time. As previously noted,  two primary sources of communication delay are those that relate to transmission and computation (also known as residence delay).
\item \emph{Clocks with different rates of change:} typical clocks deployed in a distributed system are inherently imprecise devices whose frequency or clock rate is subjected to noise disturbances due to physical, environmental, and manufacturing constraints.
\item \emph{Performance guarantees:} noting the influence of delays on networked control systems and the time-stamping approach used to alleviate them, guarantees on the synchronization performance are necessary when the control system is subjected to fast sampling periods, as noted in \cite{nakamura2008synchronization}.
\end{itemize}

\subsection{Related Work}

Among the number of existing algorithms for clock synchronization, sender-receiver (or two-way) based synchronization algorithm underpin several of the most popular clock synchronization protocols including Network Time Protocol (NTP) in \cite{mills1991internet} , Precision Time Protocol (PTP) in \cite{5} , and the Timing-sync Protocol for Sensor Networks (TPSN) in \cite{ganeriwal2003timing}. 

The core algorithm, upon which these protocols are based, relies on the existence of a known reference that is either injected to the system or provided by an elected agent in the distributed system; synchronization is then achieved through a series of chronologically ordered and time stamped two-way message exchanges between each synchronizing node and the designated reference. With sufficient information from the exchanged messages and underlying assumptions on the clocks and communication delays, the relative differences in the clock rates and offset can be estimated and applied as a correction to the clock of the synchronizing node, see \cite{freris2010fundamental}. However, while the difference in the output can be determined and implemented online, the relative clock rate is estimated through offline filtering techniques (see \cite{mills1991internet}) or least-squares estimation (see \cite{wu2010clock}). 

Each of the aforementioned protocols, however, utilize different strategies in regards to the availability of the algorithm and the layer of implementation. For instance, the Network Time Protocol is an ``always-on" implementation that runs entirely as a system process in the software layer. This level of implementation subjects the protocol to frequent computational delays due to the execution of system processes that have higher priority. These delays contribute to timing inaccuracy that renders NTP unfit for networked control systems with fast sampling periods, see \cite{nakamura2008synchronization}.

Improving upon NTP to address its concerns and meet the demands of time-sensitive distributed system, the Precision Time Protocol utilizes a hybrid implementation of software and hardware to improve the synchronization accuracy. The protocol utilizes time-stamping of the exchanged messages at the hardware layer to minimize the computational delays associated with software time-stamps on the exchanged messages.

The Timing-sync Protocol for Sensor Networks seeks to address the scalability issues posed by the NTP and PTP protocols by allowing the algorithm to work on an intermittent schedule. The intermittent strategy enable its use in low-energy sensor networks with limited computational capacity at the cost of synchronization accuracy.

Finally, while there are more recent consensus-based synchronization algorithms such as \cite{7}, \cite{carli2014TAC}, \cite{bolognani2015randomized}, \cite{garone2015clock}, and the authors own \cite{guarro2021hyntp}, we would like to emphasize that the scope of this work is specific to synchronization algorithms that operate via a sender-receiver-based messaging protocol. Most consensus-based algorithms assume asymmetric communication protocols with one-way messaging between any two agents.

\subsection{Contributions and Organization} 
 
In this paper, we present a hybrid systems approach to sender-receiver synchronization with an, online, adaptive method to synchronize the clock rates. We show that our algorithm exponentially synchronizes a pair of clocks connected over a network while preserving the messaging protocols and network dynamics of traditional sender-receiver algorithms. 

Our proposed solution provides a Lyapunov-based convergence analysis to a set in which the clocks are synchronized with sufficient conditions ensuring their synchronization.  In particular the main contributions of this paper are given as follows: 
\begin{itemize}
\item In Section \ref{sec:model}, a hybrid system model of the sender-receiver synchronization algorithm using the framework proposed in \cite{4} is presented. The proposed model captures the continuous dynamics of the clock states and the hybrid dynamics of the networking protocol by which the timing messages are exchanged for a pair of system nodes to achieve synchronization.
\item In Section \ref{sec:prop_of_hs}, we show, through the satisfaction of some basic conditions on the system model, that the algorithm is finite-time attractive to a forward invariant set of interest that represents the correct initialization of the algorithm. 
\item In Section \ref{sec:main}, we provide sufficient conditions on the algorithm parameters to show asymptotic attractivity of the hybrid system to a set of interest representing synchronization of the clocks from the initialization set. Furthermore, we characterize the bound for solution trajectories to the systems in terms of  parameters that can be used for algorithm design.
\item  In Section \ref{sec:multi_agent}, we present a multi-agent extension of the proposed model to cover the case of synchronizing the nodes on an $n$-node network. The feasibility of this multi-agent model is validated with a numerical example of the simulated system.
\end{itemize} 

Unlike the existing algorithms of NTP, PTP, and TPSN, we emphasize to the reader that previous analyses on sender-receiver synchronization have only provided results to their feasibility and that the literature lacks formal results that characterize its performance in a dynamical system setting. 

In addition to the highlighted sections, this paper is organized as follows: Section \ref{sec:prelim_sendRec} introduces the sender-receiver algorithm as presented in the literature. Section \ref{sec:motivation} outlines the motivation for this paper. Section \ref{sec:prelim} presents some preliminary material on hybrid systems. Section \ref{sec:model} formally introduces the problem under consideration and the hybrid model that solves it. Section \ref{sec:main} details the main results, while Section \ref{sec:num} provides numerical examples. \ifbool{conf}{Due to space constraints, the proofs of the results along with other details have been omitted and will be published elsewhere.}{ 

We inform the reader that this work is an extension of our preliminary conference paper \cite{guarro2020adaptive}.  In particular, we include the full proofs of the main results in addition to incremental results and their associated proofs that were originally omitted from the preliminary paper.  Moreover, the proofs and material in Section \ref{sec:main} are new. In addition, new examples for the nominal and multi-agent setting are presented in Section \ref{sec:num}.}

\textit{Notation}: In this paper the following notation and definitions will be used. The set of natural numbers including zero, i.e., $\{0,1,2,\ldots\}$ is denoted by $\mathbb{N}$. The set of natural numbers is denoted as $\mathbb{N}_{> 0}$, i.e., $\mathbb{N}_{> 0} = \{1,2,\ldots\}$. The set of real numbers is denoted as $\reals$. The set of non-negative real numbers is denoted by $\mathbb{R}_{\geq 0}$, i.e., $\mathbb{R}_{\geq 0} = [0, \infty )$. The $n$-dimensional Euclidean space is denoted $\mathbb{R}^n$. Given topological spaces $A$ and $B$, $F: A \rightrightarrows B$ denotes a set-valued map from $A$ to $B$. For a matrix $A \in \mathbb{R}^{n \times m}$, $A^\top$ denotes the transpose of $A$. Given a vector $x \in \mathbb{R}^n$, $|x|$ denotes the Euclidean norm. Given two vectors $x \in \mathbb{R}^n$ and $y \in \mathbb{R}^m$, $(x,y) = [ x^\top \hspace{2mm} y^\top \hspace{1mm} ]^\top$. For two symmetric matrices $A \in \mathbb{R}^{n \times m}$ and $B \in \mathbb{R}^{n \times m}$, $A \succ B$ means that $A - B$ is positive definite, conversely $A \prec B$ means that $A - B$ is negative definite. Given a function $f : \mathbb{R}^n \rightarrow \mathbb{R}^m$, the range of $f$ is given by $\mbox{rge } f := \{ y \hspace{1mm} | \hspace{1mm} \exists \hspace{1mm} x \mbox{ with } y = f(x) \}$.

\ifbool{conf}{\section{Preliminaries on the Sender-Receiver Algorithm} \label{sec:prelim_sendRec}

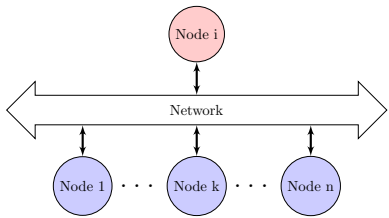
\begin{figure}
\centering
\begin{adjustbox}{max width=0.4\textwidth}
\begin{tikzpicture}[auto, node distance=0.5cm,>=latex']
    \node (node_i) [circle, draw, align=left, fill=red!20]{Node i};
	\node (cloud) [double arrow, draw, align=center, below of = node_i, node distance=2cm, minimum width=1.5cm, minimum height=10cm]{Network};	
	\node (node_k) [circle, draw, align=left, fill=blue!20, below of=cloud, node distance=2cm]{Node k};    
    \node (node_one) [circle, draw, align=left, fill=blue!20, left of=node_k, node distance=3cm]{Node 1};
    \node (node_n) [circle, draw, align=left, fill=blue!20, right of=node_k, node distance=3cm]{Node n};
    
    \node [output, above of=node_one,node distance=1.6cm] (output_one) {};
    \node [output, above of=node_n,node distance=1.6cm] (output_n) {};
    \path (node_one) -- node[auto=false]{\Huge \ldots} (node_k);
    \path (node_k) -- node[auto=false]{\Huge \ldots} (node_n);
    \draw [ultra thick,<->] (node_i) to (cloud);
    \draw [ultra thick,<->] (node_one) to (output_one);
    \draw [ultra thick,<->] (node_n) to (output_n);
    \draw [ultra thick,<->] (node_k) to (cloud);
\end{tikzpicture}
\end{adjustbox}
\caption{General architecture of the system under consideration.}
\label{fig:sys_arch}
\end{figure}

In a network of $n$ nodes, consider nodes $i$ and $k$ in a sender-receiver hierarchy where Node $i$ is a designated reference or parent agent of a synchronizing child agent Node $k$, see Figure \ref{fig:sys_arch}. Each node has an attached internal clock $\tau_{i}, \tau_{k} \in \reals$ whose dynamics are given by
\begin{equation} \label{eqn:clocks_dyn}
\begin{aligned}
& \dot{\tau}_{i} = a_{i} \\
& \dot{\tau}_{k} = a_{k}
\end{aligned}
\end{equation}
\noindent
where $a_{i}$,$a_{k} \in \reals$ denote the respective clock drift or skew. At times $t_j$ for $j \in \mathbb{N}$ (with $t_0 = 0$), nodes $i$ and $k$ exchange timing measurements with embedded timestamps 
\begin{equation} \label{eqn:timestamps}
\begin{aligned}
T_j^i := \tau_i(t_j) \\
T_j^k := \tau_k(t_j)
\end{aligned}
\end{equation}
\noindent 
which, integrating (2), are equal to $$\tau_i(t_j) = a_i t_j + \tau_i(0)$$ $$\tau_k(t_j) = a_k t_j + \tau_k(0)$$ respectively. The goal is to then synchronize the internal clock of Node $k$ to that of Node $i$ using the exchanged timing measurements.

\ifbool{conf}{
For a sequence of time instants $\{t_j\}_{j=1}^{\infty}$ that is assumed to be strictly increasing and unbounded, at each $t_j$ the standard sender-receiver synchronization algorithm as described in the literature (see \cite{wu2010clock}, \cite{freris2009model}, and \cite{eidson2006measurement}) is given as follows:

\begin{enumerate}[label={(P\arabic*})]
\item \label{itm:one} At time $t_j$, Node $i$ broadcasts a synchronization message with its time $T_j^i$ to Node $k$.
\item \label{itm:two} At time $t_{j+1}$, Node $k$ receives the synchronization message and records its time of arrival $T_{j+1}^k$
\item \label{itm:three} At time $t_{j+2}$, Node $k$ sends a response message with timestamp $T_{j+2}^k$
\item \label{itm:four} At time $t_{j+3}$, Node $i$ receives the response message from Node $k$ and records its time of arrival $T_{j+3}^i$
\item \label{itm:five} At time $t_{j+4}$, Node $i$ sends a response receipt message with timestamp $T_{j+4}^i$
\item \label{itm:six} At time $t_{j+5}$, Node $k$ receives the response message from Node $i$ and records its time of arrival $T_{j+5}^k$ and then updates its clock to synchronize with the clock of Node $i$ using the collected timestamps $T_{j}^i$, $T_{j+1}^k$, $T_{j+2}^k$, $T_{j+3}^i$, and $T_{j+4}^i$. 
\end{enumerate}
\noindent
Moreover, as done in the literature, it is assumed that the time elapsed between each time instant is given by 
\begin{equation} \label{eqn:t_bounds}
t_{j+1} - t_j = \begin{cases}
d \hspace{5mm} \forall j \in \{ 2i + 1 : i \in \mathbb{N} \}, j>0 \\
c \hspace{5mm} \forall j \in \{ 2i : i \in \mathbb{N} \}, j>0 
\end{cases}
\end{equation}
\noindent
where $0 < c \leq d$. The constant $c$ defines the residence or response time delay while $d$ defines the propagation delay of the message transmission.
}{
Before introducing the mechanics of the sender-receiver algorithm, we refer the reader to a visual model of the algorithm in Figure \ref{fig:pairwise} as a reference. By assuming the sequence of time instants $\{t_j\}_{j=1}^{\infty}$ is strictly increasing and unbounded, the sender-receiver synchronization algorithm as described in the literature (see \cite{wu2010clock}, \cite{freris2009model}, and \cite{eidson2006measurement}) is given as follows:

\begin{enumerate}[label={(P\arabic*})]
\item \label{itm:one} At time $t_j$, Node $i$ broadcasts a synchronization message with its time $$T_j^i = a_i t_j + \tau_i(0)$$ to Node $k$.
\item \label{itm:two} At time $t_{j+1}$, Node $k$ receives the synchronization message and records its time of arrival $$T_{j+1}^k = a_k t_{j+1} + \tau_k(0)$$
\item \label{itm:three} At time $t_{j+2}$, Node $k$ sends a response message with timestamp $$T_{j+2}^k = a_k t_{j+2} + \tau_k(0)$$
\item \label{itm:four} At time $t_{j+3}$, Node $i$ receives the response message from Node $k$ and records its time of arrival $$T_{j+3}^i = a_i t_{j+3} + \tau_i(0)$$
\item \label{itm:five} At time $t_{j+4}$, Node $i$ sends a response receipt message with timestamp $$T_{j+4}^i = a_i t_{j+4} + \tau_i(0)$$
\item \label{itm:six} At time $t_{j+5}$, Node $k$ receives the response message from Node $i$ and records its time of arrival $$T_{j+5}^k = a_k t_{j+5} + \tau_k(0)$$ and then updates its clock to synchronize with the clock of Node $i$ using the collected timestamps $T_{j}^i$, $T_{j+1}^k$, $T_{j+2}^k$, $T_{j+3}^i$, and $T_{j+4}^i$. 
\end{enumerate}
\noindent
Moreover, as done in the literature, it is assumed that the time elapsed between each time instant is governed by 
\begin{equation} \label{eqn:t_bounds}
t_{j+1} - t_j = \begin{cases}
d \hspace{5mm} \forall j \in \{ 2i + 1 : i \in \mathbb{N} \}, j>0 \\
c \hspace{5mm} \forall j \in \{ 2i : i \in \mathbb{N} \}, j>0 
\end{cases}
\end{equation}
\noindent
where $0 < c \leq d$. The constant $c$ defines the delay associated with the residence or response time associated with message turnaround while $d$ defines the propagation delay associated with message transmission.}

Most pairwise synchronization protocols such as the Network Time Protocol (NTP), Precision Time Protocol (PTP, IEEE 1588), and the Timing-sync Protocol for Sensor Networks (TPSN) assume that the propagation delay in the message transmission from parent to child and child to parent is symmetric. If the propagation delay between the two nodes is asymmetric it introduces an error to the calculated offset correction that cannot be accounted for, see \cite{freris2010fundamental}. Thus, the propagation delay and residence time are assumed to be symmetric.

\begin{figure}
\centering
\includegraphics[width=0.4\textwidth]{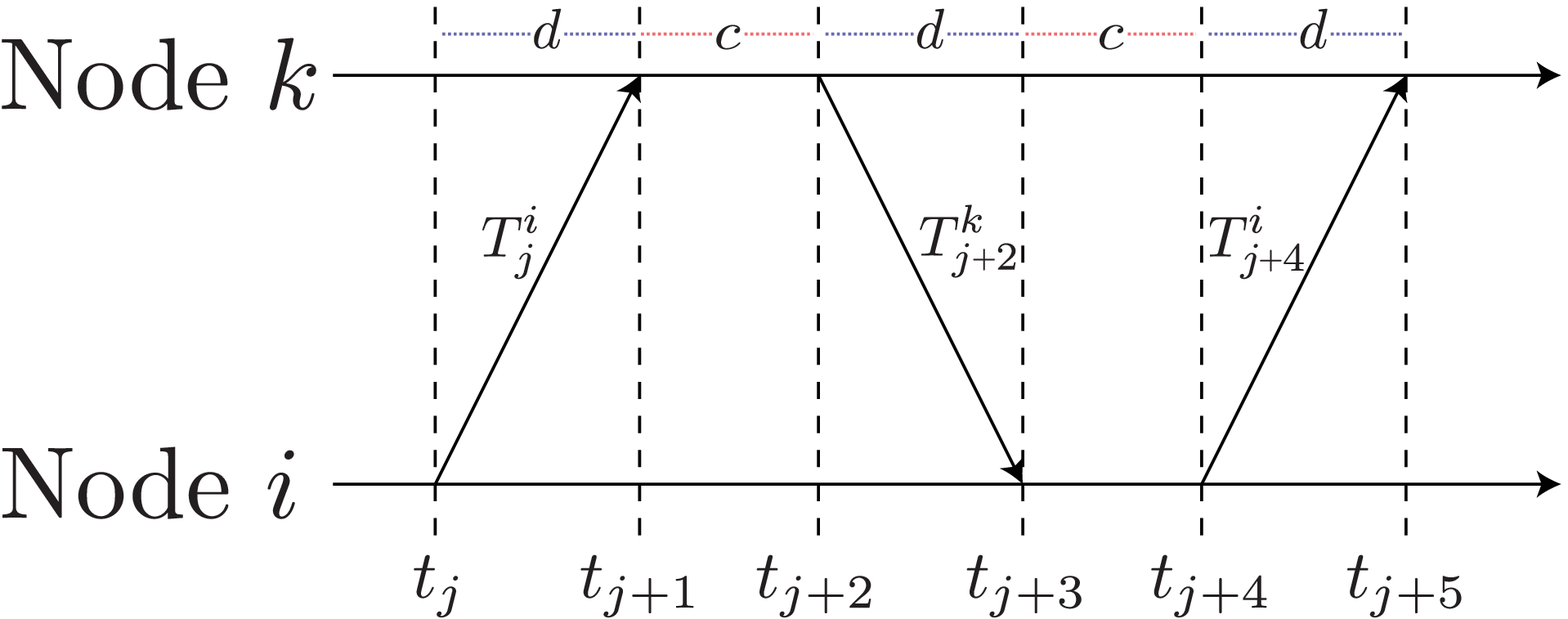}
\caption{\label{fig:pairwise} Diagram illustrating the message exchange between Nodes $i$ and $k$ for the synchronization algorithm.}
\end{figure}


With the available timestamps, at times $t_{j+5}$, the relative offset $\tilde{o} := \tau_i(0) - \tau_k(0)$ is calculated via
\begin{equation} \label{eqn:offset_law1} 
K_{\tilde{o}} = \frac{1}{2} \Big ( (T_{j+3}^i {-} T_{j+2}^k) {-} (T_{j+1}^k {-} T_{j}^i) \Big ) 
\end{equation}
\noindent
by making the appropriate substitutions one has
\begin{align*}
K_{\tilde{o}} & = \frac{1}{2} \Big ( \big ( (a_i t_{j+3} + \tau_i(0)) {-} (a_k t_{j+2} + \tau_k(0)) \big ) \\ 
& \hspace{10mm} - \big ((a_k t_{j+1} + \tau_k(0) ) {-} (a_i t_{j} + \tau_i(0)) \big ) \Big ) \\
& = \frac{1}{2} \Big ( \big ( a_i t_{j+3} {-} a_k t_{j+2} + \tilde{o} \big ) {-} \big ( a_k t_{j+1} {-} a_i t_{j} - \tilde{o} \big ) \Big )
\end{align*}
\noindent
Rearranging terms gives,
\begin{equation} \label{eqn:twoway2} 
\begin{aligned} 
K_{\tilde{o}} & = \tilde{o} + \frac{1}{2} \Big ( \big ( a_i t_{j+3} - a_k t_{j+2} \big ) {-} \big ( a_k t_{j+1} - a_i t_{j} \big ) \Big )
\end{aligned}
\end{equation}
\noindent
If the clock drifts are synchronized, i.e., $a_k = a_i$, then 
\begin{align*} 
K_{\tilde{o}} & = \tilde{o} + \frac{1}{2} \Big ( \big ( a_i(t_{j+3} {-} t_{j+2}) \big ) {-} \big ( a_i(t_{j+1} {-} t_{j}) \big ) \Big )
\end{align*}
Then, by noting the bounds on the time elapsed between time instants $t_j$, as given in (\ref{eqn:t_bounds}), one has 
\begin{equation} \label{eqn:delay_sym}
t_{j+1} - t_{j} = t_{j+3} - t_{j+2} = d
\end{equation} 
\noindent
Making the appropriate substitutions in (\ref{eqn:twoway2}) gives
\begin{align*}
K_{\tilde{o}} & = \tilde{o} + \frac{1}{2} \big ( a_i d  -  a_i d \big ) = \tilde{o}
\end{align*}
\noindent 
which is then applied to the clock state of Node $k$ at times $t_{j+5}$. To demonstrate how this solves the synchronization problem, consider the error between the clocks of nodes $i$ and $k$ at $t_{j+5}$,
\begin{equation*} 
\begin{aligned}
\tau_{i} (t_{j+5}) - \tau_{k}(t_{j+5}) & = \tau_{i}(t_{j+5}) - (\tau_{k}(t_{j+5}) - K_{\tilde{o}}) \\
& = \big ( a_i t_{j+5} {+} \tau_{i}(0) \big ) \\
& \hspace{5mm} -  \big ( a_k t_{j+5} {+} \tau_{k}(0)  {\minus} (\tau_i(0) {\minus} \tau_k(0)) \big ) \\
& = a_i t_{j+5} - a_k t_{j+5} \\
& = 0
\end{aligned}
\end{equation*} 
Here $\tau_{k}(t_{j+5})$ is replaced by $\tau_{k}(t_{j+5}) - K_{\tilde{o}}$ where $K_{\tilde{o}}$ is defined in (\ref{eqn:twoway2}). Thus, the clocks at nodes $i$ and $k$ synchronize for the case where the clock drifts are already assumed to be synchronized.

\section{Motivation for An Adaptive Clock Synchronization Algorithm} \label{sec:motivation}

Now, consider the following system data $a_i = 1$, $a_k = 0.8$ with $c = d = 0.5$ and the given sender-receiver algorithm with only the offset correction $K_{\tilde{o}}$ being applied. After, simulating the algorithm, Figure \ref{fig:ex1} shows the plots of the behavior in the error of clocks and the clock rates. As depicted in the figure, the algorithm continually applies the offset correction but due to the mismatch in the clock rates, the error in the clocks fails to converge to zero. This is further evidenced analytically when noting that a mismatch in the clock rates in equation (\ref{eqn:twoway2}) yields an error on the offset $\tilde{o}$ in  (\ref{eqn:offset_law1}).

\begin{figure}
\centering
\includegraphics[width=0.4\textwidth]{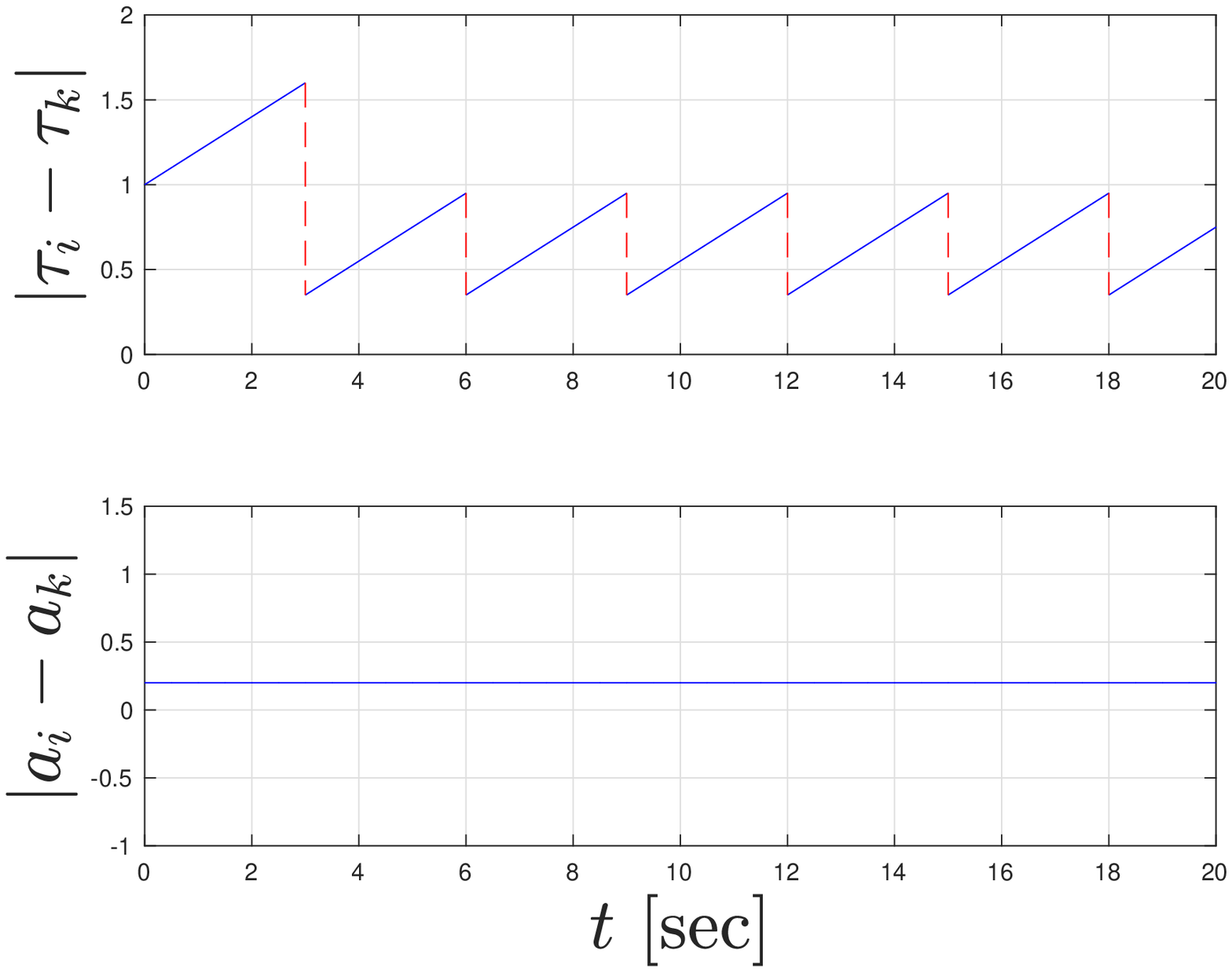}
\caption{\label{fig:ex1} The evolution of the error in the clocks and clock rates of Nodes $i$ and $k$ when the algorithm only applies the offset correction $K_{\tilde{o}}$.}
\end{figure}

Though various strategies exist to mitigate the effects of the error from the mismatched clock rates, the choice of strategy is often left to the system designer, see \cite{eidson2006measurement}. Moreover, these methods are often complicated to implement and too expensive for low-cost applications such as sensor networks. In fact, protocols such as TPSN, designed specifically for low-cost sensor networks, do not provide provisions to correct for the clock rate error, see \cite{ganeriwal2003timing}. Finally, the authors are not aware of any proposed sender-receiver algorithm that provides convergence guarantees for both offset and clock rate correction.

}{

\section{Motivation for An Adaptive Clock Synchronization Algorithm} \label{sec:motivation}

\subsection{Preliminaries on the Sender-Receiver Algorithm} \label{sec:prelim_sendRec}

\begin{figure}
\centering
\begin{adjustbox}{max width=0.4\textwidth}
\begin{tikzpicture}[auto, node distance=0.5cm,>=latex']
    \node (node_i) [circle, draw, align=left, fill=red!20]{Node i};
	\node (cloud) [double arrow, draw, align=center, below of = node_i, node distance=2cm, minimum width=1.5cm, minimum height=10cm]{Network};	
	\node (node_k) [circle, draw, align=left, fill=blue!20, below of=cloud, node distance=2cm]{Node k};    
    \node (node_one) [circle, draw, align=left, fill=blue!20, left of=node_k, node distance=3cm]{Node 1};
    \node (node_n) [circle, draw, align=left, fill=blue!20, right of=node_k, node distance=3cm]{Node n};
    
    \node [output, above of=node_one,node distance=1.6cm] (output_one) {};
    \node [output, above of=node_n,node distance=1.6cm] (output_n) {};
    \path (node_one) -- node[auto=false]{\Huge \ldots} (node_k);
    \path (node_k) -- node[auto=false]{\Huge \ldots} (node_n);
    \draw [ultra thick,<->] (node_i) to (cloud);
    \draw [ultra thick,<->] (node_one) to (output_one);
    \draw [ultra thick,<->] (node_n) to (output_n);
    \draw [ultra thick,<->] (node_k) to (cloud);
\end{tikzpicture}
\end{adjustbox}
\caption{General architecture of the system under consideration.}
\label{fig:sys_arch}
\end{figure}

In a network of $n$ nodes, consider nodes $i$ and $k$ in a sender-receiver hierarchy where Node $i$ is a designated reference or parent agent of a synchronizing child agent Node $k$, see Figure \ref{fig:sys_arch}. Each node has an attached internal clock $\tau_{i}, \tau_{k} \in \reals$ whose dynamics are given by
\begin{equation} \label{eqn:clocks_dyn}
\begin{aligned}
& \dot{\tau}_{i} = a_{i} \\
& \dot{\tau}_{k} = a_{k}
\end{aligned}
\end{equation}
\noindent
where $a_{i}$,$a_{k} \in \reals$ denote the respective clock rates.\footnote{ In this paper, we use the term clock rate to explicitly denote the slope of the given linear affine model of a clock. Other terms for this notion include clock drift or clock skew.}  At times $t_j$ for $j \in \mathbb{N}$ (with $t_0 = 0$), nodes $i$ and $k$ exchange timing measurements with embedded timestamps 
\begin{equation} \label{eqn:timestamps}
\begin{aligned}
T_j^i := \tau_i(t_j) \\
T_j^k := \tau_k(t_j)
\end{aligned}
\end{equation}
\noindent 
which, integrating (\ref{eqn:clocks_dyn}), are equal to $$\tau_i(t_j) = a_i t_j + \tau_i(0)$$ $$\tau_k(t_j) = a_k t_j + \tau_k(0)$$ respectively. Furthermore, $\tau_i(0)$ and $\tau_k(0)$ represent the clock offset from the initial reference time $t = 0$. The goal is to then synchronize the internal clock of Node $k$ to that of Node $i$ using the exchanged timing measurements given in (\ref{eqn:timestamps}).

\ifbool{conf}{
For a sequence of time instants $\{t_j\}_{j=1}^{\infty}$ that is assumed to be strictly increasing and unbounded, at each $t_j$ the standard sender-receiver synchronization algorithm as described in the literature (see \cite{wu2010clock}, \cite{freris2009model}, and \cite{eidson2006measurement}) is given as follows:

\begin{enumerate}[label={(P\arabic*})]
\item \label{itm:one} At time $t_j$, Node $i$ broadcasts a synchronization message with its time $T_j^i$ to Node $k$.
\item \label{itm:two} At time $t_{j+1}$, Node $k$ receives the synchronization message and records its time of arrival $T_{j+1}^k$
\item \label{itm:three} At time $t_{j+2}$, Node $k$ sends a response message with timestamp $T_{j+2}^k$
\item \label{itm:four} At time $t_{j+3}$, Node $i$ receives the response message from Node $k$ and records its time of arrival $T_{j+3}^i$
\item \label{itm:five} At time $t_{j+4}$, Node $i$ sends a response receipt message with timestamp $T_{j+4}^i$
\item \label{itm:six} At time $t_{j+5}$, Node $k$ receives the response message from Node $i$ and records its time of arrival $T_{j+5}^k$ and then updates its clock to synchronize with the clock of Node $i$ using the collected timestamps $T_{j}^i$, $T_{j+1}^k$, $T_{j+2}^k$, $T_{j+3}^i$, and $T_{j+4}^i$. 
\end{enumerate}
\noindent
Moreover, as done in the literature, it is assumed that the time elapsed between each time instant is given by 
\begin{equation} \label{eqn:t_bounds}
t_{j+1} - t_j = \begin{cases}
d \hspace{5mm} \forall j \in \{ 2i + 1 : i \in \mathbb{N} \}, j>0 \\
c \hspace{5mm} \forall j \in \{ 2i : i \in \mathbb{N} \}, j>0 
\end{cases}
\end{equation}
\noindent
where $0 < c \leq d$. The constant $c$ defines the residence or response time delay while $d$ defines the propagation delay of the message transmission. Figure \ref{fig:pairwise} gives a visual representation of the exchange of timestamps between Nodes $i$ and $k$ against reference time $t$. Note that the propagation delay from Node $i$ to Node $k$ and vice versa is assumed to symmetric. Moreover, it is also assumed that the delay due to residence time is the same across all nodes.\footnote{Most pairwise synchronization protocols such as the Network Time Protocol (NTP), Precision Time Protocol (PTP, IEEE 1588), and the Timing-sync Protocol for Sensor Networks (TPSN) assume that the propagation delay in the message transmission from parent to child and child to parent is symmetric. If the propagation delay between the two nodes is asymmetric it introduces an error to the calculated offset correction that cannot be accounted for, see \cite{freris2010fundamental}.} 
}{
Before introducing the mechanics of the sender-receiver algorithm, we refer the reader to a visual model of the algorithm in Figure \ref{fig:pairwise} as a reference. By assuming the sequence of time instants $\{t_j\}_{j=1}^{\infty}$ is strictly increasing and unbounded, the sender-receiver synchronization algorithm as described in the literature (see \cite{wu2010clock}, \cite{freris2009model}, and \cite{eidson2006measurement}) is given as follows:

\begin{enumerate}[label={(P\arabic*})]
\item \label{itm:one} At time $t_j$, Node $i$ broadcasts a synchronization message with its local time $$T_j^i = a_i t_j + \tau_i(0)$$ to Node $k$.
\item \label{itm:two} At time $t_{j+1}$, Node $k$ receives the synchronization message and records its local time of arrival, $T_{j+1}^k$, given in local time at  $$T_{j+1}^k = a_k t_{j+1} + \tau_k(0)$$
\item \label{itm:three} At time $t_{j+2}$, Node $k$ sends a response message with timestamp $$T_{j+2}^k = a_k t_{j+2} + \tau_k(0)$$
\item \label{itm:four} At time $t_{j+3}$, Node $i$ receives the response message from Node $k$ and records its time of arrival $$T_{j+3}^i = a_i t_{j+3} + \tau_i(0)$$
\item \label{itm:five} At time $t_{j+4}$, Node $i$ sends a response receipt message with timestamp $$T_{j+4}^i = a_i t_{j+4} + \tau_i(0)$$
\item \label{itm:six} At time $t_{j+5}$, Node $k$ receives the response message from Node $i$ and records its time of arrival $$T_{j+5}^k = a_k t_{j+5} + \tau_k(0)$$ and then updates its clock to synchronize with the clock of Node $i$ using the collected timestamps $T_{j}^i$, $T_{j+1}^k$, $T_{j+2}^k$, $T_{j+3}^i$, and $T_{j+4}^i$. 
\end{enumerate}
\noindent
Moreover, as done in the literature (see \cite{freris2009model} and \cite{simeone2008distributed}), it is assumed that the time elapsed between each time instant is governed by 
\begin{equation} \label{eqn:t_bounds}
t_{j+1} - t_j = \begin{cases}
d \hspace{5mm} \forall j \in \{ 2i + 1 : i \in \mathbb{N} \}, j>0 \\
c \hspace{5mm} \forall j \in \{ 2i : i \in \mathbb{N} \}, j>0 
\end{cases}
\end{equation}
\noindent
where $0 < c \leq d$. The constant $c$ defines the delay associated with the residence or response time associated with message turnaround while $d$ defines the propagation delay associated with message transmission. Figure \ref{fig:pairwise} gives a visual representation of the exchange of timestamps between Nodes $i$ and $k$ against reference time $t$. Note that the propagation delay from Node $i$ to Node $k$ and vice versa is assumed to symmetric. Moreover, it is also assumed that the delay due to residence time is the same across all nodes.\footnote{Most pairwise synchronization protocols such as the Network Time Protocol (NTP), Precision Time Protocol (PTP, IEEE 1588), and the Timing-sync Protocol for Sensor Networks (TPSN) assume that the propagation delay in the message transmission from parent to child and child to parent is symmetric. If the propagation delay between the two nodes is asymmetric it introduces an error to the calculated offset correction that cannot be accounted for, see \cite{freris2010fundamental}.} }

\begin{figure}
\centering
\includegraphics[width=0.4\textwidth]{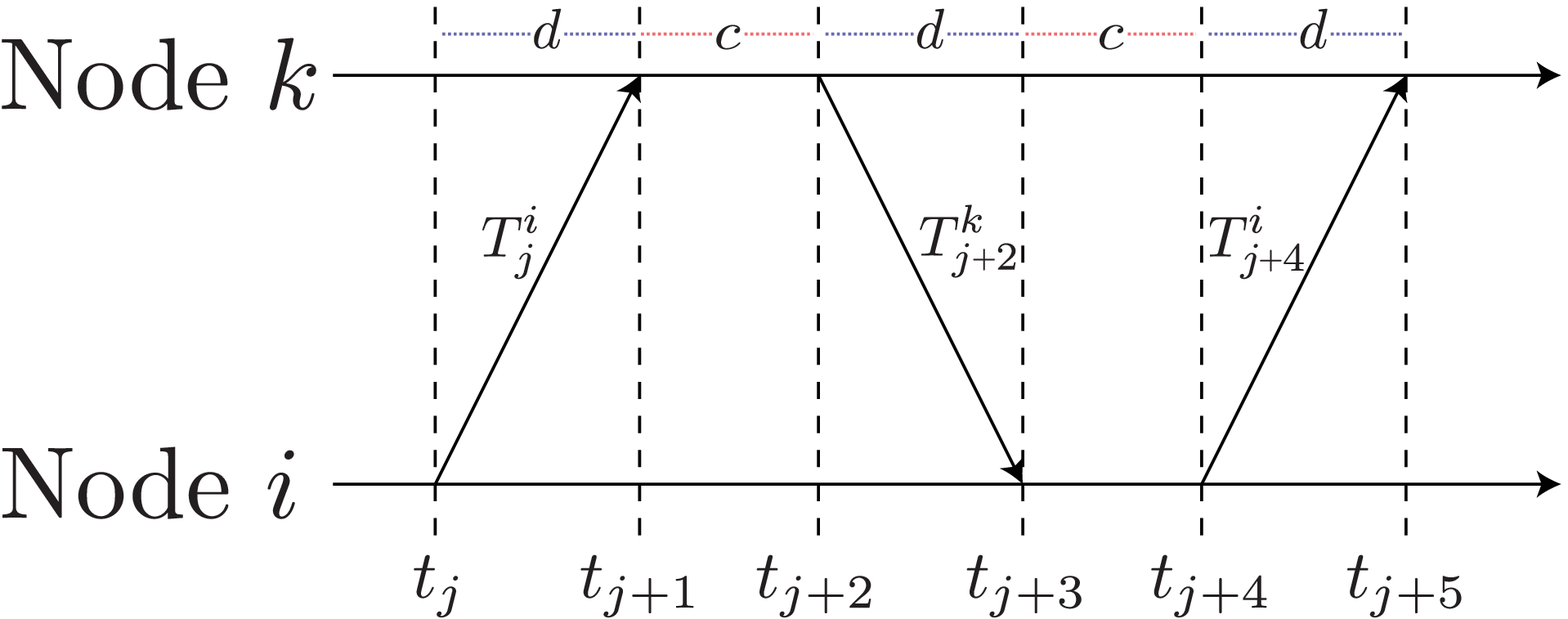}
\caption{\label{fig:pairwise} Diagram illustrating the message exchange between Nodes $i$ and $k$ for the synchronization algorithm.}
\end{figure}

With the available timestamps, at times $t_{j+5}$, we can calculate the relative offset $\tilde{o} := \tau_i(0) - \tau_k(0)$ as follows, by first rearranging the terms in the timestamps given in \ref{itm:one}-\ref{itm:six} one has
\begin{align*}
\tau_i(0) & = T_j^i - a_i t_j \\
\tau_k(0) & = T_{j+1}^k - a_k t_{j+1} \\
\tau_k(0) & = T_{j+2}^k - a_k t_{j+2} \\
\tau_i(0) & = T_{j+3}^i - a_i t_{j+3} \\
\tau_i(0) & = T_{j+4}^i - a_i t_{j+4} \\
\tau_k(0) & = T_{j+5}^k - a_k t_{j+5}
\end{align*}
\noindent
then we have the following expressions for the offset
\begin{align*} 
\tilde{o} & = \tau_i(0) - \tau_k(0) = T_j^i - a_i t_j - T_{j+1}^k + a_k t_{j+1} \\
\tilde{o} & = \tau_i(0) - \tau_k(0) = T_{j+3}^i - a_i t_{j+3} - T_{j+2}^k + a_k t_{j+2} \\
\tilde{o} & = \tau_i(0) - \tau_k(0) = T_{j+4}^i - a_i t_{j+4} - T_{j+5}^k + a_k t_{j+5} 
\end{align*}
rearranging terms one has
\begin{equation} \label{eqn:twoway2} 
\begin{aligned}
T_j^i  - T_{j+1}^k & = a_i t_j - a_k t_{j+1} + \tilde{o} \\
T_{j+3}^i - T_{j+2}^k & = a_i t_{j+3} - a_k t_{j+2} + \tilde{o} \\
T_{j+4}^i - T_{j+5}^k & = a_i t_{j+4} - a_k t_{j+5} + \tilde{o} 
\end{aligned}
\end{equation}
\noindent
Now, if the clock drifts are synchronized, i.e., $a_k = a_i$, we have
\begin{equation} \label{eqn:timestamp_eqn}
\begin{aligned}
T_j^i  - T_{j+1}^k & = a_i (t_j - t_{j+1}) + \tilde{o} \\
T_{j+3}^i - T_{j+2}^k & = a_i (t_{j+3} - t_{j+2}) + \tilde{o} \\
T_{j+4}^i - T_{j+5}^k & = a_i (t_{j+4} - t_{j+5}) + \tilde{o} 
\end{aligned}
\end{equation} 
\noindent
then by noting the bounds on the time elapsed between time instants $t_j$, as given in (\ref{eqn:t_bounds}), one has 
\begin{equation} \label{eqn:delay_sym}
t_{j+1} - t_{j} = t_{j+3} - t_{j+2} =  t_{j+5} - t_{j+4} = d \quad \forall j \in \{ 2i + 1 : i \in \mathbb{N} \}, j>0
\end{equation} 
then by making the appropriate substitutions in (\ref{eqn:timestamp_eqn}) we have
\begin{equation} \label{eqn:timestamp_eqn2}
\begin{aligned}
T_j^i  - T_{j+1}^k & = - a_i d + \tilde{o} \\
T_{j+3}^i - T_{j+2}^k & = a_i d + \tilde{o} \\
T_{j+4}^i - T_{j+5}^k & = - a_i d + \tilde{o} 
\end{aligned}
\end{equation}
Since the clock rates $a_i = a_k$ and the quantity of the propagation delay $d$ are currently unknowns to the system, we are left with a linear system of equations to solve for the offset, i.e., 
\begin{equation} \label{eqn:offset_law1} 
\tilde{o} = \frac{1}{2} \Big ( (T_j^i  - T_{j+1}^k) + (T_{j+3}^i - T_{j+2}^k) \Big )
\end{equation}

To demonstrate how this solves the synchronization problem, consider the error between the clocks of nodes $i$ and $k$ at $t_{j+5}$,
\begin{equation*} 
\begin{aligned}
e_{ik}(t_{j+5}) = \tau_{i} (t_{j+5}) - \tau_{k}(t_{j+5})
\end{aligned}
\end{equation*}
\noindent
at time $t_{j+5}$ node $k$ applies the offset correction $K_{\tilde{o}} = \tilde{o}$ as follows
\begin{equation*} 
\begin{aligned}
e_{ik}(t_{j+5}) & = \tau_{i}(t_{j+5}) - (\tau_{k}(t_{j+5}) - K_{\tilde{o}}) \\
& = \big ( a_i t_{j+5} {+} \tau_{i}(0) \big ) -  \big ( a_k t_{j+5} {+} \tau_{k}(0)  {\minus} (\tau_i(0) {\minus} \tau_k(0)) \big ) \\
& = a_i t_{j+5} - a_k t_{j+5} \\
& = 0
\end{aligned}
\end{equation*} 
\noindent
Thus, the clocks at nodes $i$ and $k$ synchronize for the case where the clock rates $a_i$ and $a_k$ are already assumed to be synchronized.

\subsection{The Key Issue: Clock synchronization in the presence of mismatched clock rates.}

With the mechanics of the sender-receiver algorithm defined, we will now outline the motivation of this paper by demonstrating the issues that arise with the algorithm and how our proposed solution addresses them.

Now, consider the following system data $a_i = 1$, $a_k = 0.8$ with $c = d = 0.5$ and the given sender-receiver algorithm with only the offset correction $K_{\tilde{o}}$ being applied. Simulating the algorithm, Figure \ref{fig:ex1} shows the plots of the behavior in the error of clocks and the clock rates. As depicted in the figure, the algorithm continually applies the offset correction but due to the mismatch in the clock rates, the error in the clocks fails to converge to zero. This is further evidence analytically when noting that if the clock rates are not synchronized in equation (\ref{eqn:twoway2}), the formula for the offset calculation in (\ref{eqn:offset_law1}) will yield an error on the true offset $\tilde{o}$.

\begin{figure}
\centering
\includegraphics[width=0.5\textwidth]{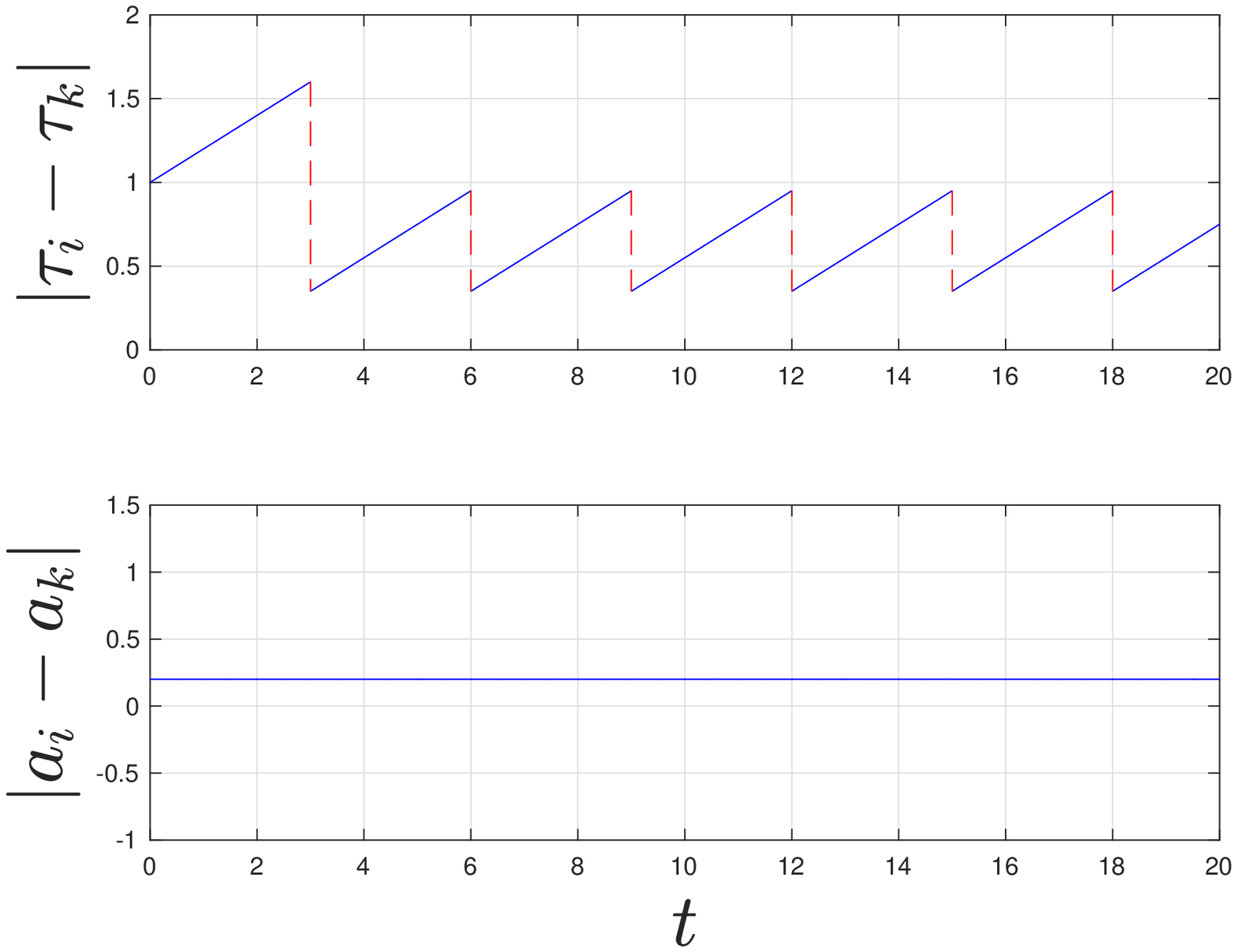}
\caption{\label{fig:ex1} The evolution of the error in the clocks and error in the clock rates of Nodes $i$ and $k$ when the algorithm only applies the offset correction $K_{\tilde{o}}$.}
\end{figure}
 
To mitigate the effects of the error, protocols such as NTP and IEEE 1588 utilize a variety of bespoke methods to minimize the error in clock rates including but not limited to, control of variable frequency hardware oscillators, pulse addition and deletion of the counted pulses at the hardware oscillator, and an error register to track the deviation of the error, see \cite{mills1991internet} and \cite{eidson2006measurement}. These methods, while suitable for industrial-grade equipment, are often expensive solutions for low-cost applications such as sensor networks. In fact, protocols such as TPSN, designed specifically for low-cost sensor networks, do not provide provisions to correct for the clock rate error, see \cite{ganeriwal2003timing}.}


%

\ifbool{conf}{\section{Proposed Algorithm}
}{
\subsection{Problem Formulation and Proposed Algorithm} \label{sec:proposed_algo}

The  problem to solve consists of synchronizing  the internal clock of Node $k$ to that of Node $i$.  More precisely, the goal is to  design a hybrid algorithm  that is based on exchanging timestamps and guarantees   that the clock  variable  $\tau_k$ and the clock rate $a_k$ of Node $k$ are driven to synchronization with $\tau_i$ and $a_i$ of the reference Node $i$, respectively.  Moreover, our goal is to  provide tractable design conditions that  ensure attractivity of  a set of interest.  This problem is formally stated as follows:

\begin{problem} \label{prob:1}
Given two nodes in a sender-receiver hierarchy with clocks having dynamics as in (\ref{eqn:clocks_dyn}) with timestamps $T_j^i$, $T_j^k$ and parameters $c$ and $d$, design a hybrid algorithm such that each trajectory $t \mapsto (\tau_{i}(t), \tau_{k}(t))$ satisfies  the clock synchronization property  $$\lim_{t \to\infty} | \tau_{i}(t) - \tau_{k}(t) | = 0$$ and  the rate synchronization property  $$\lim_{t \to\infty} | \dot{\tau}_{i}(t) - \dot{\tau}_{k}(t) | = 0$$
\end{problem}}

Given the inability of the sender-receiver algorithm to synchronize the clocks, we propose a modification to the algorithm that incorporates an adaptive strategy to synchronize the clock rates. Consider the control law for the synchronization of the clock rate for Node $k$
\begin{equation} \label{eqn:skew_law1}
K_{a} = \mu  (  T_{j+4}^i - T_{j}^i -  T_{j+5}^k -  T_{j+1}^k ) 
\end{equation}
\noindent
with $\mu > 0$ being a controllable parameter. Making the necessary substitutions one has
\begin{equation}
\begin{aligned}
K_{a} & = \mu \Big ( \big (a_i t_{j+4} + \tau_i(0) \big ) {-} \big  (a_i t_{j} + \tau_i(0) \big ) \\ 
& \hspace{1cm} - \big  (a_k t_{j+5} + \tau_k(0)) {-} (a_k t_{j+1} + \tau_k(0) \big ) \Big ) \\ 
& = \mu \big ( a_i (2c + 2d) - a_k (2c + 2d) \big ) \\
& = \mu (2c + 2d) \big ( a_i - a_k \big ) \\
\end{aligned}
\end{equation}
\noindent
The correction $K_{a}$ can then be applied to the clock dynamics of Node $k$ at times $t_{j+5}$ as follows:
\begin{equation}
\begin{aligned}
a_k^+ & = a_k + K_{a} =  a_k + \mu (2c + 2d) \big ( a_i - a_k \big )
\end{aligned}
\end{equation}
\noindent
Observe that this strategy operates under the existing assumptions of the sender-receiver algorithm (symmetric propagation delays and residence times) and does not rely on any additional information that is not already available via the exchanged timing messages. Moreover, since it exploits the integrator dynamics of the system, the computation costs to calculate $K_{a}$ are minimal. In this next example, we demonstrate the proposed strategy under the same scenario of mismatched skews between Nodes $i$ and $k$.

To illustrate, the capabilities of the algorithm outlined above, consider the same system data as in Section \ref{sec:motivation}, namely, $a_i = 1$, $a_k = 0.8$ with $c = d = 0.5$ and the given sender-receiver algorithm now with both the offset correction $K_{\tilde{o}}$ and clock rate correction $K_a$ being applied. In Figure \ref{fig:ex2}, two sets of error plots are presented for two different simulations. Figure \ref{fig:ex2a} gives plots of the errors for the case where the $\mu$ is chosen using information on $c$ and $d$ following our forthcoming design conditions while Figure \ref{fig:ex2b} provide the error plots for the case where $\mu$ is chosen arbitrarily. In the case of the ideal $\mu$, the error in the clocks and clock rate converge to zero whereas in the case of the arbitrarily chosen $\mu$, the error fails to converge. This suggests that a sufficient condition to appropriately design $\mu$ is necessary to ensure convergence of the error.

\begin{figure}
\centering
\subfigure[\label{fig:ex2a}]{\includegraphics[width=0.4\textwidth]{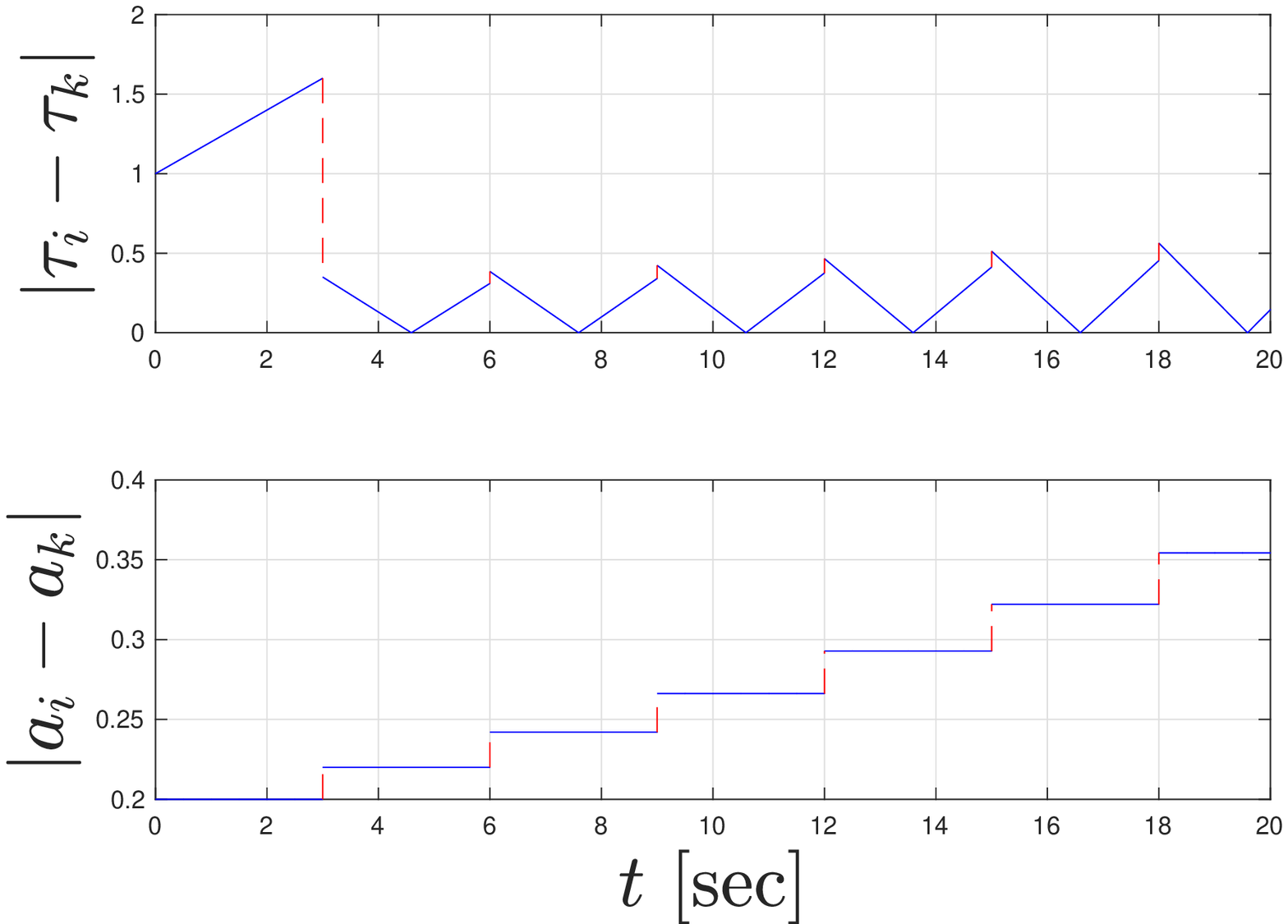}}
\subfigure[\label{fig:ex2b}]{\includegraphics[width=0.4\textwidth]{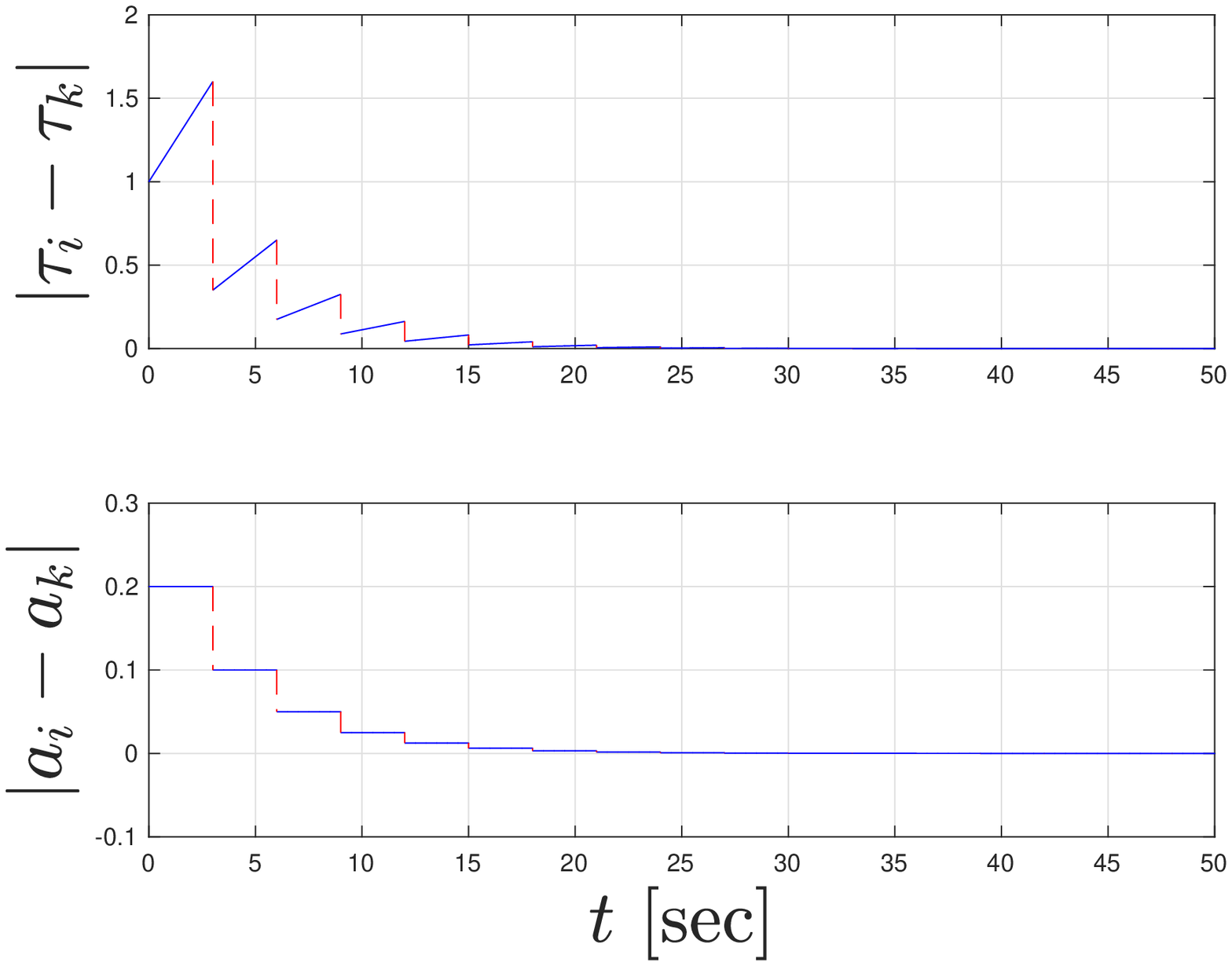}}
\caption{\label{fig:ex2} The evolution of the error in the clocks and clock rates of Nodes $i$ and $k$ when the algorithm applies both offset correction $K_{\tilde{o}}$ and clock rate correction $K_a$. Plot (a) demonstrates the case when $\mu$ is chosen arbitrarily while plot (b) depicts the scenario where $\mu$ is optimally chosen.}
\end{figure}

\ifbool{conf}{}{\section{Preliminaries on Hybrid Systems} \label{sec:prelim}

A hybrid system $\cal H$ in $\mathbb{R}^n$ is composed by the following \textit{data}: a set $C \subset \mathbb{R}^n$, called the flow set; a set-valued mapping $F: \mathbb{R}^n \rightrightarrows \mathbb{R}^n$ with $C \subset \mbox{dom } F$, called the flow map; a set $D \subset \mathbb{R}^n$, called the jump set; a set-valued mapping $G: \mathbb{R}^n \rightrightarrows \mathbb{R}^n$ with $D \subset \mbox{dom } G$, called the jump map. Then, a hybrid system $\mathcal{H} := (C,F,D,G)$ is written in the compact form
\begin{equation} \label{eqn:Hy}
\cal H \begin{cases} \dot{x} \hspace{2mm} \in F(x) \hspace{1cm} & x \in C \\
x^+ \in G(x) \hspace{1cm} & x \in D
\end{cases}
\end{equation}
\normalsize
\noindent
where $x$ is the system state. Solutions to hybrid systems are parameterized by $(t,j)$, where $t \in \mathbb{R}_{\geq 0}$ defines ordinary time and $j \in \mathbb{N}$ is a counter that defines the number of jumps. The evolution of $\phi$ is described by a \textit{hybrid arc} on a \textit{hybrid time domain} \cite{4}. A hybrid time domain is given by $\mbox{dom } \phi \subset \mathbb{R}_{\geq 0} \times \mathbb{N}$ if, for each $(T,J) \in \mbox{dom } \phi,$ $\mbox{dom } \phi \cap ([0,T] \times \{0,1,...,J\})$ is of the form $\bigcup_{j = 0}^J ([t_j, t_{j+1}] \times \{j\})$, with $0 = t_0 \leq t_1 \leq t_2 \leq t_{J + 1}$. A solution $\phi$ is said to be \textit{maximal} if it cannot be extended by flow or a jump, and \textit{complete} if its domain is unbounded. For a hybrid system that is \textit{well-posed}, the closed set $\A \subset \mathbb{R}^n$ is said to be: \textit{attractive} for $\cal H$ if there exists $\mu > 0$ such that every solution $\phi$ to $\cal H$ with $|\phi(0,0)|_{\A} \leq \mu$ is complete and satisfies $\lim_{t + j \rightarrow \infty} |\phi(t,j)|_{\A} = 0$.}

\section{A Hybrid Algorithm for Sender-Receiver Clock Synchronization} \label{sec:model}

\ifbool{conf}{In this section, we present our hybrid model that unites the characterization of the network dynamics for the message exchange with our proposed algorithm that ensures synchronization of the clocks. In addition, we present results and simulations to validate our model.}{}

\ifbool{conf}{
\subsection{Preliminaries on Hybrid Systems}
\label{sec:prelim}

A hybrid system $\cal H$ in $\mathbb{R}^n$ is composed by the following \textit{data}: a set $C \subset \mathbb{R}^n$, called the flow set; a set-valued mapping $F: \mathbb{R}^n \rightrightarrows \mathbb{R}^n$ with $C \subset \mbox{dom } F$, called the flow map; a set $D \subset \mathbb{R}^n$, called the jump set; a set-valued mapping $G: \mathbb{R}^n \rightrightarrows \mathbb{R}^n$ with $D \subset \mbox{dom } G$, called the jump map. Then, a hybrid system $\mathcal{H} := (C,F,D,G)$ is written in the compact form
\begin{equation} \label{eqn:Hy}
\cal H \begin{cases} \dot{x} \hspace{2mm} \in F(x) \hspace{1cm} & x \in C \\
x^+ \in G(x) \hspace{1cm} & x \in D
\end{cases}
\end{equation}
\normalsize
\noindent
where $x$ is the system state. Solutions to hybrid systems are parameterized by $(t,j)$, where $t \in \mathbb{R}_{\geq 0}$ defines ordinary time and $j \in \mathbb{N}$ is a counter that defines the number of jumps. The evolution of a solution is described by a \textit{hybrid arc} $\phi$ on a \textit{hybrid time domain} \cite{4}. A hybrid time domain is given by $\mbox{dom } \phi \subset \mathbb{R}_{\geq 0} \times \mathbb{N}$ if, for each $(T,J) \in \mbox{dom } \phi,$ $\mbox{dom } \phi \cap ([0,T] \times \{0,1,...,J\})$ is of the form $\bigcup_{j = 0}^J ([t_j, t_{j+1}] \times \{j\})$, with $0 = t_0 \leq t_1 \leq t_2 \leq t_{J + 1}$. A solution $\phi$ is said to be \textit{maximal} if it cannot be extended by flow or a jump, and \textit{complete} if its domain is unbounded. The set of all maximal solutions to a hybrid system $\HS$ is denoted by $\mathcal{S}_{\HS}$ and the set of all maximal solutions to $\HS$ with initial condition belonging to a set $A$ is denoted by $\mathcal{S}_{\HS}(A)$. A hybrid system is \textit{well-posed} if it satisfies the hybrid basic conditions in \cite[Assumption 6.5]{4}. 
Let $\A \subset \mathbb{R}^n$ be a closed set and $|x|_{\A} := \mbox{inf}_{y \in \A} |x - y|$. A closed set $\A \subset \mathbb{R}^n$ is said to be: \textit{attractive} for $\cal H$ if there exists $\mu > 0$ such that every solution $\phi$ to $\cal H$ with $|\phi(0,0)|_{\A} \leq \mu$ is complete and satisfies $\lim_{t + j \rightarrow \infty} |\phi(t,j)|_{\A} = 0$.
}{}

\ifbool{conf}{}{In this section we present our hybrid model that  captures  the network dynamics for the message exchange  and  our proposed algorithm that ensures synchronization of the clocks. Using the sender-receiver mechanism for exchanging the timing messages, our algorithm combines the offset correction law in (\ref{eqn:offset_law1}) with the proposed online, adaptive clock rate correction law given in (\ref{eqn:skew_law1}).}

\ifbool{conf}{\subsection{Problem Statement}

Our goal is to synchronize the internal clock of Node $k$ to that of Node $i$. In particular, our goal is to design a hybrid algorithm incorporating the sender-receiver algorithm such that the clock $\tau_k$ and clock rate $a_k$ of Node $k$ is driven to synchronization with $\tau_i$ and $a_i$ of Node $i$, respectively. Moreover, our objective is to provide tractable conditions that ensures attractivity. This problem is formally stated as follows:

\begin{problem} \label{prob:1}
Given two nodes in a sender-receiver hierarchy with clocks having dynamics as in (\ref{eqn:clocks_dyn}) with timestamps $T_j^i$, $T_j^k$ and parameters $c$, $d$, design a hybrid algorithm such that each trajectory $t \mapsto (\tau_{i}(t), \tau_{k}(t))$ satisfies $\lim_{t \to\infty} | \tau_{i}(t) - \tau_{k}(t) | = 0$ and $\lim_{t \to\infty} | \dot{\tau}_{i}(t) - \dot{\tau}_{k}(t) | = 0$
\end{problem}

}{}

\subsection{Modeling}

Given the mix of continuous and discrete dynamics of the system, i.e., the continuous evolution of the clocks and the discrete events of the computation and network transmission, a hybrid modeling approach is a natural fit to  perform the needed  analysis and design goals to solve Problem \ref{prob:1}. Thus, with our problem defined formally, we present a hybrid model that captures the proposed algorithm given in Section \ref{sec:proposed_algo}.  To model the hardware and communication dynamics of the system, namely, the residence and transit times elapsed between the timing messages, we consider a global timer $\tau \in [0,d]$ with dynamics 
\begin{equation} \label{eqn:timer_dyn} 
\begin{aligned}
& \dot{\tau} = -1 & & \tau \in [0,d] \\
& \tau^+ \in \{c, d \}  & & \tau = 0 \\
\end{aligned}
\end{equation}
\noindent
 In this model, the  timer $\tau$ is  reset to either $c$ or $d$ when $\tau = 0$ in order to preserve the bounds given in (\ref{eqn:t_bounds}). We remind the reader that the constant $c$ denotes the residence delay and $d$ denotes the transmission or propagation delay. To determine the appropriate choice for the new value of $\tau$, namely, $\tau^+$, we define a discrete variable $q \in \{0, 1\} =: \mathcal{Q}$ to indicate the residence or transmission state of the system, namely, whether the system is servicing a message at one of the two nodes or whether the system is waiting for the arrival of a message at either of the two nodes, respectively.  
\ifbool{conf}{The state vectors 
\begin{align*}
\m^i & = [\m^i_1, \m^i_2, \m^i_3, \m^i_4, \m^i_5, \m^i_6]^{\top} \in \reals^6 \\
\m^k & = [\m^k_1, \m^k_2, \m^k_3, \m^k_4, \m^k_5, \m^k_6]^{\top} \in \reals^6
\end{align*}
represent memory buffers to store the received and transmitted timestamps  respectively, for Node $i$ and Node $k$.   In addition, a second discrete variable $p \in \{0,1,2,3,4,5\} =: \mathcal{P}$ is used to track the state of the protocol corresponding to the events defined in \ref{itm:one}-\ref{itm:six} of the sender-receiver algorithm. Then, by incorporating the clocks $\tau_i$, $\tau_k$ and the clock rates $a_i$, $a_k$ as state variables to the model, the state $x$ of the complete hybrid system is defined as
\begin{equation*}
x := (\tau_{i}, \tau_{k}, a_{i}, a_{k}, \tau, \m^i, \m^k, p, q) \in \mathcal{X}
\end{equation*}
\noindent
where
\begin{equation*}
\mathcal{X} := \reals \times \reals \times \reals \times \reals \times [0,d] \times \mathbb{R}^6 \times \mathbb{R}^6 \times \mathcal{P} \times \mathcal{Q}
\end{equation*}
\noindent
Then by noting the dynamics of the clocks as given in (\ref{eqn:clocks_dyn}) and those of the timer $\tau$ above, the continuous dynamics of $x$ is given by the following flow map
\begin{equation*}
\begin{aligned}
f(x) = (a_{i},a_{k},0, 0, -1,0,0,0,0) & & \forall x \in C := \mathcal{X}
\end{aligned}
\end{equation*}
}{
 The state  vectors $$
\m^i = [\m^i_1, \m^i_2, \m^i_3, \m^i_4, \m^i_5, \m^i_6]^{\top} \in \reals^6$$ and $$\m^k = [\m^k_1, \m^k_2, \m^k_3, \m^k_4, \m^k_5, \m^k_6]^{\top} \in \reals^6$$ represent memory buffers to store the received and transmitted timestamps  respectively, for Node $i$ and Node $k$.  In addition, a second discrete variable $p \in \{0,1,2,3,4,5\} =: \mathcal{P}$ is used to track  at which stage of the message exchange, defined in \ref{itm:one}-\ref{itm:six}, the algorithm is at.  Then, by incorporating the clocks $\tau_i$, $\tau_k$ and the clock rates $a_i$, $a_k$ as state variables to the model as in Section \ref{sec:prelim_sendRec},  the state $x$ of the hybrid system model, denoted $\HS$, is given by 
\begin{equation*}
x := (\tau_{i}, \tau_{k}, a_{i}, a_{k}, \tau, \m^i, \m^k, p, q) \in \mathcal{X}
\end{equation*}
where
\begin{equation*}
\mathcal{X} := \reals \times \reals \times \reals \times \reals \times [0,d] \times \mathbb{R}^6 \times \mathbb{R}^6 \times \mathcal{P} \times \mathcal{Q}
\end{equation*}
With the dynamics of the clocks as given in (\ref{eqn:clocks_dyn}) and those of the timer $\tau$ in (\ref{eqn:timer_dyn}), the flow map is defined as
\begin{equation} \label{eqn:Hy_f}
\begin{aligned}
F(x) := (a_{i},a_{k},0, 0, -1,0,0,0,0) & & \forall x \in C
\end{aligned}
\end{equation}}
 the flow set $C$ is defined as  
\begin{equation}
C := C_1 \cup C_2
\end{equation}
where $$C_1 := \{ x \in \mathcal{X} : q = 0, \tau \in [0,c] \}$$ and $$C_2 := \{ x \in \mathcal{X} : q = 1, \tau \in [0,d] \}$$  To model the communication and arrival events of the message exchange and the proposed mechanisms correcting the clock rate and offset, we define the jump map $G : \mathbb{R}^n \rightarrow \mathbb{R}^n$ as  
\begin{equation} \label{eqn:Hy_g}
G(x) := G_i(x) \hspace{2mm} \mbox{if } x \in D_{i}
\end{equation}
\noindent
\ifbool{conf}{Each map in $G$ corresponds to the message exchange events \ref{itm:one}-\ref{itm:six}, i.e., $G_{1}$ is equivalent to the event \ref{itm:one}, $G_{2}$ is equivalent to the event \ref{itm:two}, etc.}{ where each mapping  $G_i$ used to define $G$  corresponds to the message exchange events \ref{itm:one}-\ref{itm:six} as follows
\begin{itemize} 
\item $G_{1}$: Node $i$ broadcasts a synchronization message to Node $k$ timestamped with $\tau_{i}$ as in \ref{itm:one}. This event is triggered by the jump set $D_1$, namely, when the timer $\tau = 0$ and the discrete variable $p$ describing the protocol state is zero.  At this event, the timer $\tau$ is reset to $d$, to initiate the message transmission delay. Similarly, the state $q$ is reset to $1$ to indicate the message transmission state of the system with $p$ augmented by one to trigger the next protocol state. Finally, $m_1^i$ is set to $\tau_i$ to record the time of message broadcast, relative to the clock of Node $i$. The subsequent memory states $m_2^i, \ldots m_6^i$ are reset to $m_1^i, \ldots m_5^i$, respectively.
\item $G_{2}$: Node $k$ receives the synchronization message and timestamps its arrival with $\tau_{k}$ as in \ref{itm:two}. This event is triggered by the jump set $D_2$, namely, when the timer $\tau = 0$ and the discrete variable $p$ describing the protocol state is one. At this event, the timer $\tau$ is reset to $c$, to initiate the residence delay. Similarly, the state $q$ is reset to $0$ to indicate the residence state of the system. Finally, $m_1^k$ is set to $\tau_k$ to record the time of message broadcast, relative to the clock of Node $i$. The subsequent memory states $m_2^k, \ldots m_6^k$ are reset to $m_1^i, \ldots m_5^i$, respectively. 
\item $G_{3}$: Node $k$ broadcasts a response message timestamped with $\tau_{k}$ as in \ref{itm:three}. This event is triggered by the jump set $D_3$, namely, when the timer $\tau = 0$ and the discrete variable $p$ describing the protocol state is two. At this event,  the timer $\tau$ is reset to $d$, to initiate the message transmission delay. Similarly, the state $q$ is reset to $1$ to indicate the message transmission state of the system. Finally, $m_1^i$ is set to $\tau_i$ to record the time of message broadcast, relative to the clock of Node $i$. The subsequent memory states $m_2^i, \ldots m_6^i$ are reset to $m_1^i, \ldots m_5^i$, respectively.
\item $G_{4}$: Node $i$ receives the response message and timestamps its arrival with $\tau_{k}$ as in \ref{itm:four}. This event is triggered by the jump set $D_4$, namely, when the timer $\tau = 0$ and the discrete variable $p$ describing the protocol state is three. At this event, the timer $\tau$ is reset to $c$, to initiate the residence delay. Similarly, the state $q$ is reset to $0$ to indicate the residence state of the system. Finally, $m_1^k$ is set to $\tau_k$ to record the time of message broadcast, relative to the clock of Node $i$. The subsequent memory states $m_2^k, \ldots m_6^k$ are reset to $m_1^i, \ldots m_5^i$, respectively.
\item $G_{5}$: Node $i$ broadcasts a response receipt message timestamped with $\tau_{i}$ as in \ref{itm:five}. This event is triggered by the jump set $D_5$, namely, when the timer $\tau = 0$ and the discrete variable $p$ describing the protocol state is four. At this event,  the timer $\tau$ is reset to $d$, to initiate the message transmission delay. Similarly, the state $q$ is reset to $1$ to indicate the message transmission state of the system. Finally, $m_1^i$ is set to $\tau_i$ to record the time of message broadcast, relative to the clock of Node $i$. The subsequent memory states $m_2^i, \ldots m_6^i$ are reset to $m_1^i, \ldots m_5^i$, respectively.
\item $G_{6}$: Node $k$ uses the timestamped messages to update its clock rate and offset via $K_{\tilde{o}}(x)$ in (\ref{eqn:offset_law1}) and $K_{a}(x)$ in (\ref{eqn:offset_law2}), respectively as in \ref{itm:six}. This event is triggered by the jump set $D_6$, namely, when the timer $\tau = 0$ and the discrete variable $p$ describing the protocol state is five. At this event, the timer $\tau$ is reset to $c$, to initiate the residence delay. Similarly, the state $q$ is reset to $0$ to indicate the residence state of the system. Finally, $m_1^k$ is set to $\tau_k$ to record the time of message broadcast, relative to the clock of Node $i$. The subsequent memory states $m_2^k, \ldots m_6^k$ are reset to $m_1^i, \ldots m_5^i$, respectively.
\end{itemize}}  More precisely, the maps  $G_1, G_2, G_3, G_4, G_5, G_6$, updating $x = (\tau_{i}, \tau_{k}, a_{i}, a_{k}, \tau, \m^i, \m^k, p, q)$,  are defined by\footnote{Note that $[x^{\top}, y^{\top}]^\top = (x,y)$.} 
\ifbool{conf}{
\small
\begin{equation*}
\begin{aligned}
& \hspace{-0.5mm}
G_{1}(x) \mequals \begin{bmatrix}
\renewcommand\arraystretch{0.7}
\begin{bmatrix}
\tau_{i} & \tau_{k}
\end{bmatrix}^{\mtop} \\
\begin{bmatrix}
a_{i} & a_{k}
\end{bmatrix}^{\mtop} \\
d \\
\begin{bmatrix}
\begin{bmatrix}
\tau_{i} & \m_1^i & ... & \m_5^i
\end{bmatrix} & \m^k 
\end{bmatrix}^{\mtop} \\
p + 1 \\ 
1
\end{bmatrix}
\hspace{-0.5mm}
G_{2}(x) \mequals \begin{bmatrix}
\renewcommand\arraystretch{0.7}
\begin{bmatrix}
\tau_{i} & \tau_{k}
\end{bmatrix}^{\mtop} \\
\begin{bmatrix}
a_{i} & a_{k}
\end{bmatrix}^{\mtop} \\
c \\
\begin{bmatrix}
\m^i & \begin{bmatrix}
\tau_{k} & \m_1^i & ... & \m_5^i
\end{bmatrix} 
\end{bmatrix}^{\mtop} \\
p + 1 \\ 
0
\end{bmatrix} \\
& \hspace{-0.5mm} G_{3}(x) \mequals \begin{bmatrix}
\renewcommand\arraystretch{0.7}
\begin{bmatrix}
\tau_{i} & \tau_{k}
\end{bmatrix}^{\mtop} \\
\begin{bmatrix}
a_{i} & a_{k}
\end{bmatrix}^{\mtop} \\
d \\
\begin{bmatrix}
\m^i & \begin{bmatrix}
\tau_{k} & \m_1^k & ... & \m_5^k
\end{bmatrix} 
\end{bmatrix}^{\mtop} \\
p + 1 \\ 
1
\end{bmatrix}
\hspace{-0.5mm}
G_{4}(x) \mequals \begin{bmatrix}
\renewcommand\arraystretch{0.7}
\begin{bmatrix}
\tau_{i} & \tau_{k}
\end{bmatrix}^{\mtop} \\
\begin{bmatrix}
a_{i} & a_{k}
\end{bmatrix}^{\mtop} \\
c \\
\begin{bmatrix}
\begin{bmatrix}
\tau_{i} & \m_1^k & ... & \m_5^k
\end{bmatrix} \m^k 
\end{bmatrix}^{\mtop} \\
p + 1 \\ 
0
\end{bmatrix} \\
& \hspace{-0.5mm} G_{5}(x) \mequals \begin{bmatrix}
\renewcommand\arraystretch{0.7}
\begin{bmatrix}
\tau_{i} & \tau_{k}
\end{bmatrix}^{\mtop} \\
\begin{bmatrix}
a_{i} & a_{k}
\end{bmatrix}^{\mtop} \\
d \\
\begin{bmatrix}
\begin{bmatrix}
\tau_{i} & \m_1^i & ... & \m_5^i
\end{bmatrix} & \m^k 
\end{bmatrix}^{\mtop} \\
p + 1 \\ 
1
\end{bmatrix}
\hspace{-0.5mm}
G_{6}(x) \mequals \begin{bmatrix}
\renewcommand\arraystretch{0.7}
\begin{bmatrix}
\tau_{i} & \tau_{k} - K_{\tilde{o}}(x) 
\end{bmatrix}^{\mtop} \\
\begin{bmatrix}
a_{i} & a_{k} + K_{a}(x) 
\end{bmatrix}^{\mtop} \\
c \\
\begin{bmatrix}
\m^i & \begin{bmatrix}
\tau_{k} & \m_1^i & {...} & \m_5^i
\end{bmatrix}
\end{bmatrix}^{\mtop} \\
0 \\ 
0
\end{bmatrix}
\end{aligned}
\end{equation*}}{
\begin{equation} \label{eqn:G_i's}
\begin{aligned}
&
G_{1}(x) := \begin{bmatrix}
\renewcommand\arraystretch{0.7}
(\tau_{i}, \tau_{k}) \\
(a_{i}, a_{k}) \\
d \\
(\tau_{i} , \m_1^i , \ldots , \m_5^i), \m^k) \\
p + 1 \\ 
1
\end{bmatrix},
& &
G_{2}(x) := \begin{bmatrix}
\renewcommand\arraystretch{0.7}
(\tau_{i}, \tau_{k}) \\
(a_{i}, a_{k}) \\
c \\
(\m^i , \tau_{k} , \m_1^i , \ldots , \m_5^i ) \\
p + 1 \\ 
0
\end{bmatrix} \\
&
G_{3}(x) := \begin{bmatrix}
\renewcommand\arraystretch{0.7}
(\tau_{i}, \tau_{k}) \\
(a_{i}, a_{k}) \\
d \\
( \m^i , \tau_{k} , \m_1^k , \ldots , \m_5^k ) \\
p + 1 \\ 
1
\end{bmatrix},
& &
G_{4}(x) := \begin{bmatrix}
\renewcommand\arraystretch{0.7}
(\tau_{i}, \tau_{k}) \\
(a_{i}, a_{k}) \\
c \\
( \tau_{i} , \m_1^k , \ldots , \m_5^k, \m^k ) \\
p + 1 \\ 
0
\end{bmatrix} \\
& G_{5}(x) := \begin{bmatrix}
\renewcommand\arraystretch{0.7}
(\tau_{i}, \tau_{k}) \\
(a_{i}, a_{k}) \\
d \\
( \tau_{i} , \m_1^i , \ldots , \m_5^i, \m^k) \\
p + 1 \\ 
1
\end{bmatrix},
& &
G_{6}(x) := \begin{bmatrix}
\renewcommand\arraystretch{0.7}
(\tau_{i}, \tau_{k} - K_{\tilde{o}}(\m^i) ) \\
(a_{i}, a_{k} + K_{a}(\m^i, \tau_k) ) \\
c \\
( \m^i , \tau_{k} , \m_1^i , \ldots , \m_5^i ) \\
0 \\ 
0
\end{bmatrix}
\end{aligned}
\end{equation}}
\normalsize
\noindent
with 
\begin{equation} \label{eqn:offset_law2}
K_{\tilde{o}}(\m^i) = \frac{1}{2} (\m_4^i- \m_5^i - \m_2^i + \m_3^i)
\end{equation}
\noindent
and
\begin{equation} \label{eqn:skew_law2}
K_{a}(\m^i, \tau_k) = \mu \big ( (\m_1^i - \m_5^i ) - (\tau_{k} - \m_4^i) \big )
\end{equation} 
\noindent
with $\mu > 0$. The offset correction  implemented by the feedback law  $K_{\tilde{o}}$ in (\ref{eqn:offset_law2}) is an adapted version of the offset correction algorithm given in (\ref{eqn:offset_law1}) suitable for the hybrid system model  where the memory states $\m^i$ and $\m^k$ contain the stored timestamps $T_j^i$ and $T_j^k$, respectively. Note that the feedback laws $K_{\tilde{o}}$ and $K_{a}$  depend on the correct assignment of the timestamps to the memory states. In the forthcoming Lemmas \ref{lem:fwd_inv_M} and \ref{lem:finite_M}, we show finite time attractivity of a set containing the correct assignment of the memory states for the appropriate feedback. 
\noindent
 To trigger the jumps corresponding to the particular protocol  events \ref{itm:one}-\ref{itm:six},  we define the  jump set as  $$D := D_{1} \cup D_{2} \cup D_{3} \cup D_{4} \cup D_{5} \cup D_{6}$$ where
\ifbool{conf}{
\begin{equation*}
\begin{aligned}
D_{1} & := \{x \in \mathcal{X} : \tau {=} 0, p {=} 0 \}, \hspace{1.5mm}
D_{2} := \{x \in \mathcal{X} : \tau {=} 0, p {=} 1 \} \\
D_{3} & := \{x \in \mathcal{X} : \tau {=} 0, p {=} 2 \}, \hspace{1.5mm}
D_{4} := \{x \in \mathcal{X} : \tau {=} 0, p {=} 3 \} \\
D_{5} & := \{x \in \mathcal{X} : \tau {=} 0, p {=} 4 \}, \hspace{1.5mm}
D_{6} := \{x \in \mathcal{X} : \tau {=} 0, p {=} 5 \}
\end{aligned}
\end{equation*}}{
\begin{equation*}
\begin{aligned}
D_{1} & := \{x \in \mathcal{X} : \tau = 0, p = 0 \}, \hspace{1cm}
D_{2} & := \{x \in \mathcal{X} : \tau = 0, p = 1 \} \\
D_{3} & := \{x \in \mathcal{X} : \tau = 0, p = 2 \}, \hspace{1cm}
D_{4} & := \{x \in \mathcal{X} : \tau = 0, p = 3 \} \\
D_{5} & := \{x \in \mathcal{X} : \tau = 0, p = 4 \}, \hspace{1cm}
D_{6} & := \{x \in \mathcal{X} : \tau = 0, p = 5 \}
\end{aligned}
\end{equation*}}
\noindent
With the data defined, we let $\HS = (C,F,D,G)$ denote the hybrid system for the pairwise broadcast synchronization algorithm between Node $i$ and Node $k$.

\subsection{Error Model}

To show that the proposed algorithm solves Problem \ref{prob:1}, we recast the problem as a set stabilization problem. Namely, we show that solutions $\phi$ to $\HS$,  with data $(C,F,D,G)$ given in (\ref{eqn:Hy}),  converge to a set of interest wherein the clock states $\tau_i$, $\tau_k$ and clock rates $a_i$,$a_k$, respectively, coincide. To this end, we consider an augmented model of $\HS$ in error coordinates to capture such a property. Let $\varepsilon := (\varepsilon_{\tau}, \varepsilon_a) \in \reals^2$, where  $\varepsilon_{\tau} := \tau_{i} - \tau_{k}$  defines the clock error and $\varepsilon_a := a_{i} - a_{k}$ defines the clock rate error of Nodes $i$ and $k$.  Then, define
\begin{equation*}
x_{\varepsilon} := (\varepsilon, x) \in \mathcal{X}_{\varepsilon} := \reals^2 \times \mathcal{X}
\end{equation*}
\noindent
 which is the state\footnote{ The full state vector $x$ to $\HS$ is retained to facilitate the implementation of the synchronization algorithm for $\HS_{\varepsilon}$.} that collects the clock errors, clock rate errors, and the state of the system $\HS$. The continuous evolution of $x_{\varepsilon}$ is governed by
\begin{equation} \label{eqn:HyEps_fx}
\begin{aligned}
\dot{x}_{\varepsilon} = F_{\varepsilon}(x_{\varepsilon}) := \big ( A_f \varepsilon, F(x) \big ) & & x_{\varepsilon} \in C_{\varepsilon}
\end{aligned}
\end{equation}
\noindent
where $A_f = \begin{bmatrix}
0 & 1 \\ 0 & 0
\end{bmatrix}$ and $f$ is defined in (\ref{eqn:Hy_f}).  The flow set $C_{\varepsilon}$ is defined as
\begin{equation}
C_{\varepsilon} := C_{\varepsilon_1} \cup C_{\varepsilon_1}
\end{equation}
where $$C_{\varepsilon_1} := \{ x_{\varepsilon} \in \mathcal{X}_{\varepsilon} : q = 0, \tau \in [0,c] \}$$ and $$C_{\varepsilon_2} := \{ x_{\varepsilon} \in \mathcal{X}_{\varepsilon} : q = 1, \tau \in [0,d] \}$$ The discrete changes of $x_{\varepsilon}$ are determined by the discrete changes of $\varepsilon$ and $x$, the latter of which is given in (\ref{eqn:Hy_g}).  Through the computation of $\varepsilon^+ = (\varepsilon_{\tau}^+, \varepsilon_{a}^+)$ using the jump maps in (\ref{eqn:G_i's}),  the resulting evolution is modeled by the jump map $G_{\varepsilon} : \mathcal{X}_{\varepsilon} \rightarrow \mathcal{X}_{\varepsilon}$ given by

\begin{equation} \label{eqn:Hy2_G}
G_{\varepsilon} (x_{\varepsilon}) := G_{\varepsilon_i} (x_{\varepsilon}) \quad \mbox{if } x_{\varepsilon} \in D_{\varepsilon_i}
\end{equation}
\noindent
where
\ifbool{conf}{
\begin{equation*}
\begin{aligned}
& G_{\varepsilon_1}(x_{\varepsilon}) {=} \begin{bmatrix}
\renewcommand\arraystretch{0.7}
\varepsilon \\
G_1(x)
\end{bmatrix},
\hspace{1mm}
G_{\varepsilon_2}(x_{\varepsilon}) {=} \begin{bmatrix}
\renewcommand\arraystretch{0.7}
\varepsilon \\
G_2(x)
\end{bmatrix} \\
& G_{\varepsilon_3}(x_{\varepsilon}) {=} \begin{bmatrix}
\renewcommand\arraystretch{0.7}
\varepsilon \\
G_3(x)
\end{bmatrix}, 
\hspace{1mm} 
G_{\varepsilon_4}(x_{\varepsilon}) {=} \begin{bmatrix}
\renewcommand\arraystretch{0.7}
\varepsilon \\
G_4(x)
\end{bmatrix} \\
& G_{\varepsilon_5}(x_{\varepsilon}) {=} \begin{bmatrix}
\renewcommand\arraystretch{0.7}
\varepsilon \\
G_5(x)
\end{bmatrix},
\hspace{1mm}
G_{\varepsilon_6}(x_{\varepsilon}) {=} \begin{bmatrix}
\renewcommand\arraystretch{0.7}
\varepsilon + \begin{bmatrix}
K_{\tilde{o}}(x) \\ \minus K_{a}(x)
\end{bmatrix} \\
G_6(x)
\end{bmatrix}
\end{aligned}
\end{equation*}}
{\begin{equation*}
\begin{aligned}
& G_{\varepsilon_1}(x_{\varepsilon}) {:=} \begin{bmatrix}
\renewcommand\arraystretch{0.7}
\varepsilon \\
G_1(x)
\end{bmatrix},
\hspace{1mm}
G_{\varepsilon_2}(x_{\varepsilon}) {:=} \begin{bmatrix}
\renewcommand\arraystretch{0.7}
\varepsilon \\
G_2(x)
\end{bmatrix},
\hspace{1mm}
G_{\varepsilon_3}(x_{\varepsilon}) {:=} \begin{bmatrix}
\renewcommand\arraystretch{0.7}
\varepsilon \\
G_3(x)
\end{bmatrix} \\
& G_{\varepsilon_4}(x_{\varepsilon}) {:=} \begin{bmatrix}
\renewcommand\arraystretch{0.7}
\varepsilon \\
G_4(x)
\end{bmatrix},
\hspace{1mm}
G_{\varepsilon_5}(x_{\varepsilon}) {:=} \begin{bmatrix}
\renewcommand\arraystretch{0.7}
\varepsilon \\
G_5(x)
\end{bmatrix},
\hspace{1mm}
G_{\varepsilon_6}(x_{\varepsilon}) {:=} \begin{bmatrix}
\renewcommand\arraystretch{0.7}
\varepsilon + \begin{bmatrix}
K_{\tilde{o}}(\m^i) \\ \minus K_{a}(\m^i, \tau_k)
\end{bmatrix} \\
G_6(x)
\end{bmatrix}
\end{aligned}
\end{equation*}}
\normalsize
\noindent
Observe that the  feedback laws $K_{\tilde{o}}$ and $K_{a}$ are employed when $\varepsilon$ is updated by $G_{\varepsilon_6}$, similarly to when $G_6$ is employed $\HS$. These discrete dynamics apply when $x$ is in $D_{\varepsilon} := D_{\varepsilon_1} \cup D_{\varepsilon_2} \cup D_{\varepsilon_3} \cup D_{\varepsilon_4} \cup D_{\varepsilon_5} \cup D_{\varepsilon_6}$, where 
\ifbool{conf}{
\begin{equation*}
\begin{aligned}
D_{\varepsilon_1} & {:=} \{x_{\varepsilon} \in \mathcal{X}_{\varepsilon} : \tau {=} 0, p {=} 0 \}, \hspace{1.5mm}
D_{\varepsilon_2} {:=} \{x_{\varepsilon} \in \mathcal{X}_{\varepsilon} : \tau {=} 0, p {=} 1 \} \\
D_{\varepsilon_3} & {:=} \{x_{\varepsilon} \in \mathcal{X}_{\varepsilon} : \tau {=} 0, p {=} 2 \}, \hspace{1.5mm}
D_{\varepsilon_4} {:=} \{x_{\varepsilon} \in \mathcal{X}_{\varepsilon} : \tau {=} 0, p {=} 3 \} \\
D_{\varepsilon_5} & {:=} \{x_{\varepsilon} \in \mathcal{X}_{\varepsilon} : \tau {=} 0, p {=} 4 \}, \hspace{1.5mm}
D_{\varepsilon_6} {:=} \{x_{\varepsilon} \in \mathcal{X}_{\varepsilon} : \tau {=} 0, p {=} 5 \} \\
\end{aligned}
\end{equation*}
}
{\begin{equation*}
\begin{aligned}
D_{\varepsilon_1} & := \{x_{\varepsilon} \in \mathcal{X}_{\varepsilon} : \tau = 0, p = 0 \}, \hspace{1cm}
D_{\varepsilon_2} & := \{x_{\varepsilon} \in \mathcal{X}_{\varepsilon} : \tau = 0, p = 1 \} \\
D_{\varepsilon_3} & := \{x_{\varepsilon} \in \mathcal{X}_{\varepsilon} : \tau = 0, p = 2 \}, \hspace{1cm}
D_{\varepsilon_4} & := \{x_{\varepsilon} \in \mathcal{X}_{\varepsilon} : \tau = 0, p = 3 \} \\
D_{\varepsilon_5} & := \{x_{\varepsilon} \in \mathcal{X}_{\varepsilon} : \tau = 0, p = 4 \}, \hspace{1cm}
D_{\varepsilon_6} & := \{x_{\varepsilon} \in \mathcal{X}_{\varepsilon} : \tau = 0, p = 5 \} \\
\end{aligned}
\end{equation*}}
\noindent
This hybrid system is denoted 
\begin{equation} \label{eqn:Hy2}
\HS_{\varepsilon} = (C_{\varepsilon}, F_{\varepsilon}, D_{\varepsilon}, G_{\varepsilon})
\end{equation}

The set to render attractive so as to solve Problem \ref{prob:1} is given by
\begin{equation} \label{eqn:set}
\A_{\varepsilon} := \{ x_{\varepsilon} \in \mathcal{X}_{\varepsilon} : \varepsilon = 0 \}
\end{equation}
\noindent
where $\varepsilon = 0$ implies synchronization of both the clock offset and  the clock rate, since, when $\varepsilon_{\tau} = 0$ and $\varepsilon_a = 0$,  then $\tau_k$ is synchronized to $\tau_i$. 

\ifbool{conf}{}{\subsection{Basic Properties of ${\bf \HS_{\varepsilon}}$} \label{sec:prop_of_hs}


Having the hybrid system $\HS_{\varepsilon}$ defined, the next two results establish existence of solutions to $\HS_{\varepsilon}$ and that every maximal solution to $\HS_{\varepsilon}$ is complete. In particular, we show that, through the satisfaction of some basic conditions on the hybrid system data, which is shown first, the system $\HS_{\varepsilon}$ is well-posed and that each maximal solution to the system is defined for arbitrarily large $t + j$.

\begin{lemma} \label{lem:hbcs}
The hybrid system  $\HS_{\varepsilon} = (C_{\varepsilon}, F_{\varepsilon}, D_{\varepsilon}, G_{\varepsilon})$  satisfies the following conditions, defined in \cite[Assumption 6.5]{4} as the hybrid basic  conditions; namely, 
\begin{itemize}
\item[(A1)] $C_{\varepsilon}$ and $D_{\varepsilon}$ are closed subsets of $\reals^m$; 
\item[(A2)] $F_{\varepsilon} : \reals^m \rightarrow \reals^m$ is continuous;
\item[(A3)] $G_{\varepsilon} : \reals^m \rightrightarrows \reals^m$ is outer semicontinuous and locally bounded relative to $D_{\varepsilon}$, and $D_{\varepsilon} \subset \mbox{dom } G_{\varepsilon}$.
\end{itemize}
\end{lemma} 
\ifbool{conf}{}{
\ifbool{conf}{\begin{pf}}{
\begin{proof}}
By inspection of the hybrid system data $(C_{\varepsilon}, F_{\varepsilon}, D_{\varepsilon}, G_{\varepsilon})$ defining $\HS_{\varepsilon}$ given in (\ref{eqn:Hy2}), the following is observed:
\begin{itemize}
\item The set $C_{\varepsilon}$ is a closed subset of $\reals^m$ since  $C_{\varepsilon}$ is the union of the sets $C_{\varepsilon_1}$ and $C_{\varepsilon_2}$, both of which are the Cartesian product of closed sets.  Similar arguments show that $D_{\varepsilon}$ is closed since it can be written as  the finite union of closed sets, that is, 

\begin{align*}
D_{\varepsilon} & = \bigcup_{p \in \mathcal{P}} \big ( \reals^2 \times \reals \times \reals \times \reals \times \reals \times \{0\} \times \mathbb{R}^6 \times \mathbb{R}^6 \times \{p\} \times \mathcal{Q} \big )
\end{align*}

Thus, (A1) holds.
\item  The function $F_{\varepsilon}: \mathcal{X}_{\varepsilon} \rightarrow \mathcal{X}_{\varepsilon}$ is linear affine in the state and thus continuous on $C_{\varepsilon}$. Thus, (A2) holds.
\item  To show that the set-valued map $G_{\varepsilon}$ defined in (\ref{eqn:Hy2_G}) satisfies (A3), 
observe that by inspection,  for each $i \in \{1,2,3,4,5,6\}$ $G_{\varepsilon_i}$ is a continuous map. Moreover, for each $i \in \{1,2,3,4,5,6\}$ $D_{\varepsilon_i}$ is closed and 
$$D_{\varepsilon_i} \cap D_{\varepsilon_k} = \emptyset \hspace{1cm} \forall i,k \in \{1,2,3,4,5,6\}, i \neq k$$
which implies that there is a (uniform) finite separation between these sets. This is due to the fact that these sets are defined for different values of the logic variables.  Hence, (A3) holds  as $G_{\varepsilon}$ is a piecewise function with each piece being continuous. 
\end{itemize}
\ifbool{conf}{\hfill $\square$ \end{pf}}{
\end{proof}}}

\begin{lemma} \label{lem:hy_comp}
For every $\xi \in C_\varepsilon \cup D_\varepsilon ( = \mathcal{X}_{\varepsilon})$, there exists at least one nontrivial solution $\phi$ to $\HS_\varepsilon$ such that $\phi(0,0) = \xi$. Moreover, every maximal solution to $\HS_{\varepsilon}$ is complete.
\end{lemma}
\ifbool{conf}{}{
\ifbool{conf}{\begin{pf}}{
\begin{proof}}
Consider an arbitrary $\xi \in C_{\varepsilon} \cup D_{\varepsilon}$. The tangent cone $T_{C_{\varepsilon}}(\xi)$, as defined in \cite[Definition 5.12]{4}, given by
\begin{equation*}
\begin{aligned}
T_{C_{\varepsilon}}(\xi) {=} \begin{cases} \reals^2 \times \reals \times \reals \times \reals \times \reals \times \reals_{\geq 0} \times \mathbb{R}^6 \times \mathbb{R}^6 \times \mathcal{P} \times \mathcal{Q} & \mbox{if } \xi \in \mathcal{X}^1_{\varepsilon} \\ 
\reals^2 \times \reals \times \reals \times \reals \times \reals \times \reals \times \mathbb{R}^6 \times \mathbb{R}^6 \times \mathcal{P} \times \mathcal{Q} & \mbox{if } \xi \in \mathcal{X}^2_{\varepsilon} \\
\reals^2 \times \reals \times \reals \times \reals \times \reals \times \reals_{\leq 0} \times \mathbb{R}^6 \times \mathbb{R}^6 \times \mathcal{P} \times \mathcal{Q} & \mbox{if } \xi \in \mathcal{X}^3_{\varepsilon}
\end{cases}
\end{aligned}
\end{equation*}
\noindent
where $\mathcal{X}_{\varepsilon}^1 := \{ x_{\varepsilon} \in \mathcal{X}_{\varepsilon} : \tau = 0 \}$, $\mathcal{X}_{\varepsilon}^2 := \{ x_{\varepsilon} \in \mathcal{X}_{\varepsilon} : \tau \in (0, d) \}$, and $\mathcal{X}_{\varepsilon}^3 := \{ x_{\varepsilon} \in \mathcal{X}_{\varepsilon} : \tau = d \}$. By inspection, $F_{\varepsilon}(x_{\varepsilon}) \cap T_{C_{\varepsilon}}(x_{\varepsilon}) \neq \emptyset$ holds  for every $x_{\varepsilon} \in C_{\varepsilon} \setminus D$. Then, by \cite[Proposition 6.10]{4}, there exists a nontrivial solution $\phi$ to $\HS_{\varepsilon}$ with $\phi(0,0) = \xi$. Moreover,  by the same result,  every $\phi \in \mathcal{S}_{\HS_{\varepsilon}}$ satisfies one of the following conditions:
\begin{enumerate}
\setlength{\itemindent}{7mm}
\item[a)] $\phi$ is complete;
\item[b)] $\mbox{dom } \phi$ is bounded and the interval $I^J$, where $J = \mbox{sup}_{j} \mbox{dom } \phi$, has nonempty interior and $t \mapsto \phi(t,J)$ is a maximal solution to $\dot{x} \in F(x)$, in fact $\lim_{t \to T} |\phi(t,J)| = \infty $, where $T = \mbox{sup}_{t} \mbox{dom } \phi$;
\item[c)] $\phi(T,J) \notin C \cup D$, where $(T,J) = \mbox{sup dom } \phi$.
\end{enumerate}
\noindent
Now, since $G_{\varepsilon}(D_{\varepsilon}) \subset C_{\varepsilon} \cup D_{\varepsilon}$ case (c) does not occur. Additionally, one can eliminate case (b) since, by inspection, $F_{\varepsilon}$ is Lipschitz continuous on $C_{\varepsilon}$.
\ifbool{conf}{ \hfill $\square$ \end{pf}}{
\end{proof}}}

The effectiveness of the update laws $K_{\tilde{o}}$ and $K_{a}$, given in (\ref{eqn:offset_law2}) and (\ref{eqn:skew_law2}), in correcting the clock and clock rate of Node $k$, depend on the assigned values of $\m^i$ and $\m^k$ at the time $K_{\tilde{o}}$ and $K_{a}$, i.e., when jumps according to $G_{\varepsilon_6}$ occur.  Improper initialization of the memory states may result in updates of the offset and clock rate of Node $k$ that increase the error in the clocks and clock rates relative to Node $i$.  Therefore, to facilitate the analysis of $\HS_{\varepsilon}$ in rendering the set $\A_{\varepsilon}$ asymptotically attractive, we restrict the values of $\m^i$ and $\m^k$  to a set smaller than $\mathcal{X}$ where they remain  in forward (hybrid) time.  More precisely, we restrict the state $x_{\varepsilon}$ to the set 
\begin{equation} \label{set:M}
\mathcal{M} := \mathcal{M}_1 \cup \mathcal{M}_2 \cup \mathcal{M}_3 \cup \mathcal{M}_4 \cup \mathcal{M}_5 \cup \mathcal{M}_6
\end{equation}
\noindent
where
\ifbool{conf}{
\begin{equation*} 
\begin{aligned}
\mathcal{M}_1 & := \{ x_{\varepsilon} \in \mathcal{X}_{\varepsilon} : p {=} 0, q {=} 0 \} \\
\mathcal{M}_2 & := \{ x_{\varepsilon} \in \mathcal{X}_{\varepsilon} : p {=} 1, q {=} 1, \m_1^i {\minus} \rho_i(x_{\varepsilon}, 0)  = 0 \} \\
\mathcal{M}_3 & := \{ x_{\varepsilon} \in \mathcal{X}_{\varepsilon} : p {=} 2, q {=} 0, \m_1^k {\minus} \rho_k(x_{\varepsilon}, 0)  = 0, \\ 
& \hspace{7.5mm} \m_2^k {\minus} \rho_i(x_{\varepsilon}, d) = 0 \} \\
\mathcal{M}_4 & := \{ x_{\varepsilon} \in \mathcal{X}_{\varepsilon} : p {=} 3, q {=} 1, \m_1^k {\minus} \rho_k(x_{\varepsilon}, 0) = 0, \\
& \hspace{7.5mm} \m_2^k {\minus} \rho_k(x_{\varepsilon}, c) = 0, \m_3^k {\minus} \rho_i(x_{\varepsilon}, c{+}d) = 0 \} \\
\mathcal{M}_5 & := \{ x_{\varepsilon} \in \mathcal{X}_{\varepsilon} : p {=} 4, q {=} 0, \m_1^i {\minus} \rho_i(x_{\varepsilon}, 0) = 0, \\
& \hspace{7.5mm} \m_2^i {\minus} \rho_k(x_{\varepsilon}, d) = 0, \m_3^i {\minus} \rho_k(x_{\varepsilon},  c{+}d) = 0, \\
& \hspace{7.5mm}  \m_4^i {\minus} \rho_i(x_{\varepsilon}, 2d{+}c) {=} 0 \} \\
\mathcal{M}_6 & := \{ x_{\varepsilon} \in \mathcal{X}_{\varepsilon} : p {=} 5, q {=} 1, \m_1^i {\minus} \rho_i(x_{\varepsilon}, 0)  = 0,  \\
& \hspace{7.5mm} \m_2^i {\minus} \rho_i(x_{\varepsilon}, c) = 0, \m_3^i {-} \rho_k(x_{\varepsilon},c{+}d) = 0, \\
& \hspace{7.5mm}  \m_4^i {\minus} \rho_k(x_{\varepsilon},2c{+}d) = 0, \m_5^i {\minus} \rho_i(x_{\varepsilon}, 2c{+}2d) = 0 \}
\end{aligned}
\end{equation*}
}{
\begin{equation*} 
\begin{aligned}
\mathcal{M}_1 & := \{ x_{\varepsilon} \in \mathcal{X}_{\varepsilon} : p {=} 0, q {=} 0 \} \\
\mathcal{M}_2 & := \{ x_{\varepsilon} \in \mathcal{X}_{\varepsilon} : p {=} 1, q {=} 1, \m_1^i {\minus} \rho_i(x_{\varepsilon}, 0)  = 0 \} \\
\mathcal{M}_3 & := \{ x_{\varepsilon} \in \mathcal{X}_{\varepsilon} : p {=} 2, q {=} 0, \m_1^k {\minus} \rho_k(x_{\varepsilon}, 0)  = 0, \\ 
& \hspace{7.5mm} \m_2^k {\minus} \rho_i(x_{\varepsilon}, d) = 0 \} \\
\mathcal{M}_4 & := \{ x_{\varepsilon} \in \mathcal{X}_{\varepsilon} : p {=} 3, q {=} 1, \m_1^k {\minus} \rho_k(x_{\varepsilon}, 0) = 0, \\
& \hspace{7.5mm} \m_2^k {\minus} \rho_k(x_{\varepsilon}, c) = 0, \m_3^k {\minus} \rho_i(x_{\varepsilon}, c{+}d) = 0 \} \\
\mathcal{M}_5 & := \{ x_{\varepsilon} \in \mathcal{X}_{\varepsilon} : p {=} 4, q {=} 0, \m_1^i {\minus} \rho_i(x_{\varepsilon}, 0) = 0, \\
& \hspace{7.5mm} \m_2^i {\minus} \rho_k(x_{\varepsilon}, d) = 0, \m_3^i {\minus} \rho_k(x_{\varepsilon},  c{+}d) = 0, \\
& \hspace{7.5mm}  \m_4^i {\minus} \rho_i(x_{\varepsilon}, 2d{+}c) {=} 0 \} \\
\mathcal{M}_6 & := \{ x_{\varepsilon} \in \mathcal{X}_{\varepsilon} : p {=} 5, q {=} 1, \m_1^i {\minus} \rho_i(x_{\varepsilon}, 0)  = 0,  \\
& \hspace{7.5mm} \m_2^i {\minus} \rho_i(x_{\varepsilon}, c) = 0, \m_3^i {-} \rho_k(x_{\varepsilon},c{+}d) = 0, \\
& \hspace{7.5mm}  \m_4^i {\minus} \rho_k(x_{\varepsilon},2c{+}d) = 0, \m_5^i {\minus} \rho_i(x_{\varepsilon}, 2c{+}2d) = 0 \}
\end{aligned}
\end{equation*}
}
\noindent
and
\begin{equation} 
\begin{aligned}
\rho_i(x_{\varepsilon}, \beta) & :=\tau_{i} - a_{i}( (1-q)c+qd - \tau) - a_{i} \beta \\
\rho_k(x_{\varepsilon}, \beta) & := \tau_{k} - a_{k}( (1-q)c+qd - \tau) - a_{k} \beta
\end{aligned}
\end{equation}
\noindent
for $\beta \geq 0$.

\begin{lemma} \label{lem:fwd_inv_M}
The set $\mathcal{M}$ is forward invariant for the hybrid system $\HS_{\varepsilon}$.
\end{lemma}

\ifbool{conf}{}{
\ifbool{conf}{\begin{pf}}{
\begin{proof}}
Pick an initial condition $\phi(0,0) \in \mathcal{M}$. \begin{itemize}
\item If  $\phi(0,0) \in \mathcal{M} \cap ( C_{\varepsilon} \setminus D_{\varepsilon})$,  then the solution initially flows according to  $\dot{x}_{\varepsilon} = F_{\varepsilon}(x)$.  Observe that the trajectories of $\m^i$, $\m^k$, $p$, and $q$ remain constant since  $F_{\varepsilon}$ is defined so that  $\dot{\m}^i = \dot{\m}^k = \dot{p} = \dot{q} = 0$. Moreover,  note that the gradient of $\rho_i$ and $\rho_k$ with respect to $x_{\varepsilon} = (\varepsilon, \tau_{i}, \tau_{k}, a_{i}, a_{k}, \tau, \m^i, \m^k, p, q)$ satisfy 
\ifbool{conf}{ 
\begin{equation}
\begin{aligned}
& \nabla_{x_\varepsilon} \rho_i(x_{\varepsilon}, \beta) = \begin{bmatrix}
0_{2 \times 1} \\
1 \\
0 \\
\tau - \beta - dq + c(q - 1) \\
0 \\
a_i \\ 
0 \\
0 \\
0 \\
a_i(c - d)
\end{bmatrix} 
\\ & \nabla_{x_\varepsilon} \rho_k(x_{\varepsilon}, \beta) = \begin{bmatrix}
0 \\
0 \\
1 \\
0 \\
\tau - \beta - dq + c(q - 1) \\
a_k \\
0 \\
0 \\
0 \\
a_k(c - d)
\end{bmatrix} 
\end{aligned}
\end{equation}
}{
 \begin{equation} 
\begin{aligned}
\hspace{-10mm} \nabla_{x_\varepsilon} \rho_i(x_{\varepsilon}, \beta) = \begin{bmatrix}
\textbf{0}_{2 \times 1} \\
1 \\
0 \\
\tau - \beta - dq + c(q - 1) \\
0 \\
a_i \\ 
\textbf{0}_{6 \times 1} \\
\textbf{0}_{6 \times 1}  \\
0 \\
a_i(c - d)
\end{bmatrix}, 
& & \nabla_{x_\varepsilon} \rho_k(x_{\varepsilon}, \beta) = \begin{bmatrix}
\textbf{0}_{2 \times 1} \\
0 \\
1 \\
0 \\
\tau - \beta - dq + c(q - 1) \\
a_k \\
\textbf{0}_{6 \times 1} \\
\textbf{0}_{6 \times 1}  \\
0 \\
a_k(c - d)
\end{bmatrix} 
\end{aligned}
\end{equation}
}
\noindent
Then one has $\dot{\rho}_i(x_{\varepsilon}, \beta) = \langle \nabla \rho_i(x_{\varepsilon}, \beta), F_{\varepsilon}(x_{\varepsilon}) \rangle = 1 a_i + a_i (-1) = 0$ and $\dot{\rho}_k(x_{\varepsilon}, \beta) = \langle \nabla \rho_k(x_{\varepsilon}, \beta), F_{\varepsilon}(x_{\varepsilon}) \rangle = 1 a_k + a_k (-1) = 0$. Therefore,  when $\phi$ initially flows from a point in $\mathcal{M}$,  it remains in $\mathcal{M}$  over the interval of flow. This property holds for every solution over any of its intervals of flows  that starts at a point in $\mathcal{M}$.   

\item If  $\phi(0,0) \in \mathcal{M} \cap D_{\varepsilon}$,  then the solution initially jumps according to $x_{\varepsilon}^+ = G_{\varepsilon}(x_{\varepsilon})$. In particular, 
\begin{itemize}

\item if  $\phi(0,0) \in \mathcal{M}_1 \cap D_{\varepsilon_1}$,  the solution jumps according to $x_{\varepsilon}^+ = G_{\varepsilon_1}(x_\varepsilon)$. The timer $\tau$ resets according to  $\tau^+ = d$ while $q^+ = 1$ and $p^+ = 1$. Moreover, $(\m_1^i)^+$ is assigned to the value of $\tau_i$, evaluating $\rho_i(x_{\varepsilon}^+, 0)$, we have that for each $x_{\varepsilon} \in D_{\varepsilon_1}$
\begin{align*}
\rho_i(x_{\varepsilon}^+, 0) & = \tau_{i} - a_{i}( (1-q^+)c+q^+ d - \tau^+) - a_{i} 0 \\
& = \tau_{i} - a_{i}( (1-1)c+ d - d) \\
& = \tau_{i}
\end{align*} 

Thus, by recalling the definition of $\mathcal{M}_2 = \{ x_{\varepsilon} \in \mathcal{X}_{\varepsilon} : p {=} 1, q {=} 1, \m_1^i {\minus} \rho_i(x_{\varepsilon}, 0)  = 0 \}$, we have that $G_{\varepsilon_1}(\mathcal{M}_1 \cap D_{\varepsilon}) \subset \mathcal{M}_2$ holds for each $x_{\varepsilon} \in D_{\varepsilon_1}$.

\item if  $\phi(0,0) \in \mathcal{M}_2 \cap D_{\varepsilon}$,  the solution jumps according to $x_{\varepsilon}^+ = G_{\varepsilon_2}(x_\varepsilon)$. The timer $\tau$ resets according to  $\tau^+ = c$ while $q^+ = 0$ and $p^+ = 2$.  Then, by definition of $G_{\varepsilon_2}$, for each $x_{\varepsilon} \in D_{\varepsilon_2}$, one has
\begin{equation*} 
\begin{aligned} 
\rho_k(x_{\varepsilon}^+, 0) & = \tau_{k} - a_{k}( (1-q^+)c+q^+ d - \tau^+) - a_{k} 0 \\
& = \tau_{k} - a_{k}( (1-0)c - c) \\
& = \tau_{k} \\
\end{aligned}
\end{equation*}
\noindent
which is equal to $(\m_1^k)^+$ and
\begin{equation*} 
\begin{aligned} 
\rho_i(x_{\varepsilon}^+, d) & = \tau_{i} - a_{i}( (1-q^+)c+q^+ d - \tau^+) - a_{i} d \\
& = \tau_{i} - a_{i}( (1-0)c - c) - a_{i} d \\
& = \tau_{i} - a_i d
\end{aligned}
\end{equation*}
\noindent
which is equal to $(\m_2^k)^+ = \m_1^i $.  Therefore, by recalling the definition $\mathcal{M}_3 = \{ x_{\varepsilon} \in \mathcal{X}_{\varepsilon} : p {=} 2, q {=} 0, \m_1^k {\minus} \rho_k(x_{\varepsilon}, 0)  = 0, \m_2^k {\minus} \rho_i(x_{\varepsilon}, d) = 0 \}$, we have $G_{\varepsilon_2}(\mathcal{M}_2 \cap D_{\varepsilon}) \subset \mathcal{M}_3$ for each $x_{\varepsilon} \in D_{\varepsilon_2}$. 

\item if  $\phi(0,0) \in \mathcal{M}_3 \cap D_{\varepsilon}$,  the solution jumps according to $x_{\varepsilon}^+ = G_{\varepsilon_3}(x_\varepsilon)$. The timer $\tau$ resets according to  $\tau^+ = d$ while $q^+ = 1$ and $p^+ = 3$.  Then, by definition of $G_{\varepsilon_3}$, for each $x_{\varepsilon} \in D_{\varepsilon_3}$, one has
\begin{equation*} 
\begin{aligned}
\rho_k(x_{\varepsilon}^+, 0) & = \tau_{k} - a_{k}( (1-q^+)c+qd - \tau^+) - a_{k} 0 \\
& = \tau_{k} - a_{k}( (1-1)c + d - d) \\
& = \tau_{k} 
\end{aligned}
\end{equation*}
which is equal to $(\m_1^k)^+$, 
\begin{equation*} 
\begin{aligned}
\rho_k(x_{\varepsilon}^+, c) & = \tau_{k} - a_{k}( (1-q^+)c+qd - \tau^+) - a_{k} c \\
& = \tau_{k} - a_{k}( (1-1)c + d - d) - a_{k} c \\
& =  \tau_{k} - a_{k} c 
\end{aligned}
\end{equation*}
which is equal to $(\m_2^k)^+ = \m_1^k$, and
\begin{equation*} 
\begin{aligned}
\rho_i(x_{\varepsilon}^+, c+d) & = \tau_{i} - a_{i}( (1-q^+)c+qd - \tau^+) - a_{i} (c + d) \\
& = \tau_{i} - a_{i}( (1-1)c + d - d) - a_{i} (c + d) \\
& = \tau_{i} - a_{i} (c + d)
\end{aligned}
\end{equation*}
\noindent
which is equal to $(\m_3^k)^+ = \m_2^k $. Therefore, by recalling the definition $\mathcal{M}_4 = \{ x_{\varepsilon} \in \mathcal{X}_{\varepsilon} : p {=} 3, q {=} 1, \m_1^k {\minus} \rho_k(x_{\varepsilon}, 0) = 0, \m_2^k {\minus} \rho_k(x_{\varepsilon}, c) = 0, \m_3^k {\minus} \rho_i(x_{\varepsilon}, c{+}d) = 0 \}$, we have $G_{\varepsilon_3}(\mathcal{M}_3 \cap D_{\varepsilon}) \subset \mathcal{M}_4$ for each $x_{\varepsilon} \in D_{\varepsilon_3}$.

\item if  $\phi(0,0) \in \mathcal{M}_4 \cap D_{\varepsilon}$,  the solution jumps according to $x_{\varepsilon}^+ = G_{\varepsilon_4}(x_\varepsilon)$. The timer $\tau$ resets according to  $\tau^+ = c$ while $q^+ = 0$ and $p^+ = 4$.  Then, by definition of $G_{\varepsilon_4}$, for each $x_{\varepsilon} \in D_{\varepsilon_4}$, one has
\begin{equation*} 
\begin{aligned}
\rho_i(x_{\varepsilon}^+, 0) & = \tau_{i} - a_{i}( (1-q^+)c+qd - \tau^+) - a_{i} \beta \\
& = \tau_{i} - a_{i}( (1-0)c  - c) \\
& = \tau_i 
\end{aligned}
\end{equation*}
which is equal to $(\m_1^k)^+$, 
\begin{equation*} 
\begin{aligned}
\rho_k(x_{\varepsilon}^+, d) & = \tau_{k} - a_{k}( (1-q^+)c+qd - \tau^+) - a_{k} d \\
& = \tau_{k} - a_{k}( (1-0)c - c) - a_{k} d \\
& =  \tau_{k} - a_{k} d
\end{aligned}
\end{equation*}
which is equal to $(\m_2^i)^+ = \m_1^k$, 
\begin{equation*} 
\begin{aligned}
\rho_k(x_{\varepsilon}^+, c + d) & = \tau_{k} - a_{k}( (1-q^+)c+qd - \tau^+) - a_{k} (c+d) \\
& = \tau_{k} - a_{k}(  (1-0)c  - c) - a_{k} (c + d) \\
& = \tau_{k} - a_{k} (c + d)
\end{aligned}
\end{equation*}
which is equal to $(\m_3^i)^+ = \m_2^k$, 
\begin{equation*} 
\begin{aligned}
(\m_4^i)^+ & = \m_3^k =   \rho_i(x_{\varepsilon}^+, c+2d)  \\
& = \tau_{i} - a_{i}( (1-q^+)c+qd - \tau^+) - a_{i} (c+2d) \\
& = \tau_{i} - a_{i}(  (1-0)c  - c) - a_{i} (c + 2d) \\
& = \tau_{i} - a_{i} (c + 2d)
\end{aligned}
\end{equation*}
\noindent
which is equal to $(\m_4^i)^+ = \m_3^k$. Therefore, by recalling the definition $\mathcal{M}_5 = \{ x_{\varepsilon} \in \mathcal{X}_{\varepsilon} : p {=} 4, q {=} 0, \m_1^i {\minus} \rho_i(x_{\varepsilon}, 0) = 0, \m_2^i {\minus} \rho_k(x_{\varepsilon}, d) = 0, \m_3^i {\minus} \rho_k(x_{\varepsilon},  c{+}d) = 0, \m_4^i {\minus} \rho_i(x_{\varepsilon}, c{+}2d) {=} 0 \}$, we have $G_{\varepsilon_4}(\mathcal{M}_4 \cap D_{\varepsilon}) \subset \mathcal{M}_5$ for each $x_{\varepsilon} \in D_{\varepsilon_4}$.

\item if  $\phi(0,0) \in \mathcal{M}_5 \cap D_{\varepsilon}$,  the solution jumps according to $x_{\varepsilon}^+ = G_{\varepsilon_5}(x_\varepsilon)$. The timer $\tau$ resets according to  $\tau^+ = d$ while $q^+ = 1$ and $p^+ =5$.  Then, by definition of $G_{\varepsilon_5}$, for each $x_{\varepsilon} \in D_{\varepsilon_5}$, one has
\begin{equation*} 
\begin{aligned}
\rho_i(x_{\varepsilon}^+, 0) & = \tau_{i} - a_{i}( (1-q^+)c+qd - \tau^+) - a_{i} 0 \\
& = \tau_{i} - a_{i}( (1-1)c + d - d) \\
& = \tau_i \\
\end{aligned}
\end{equation*}
which is equal to $(\m_1^k)^+$, 
\begin{equation*} 
\begin{aligned}
\rho_i(x_{\varepsilon}^+, c) & = \tau_{k} - a_{k}( (1-q^+)c+qd - \tau^+) - a_{k} 0 \\
& = \tau_{k} - a_{k}( (1-1)c + d - d) - a_k c \\
& =  \tau_{k} - a_{k} c \\
\end{aligned}
\end{equation*}
which is equal to $(\m_2^i)^+ = \m_1^i$, 
\begin{equation*} 
\begin{aligned}
\rho_k(x_{\varepsilon}^+, c+d) & = \tau_{k} - a_{k}( (1-q^+)c+qd - \tau^+) - a_{k} (c+d) \\
& = \tau_{k} - a_{k}( (1-1)c + d - d) - a_k (c+d) \\
& = \tau_{k} - a_{k} (c + d) \\
\end{aligned}
\end{equation*}
which is equal to $(\m_3^i)^+ = \m_2^i$, 
\begin{equation*} 
\begin{aligned}
\rho_k(x_{\varepsilon}^+, 2c+d) & = \tau_{k} - a_{k}( (1-q^+)c+qd - \tau^+) - a_{k} (2c+d) \\
& = \tau_{k} - a_{k}( (1-1)c + d - d) - a_k (2c+d) \\
& = \tau_{k} - a_{k} (2c + d) \\
\end{aligned}
\end{equation*}
which is equal to $(\m_4^i)^+ = \m_3^i$, 
\begin{equation*} 
\begin{aligned}
\rho_i(x_{\varepsilon}^+, 2c+2d) & = \tau_{i} - a_{i}( (1-q^+)c+qd - \tau^+) - a_{i} (2c+2d) \\
& = \tau_{i} - a_{i}( (1-1)c + d - d) - a_{i} (2c+2d) \\
& = \tau_{i} - a_{i} (2c + 2d)
\end{aligned}
\end{equation*}
\noindent
which is equal to $(\m_5^i)^+ = \m_4^i$. Therefore, by recalling the definition $\mathcal{M}_6 = \{ x_{\varepsilon} \in \mathcal{X}_{\varepsilon} : p {=} 5, q {=} 1, \m_1^i {\minus} \rho_i(x_{\varepsilon}, 0)  = 0, \m_2^i {\minus} \rho_i(x_{\varepsilon}, c) = 0, \m_3^i {-} \rho_k(x_{\varepsilon},c{+}d) = 0, \m_4^i {\minus} \rho_k(x_{\varepsilon},2c{+}d) = 0, \m_5^i {\minus} \rho_i(x_{\varepsilon}, 2c{+}2d) = 0 \}$, we have $G_{\varepsilon_5}(\mathcal{M}_5 \cap D_{\varepsilon}) \subset \mathcal{M}_6$ for each $x_{\varepsilon} \in D_{\varepsilon_5}$.

\item if $\phi(0,0) \in \mathcal{M}_6 \cap D_{\varepsilon}$, the solution jumps according to $x_{\varepsilon}^+ = G_{\varepsilon_6}(x_\varepsilon)$. The timer $\tau$ resets according to  $\tau^+ = c$ while $q^+ = 0$ and $p^+ = 0$. Therefore, by recalling the definition $\mathcal{M}_1  = \{ x_{\varepsilon} \in \mathcal{X}_{\varepsilon} : p {=} 0, q {=} 0 \}$, we have $G_{\varepsilon_5}(\mathcal{M}_6 \cap D_{\varepsilon}) \subset \mathcal{M}_1$  for each $x_{\varepsilon} \in D_{\varepsilon_6}$.
\end{itemize}
\end{itemize}
\ifbool{conf}{\hfill $\square$ \end{pf}}{
\end{proof}}}

\begin{lemma} \label{lem:finite_M}
Let constants $d \geq c > 0$ be given. For each maximal solution $\phi$ to $\HS_{\varepsilon}$, there exists $T^* \geq 0$ such that  $\phi(t,j) \in \mathcal{M}$ for any $(t,j) \in \mbox{dom } \phi$ with $t + j \geq T^*$. 
\end{lemma}

\ifbool{conf}{}{
\ifbool{conf}{\begin{pf}}{
\begin{proof}}
 Pick a solution $\phi \in \mathcal{S}_{\HS_{\varepsilon}}$ with initial condition $\phi(0,0) \in  C_{\varepsilon} \cup D_{\varepsilon}$. Since, the flow map $F_{\varepsilon}$ enforces $\dot{p} = 0$, the $p$ component of $\phi$ remains constant during flows. At jumps, namely, when  $\phi(t,j) \in D_{\varepsilon}$, since for each $\ell \in \{1,2,3,4,5\}$, $G_{\varepsilon_{\ell}}$ enforces that $p^+ = p + 1$, the evolution of $p$ is monotonically increasing in $\{0,1,2,3,4,5\}$ until $p=5$, from where $G_6$ resets $p$ to $0$. In fact, when the solution $\phi$ jumps according to $G_{\varepsilon_{6}}$, we have that $p^+ = 0$ and $q^+ = 0$ resulting in a value for $x_{\varepsilon}$ after the jump that is in $\mathcal{M}_1$. Now, due to the monotonic behavior of $p$ and the completeness of solutions to $\HS_{\varepsilon}$ given by Lemma \ref{lem:hy_comp}, there exists $(t,j) \in \mbox{dom } \phi$ such that $\phi(t,j) = G_{\varepsilon_6} (\phi (t,j) )$. Given such $(t,j)$, let $T^* = t + j$.  Then, given that $G_{\varepsilon_{6}} (\phi(t,j)) \subset \mathcal{M}_1$ and the forward invariance of $\mathcal{M}$ given by Lemma \ref{lem:fwd_inv_M}, we have that $\phi(t,j) \in \mathcal{M}$ for each $(t,j) \in \mbox{dom } \phi$ such that $t + j \geq T^*$.
\ifbool{conf}{\hfill $\square$ \end{pf}}{
\end{proof}}}}

In our main result,  which is presented in the next section,  we show asymptotic attractivity of the synchronization set $\A_{\varepsilon}$ via a Lyapunov analysis on solutions from the initialization set $\mathcal{M}$. 

\ifbool{conf}{
\subsection{Main Results} \label{sec:main}
}{
\section{Main Results} \label{sec:main}
}

\ifbool{conf}{
In this section, we present our main result showing asymptotic attractivity of the synchronization set $\A_{\varepsilon}$ for $\HS_{\varepsilon}$. Consider the following Lyapunov function candidate
}{
In this section, we present our main result showing asymptotic attractivity of the synchronization set $\A_{\varepsilon}$ in (\ref{eqn:set})  for $\HS_{\varepsilon}$. To show this, we present a Lyapunov analysis  along solutions to $\HS$ starting from the set $\mathcal{M}$.  We remind the reader that $\mathcal{M}$ is the set that denotes valid initialization values of the memory state vectors $\m^i$ and $\m^k$ for which the update laws $K_{\tilde{o}}$ and $K_a$  give values to correct the clock rate and offset.  To this end, consider the Lyapunov function candidate}
\ifbool{conf}{
\begin{equation} \label{eqn:lyap_fun}
V(x_{\varepsilon}) = \varepsilon^{\top} e^{A_f^{\top} (\tau + d(5-p))} P e^{A_f (\tau + d(5-p))} \varepsilon
\end{equation}
\noindent
Note that there exist two positive scalars $\alpha_1$ and $\alpha_2$ such that for each $x_{\varepsilon} \in C_{\varepsilon} \cup D_{\varepsilon}$
\begin{align*}
\alpha_1|x_{\varepsilon}|^2_{\A} \leq V(x) \leq \alpha_2|x_{\varepsilon}|^2_{\A}
\end{align*}
}{
\begin{equation} \label{eqn:lyap_fun}
V(x_{\varepsilon}) = \varepsilon^{\top} \exp \big ( A_f^{\top} r(\tau, p, q) \big ) P \exp \big ( A_f r(\tau, p, q)  \big ) \varepsilon
\end{equation}
\noindent
where  $P = P^{\top} \succ 0$, $A_f$ is as given in (\ref{eqn:HyEps_fx}),  $r(\tau,p,q) := \tau  h(q) + d(5 - p)$ and $h(q) := 1 + c^{\minus 1} (1 - q) (d-c)$ are defined for each $x_{\varepsilon} \in C_{\varepsilon} \cup D_{\varepsilon}$.  Note that  there exist two positive scalars, $\alpha_1$ and $\alpha_2$, such that
\begin{equation} \label{eqn:v_bounds}
\begin{aligned}
\alpha_1|x_{\varepsilon}|^2_{\A_{\varepsilon}} \leq V(x_{\varepsilon}) \leq \alpha_2|x_{\varepsilon}|^2_{\A_{\varepsilon}} \quad \forall x_{\varepsilon} \in C_{\varepsilon} \cup D_{\varepsilon}
\end{aligned}
\end{equation}}
\ifbool{conf}{}{
\noindent
The function $V$ satisfies the following infinitesimal properties.


\begin{lemma} \label{lem:V_flows}
Let the hybrid system $\HS_{\varepsilon}$ be given as in (\ref{eqn:Hy2}).  For each point $x_{\varepsilon} \in C_{\varepsilon}$, one has 
\begin{equation} \label{eqn:V_flows}
\begin{aligned}
\langle \nabla V(x_{\varepsilon}), F_{\varepsilon}(x_{\varepsilon}) \rangle \leq \begin{cases} 0 \hspace{5mm} & \mbox{\rm if } x_{\varepsilon} \in C_{\varepsilon_2} \\ 
\frac{\gamma}{\alpha_2} V(x_{\varepsilon}) \hspace{5mm} & \mbox{\rm if } x_{\varepsilon} \in C_{\varepsilon_1}
\end{cases}
\end{aligned}
\end{equation} 
\noindent
where  
\begin{equation} \label{eqn:thm_alpha2}
\alpha_2 =  \underset{\nu \in \mathcal{Q}, \sigma \in \mathcal{P}}{\lambda_{\max}} \Big ( \exp \big ( (\nu  h(\nu) + d(5 - \sigma) ) A_f^{\top} \big ) P \exp \big ( \big (\nu h(\nu) + d(5 - \sigma) \big ) A_f  \big ) \Big )
\end{equation}  
\begin{equation} \label{eqn:thm_gamma} 
\gamma = |\alpha| {\rm max} \Big \{ \frac{ p_{11} \epsilon}{2} , \beta + \frac{p_{11}}{2 \epsilon} \Big \}
\end{equation} 
\noindent
$\alpha = \frac{2(c-d)}{c}$, $\epsilon > 0$, $\beta = p_{11} 6d - p_{12}$, and $p_{11}$ and $p_{12}$ come from $P = \begin{bmatrix}
p_{11} & p_{12} \\ p_{21} & p_{22}
\end{bmatrix} \succ 0$.
\end{lemma}

\ifbool{conf}{\begin{pf}}{
\begin{proof}} 
Before calculating $\langle \nabla V(x_{\varepsilon}), F_{\varepsilon}(x_{\varepsilon}) \rangle$, observe that the full expression of $V$ is given by
\begin{equation}
\begin{aligned}
V(x_{\varepsilon}) & = \begin{bmatrix}
\varepsilon_{\tau} \\ \varepsilon_a
\end{bmatrix}^{\top} \exp \big (A_f^{\top} r(\tau,p,q) \big ) \begin{bmatrix}
p_{11} & p_{12} \\ p_{21} & p_{22}
\end{bmatrix} \exp \big (A_f r(\tau,p,q) \big ) \begin{bmatrix}
\varepsilon_{\tau} \\ \varepsilon_a
\end{bmatrix} \\
& = \begin{bmatrix}
\varepsilon_{\tau} \\ \varepsilon_a
\end{bmatrix}^{\top} \begin{bmatrix}
1 & 0 \\
r(\tau,p,q) & 1
\end{bmatrix} \begin{bmatrix}
p_{11} & p_{12} \\ p_{21} & p_{22}
\end{bmatrix} \begin{bmatrix}
1 & r(\tau,p,q)  \\
0 & 1
\end{bmatrix} \begin{bmatrix}
\varepsilon_{\tau} \\ \varepsilon_a
\end{bmatrix} \\
& = \begin{bmatrix}
\varepsilon_{\tau} + \varepsilon_a r(\tau, p, q) \\ \varepsilon_a
\end{bmatrix}^{\top}  \begin{bmatrix}
p_{11} & p_{12} \\ p_{21} & p_{22}
\end{bmatrix}  \begin{bmatrix}
\varepsilon_{\tau} + \varepsilon_a r(\tau, p, q) \\ \varepsilon_a
\end{bmatrix} \\
& = \begin{bmatrix}
\varepsilon_{\tau} + \varepsilon_a r(\tau, p, q) \\ \varepsilon_a
\end{bmatrix}^{\top}   \begin{bmatrix}
p_{11} \big ( \varepsilon_{\tau} + \varepsilon_a r(\tau, p, q) \big ) + p_{12} \varepsilon_a \\
p_{21} \big ( \varepsilon_{\tau} + \varepsilon_a r(\tau, p, q) \big ) + p_{22} \varepsilon_a 
\end{bmatrix} \\
& = \big ( \varepsilon_{\tau} + \varepsilon_a r(\tau, p, q) \big ) \big ( p_{11} \big ( \varepsilon_{\tau} + \varepsilon_a r(\tau, p, q) \big ) + p_{12} \varepsilon_a \big ) + \varepsilon_a \big ( p_{21} \big ( \varepsilon_{\tau} + \varepsilon_a r(\tau, p, q) \big ) + p_{22} \varepsilon_a \big ) \\
& =  p_{11} \big ( \varepsilon_{\tau} + \varepsilon_a r(\tau, p, q) \big )^2 + p_{12} \varepsilon_a \big ( \varepsilon_{\tau} + \varepsilon_a r(\tau, p, q) \big ) + p_{21} \varepsilon_a \big ( \varepsilon_{\tau} + \varepsilon_a r(\tau, p, q) \big ) + p_{22} \varepsilon_a^2
\end{aligned}
\end{equation}
\noindent
then since $p_{12} = p_{21}$
\begin{equation}
\begin{aligned}
V(x_{\varepsilon}) & =  p_{11} \big ( \varepsilon_{\tau} + \varepsilon_a r(\tau, p, q) \big )^2 + 2 p_{12} \varepsilon_a \big ( \varepsilon_{\tau} + \varepsilon_a r(\tau, p, q) \big ) \big ) + p_{22} \varepsilon_a^2
\end{aligned}
\end{equation}
\noindent
In calculating $\langle \nabla V(x_{\varepsilon}), F_{\varepsilon}(x_{\varepsilon}) \rangle$, one has
\begin{equation} \label{eqn:V_dot_interim}
\begin{aligned}
\langle \nabla V(x_{\varepsilon}), F_{\varepsilon}(x_{\varepsilon}) \rangle &  = \begin{bmatrix}
\nabla_{\varepsilon_{\tau}} V(x_{\varepsilon}) &
\nabla_{\varepsilon_{a}} V(x_{\varepsilon}) &
\nabla_{\tau} V(x_{\varepsilon}) &
\nabla_{p} V(x_{\varepsilon}) &
\nabla_{q} V(x_{\varepsilon})
\end{bmatrix}
\begin{bmatrix}
\varepsilon_a \\
0 \\
-1 \\
0 \\
0 \\
\end{bmatrix} \\
& = \nabla_{\varepsilon_{\tau}} V(x_{\varepsilon}) \varepsilon_a - \nabla_{\tau} V(x_{\varepsilon})
\end{aligned}
\end{equation}
\noindent
\noindent
where
\begin{equation} \label{eqn:del_et_t}
\begin{aligned}
\nabla_{\varepsilon_{\tau}} V(x_{\varepsilon}) & = 2 p_{11} \big ( \varepsilon_{\tau} + \varepsilon_a r(\tau, p, q) \big ) + 2 p_{12} \varepsilon_a \\
\nabla_{\tau} V(x_{\varepsilon}) & = 2 p_{11} \varepsilon_a \nabla_{\tau} r(\tau, p, q) \big ( \varepsilon_{\tau} + \varepsilon_a r(\tau, p, q) \big ) + 2 p_{12} \varepsilon_a^2 \nabla_{\tau} r(\tau, p, q) \\
\end{aligned}
\end{equation}
\noindent
Substituting (\ref{eqn:del_et_t}) into (\ref{eqn:V_dot_interim}), we obtain
\begin{align*}
\langle \nabla V(x_{\varepsilon}), F_{\varepsilon}(x_{\varepsilon}) \rangle 
& = \Big ( 2 p_{11} \big ( \varepsilon_{\tau} + \varepsilon_a r(\tau, p, q) \big ) + 2 p_{12} \varepsilon_a \Big) \varepsilon_a \\ 
& \hspace{1.5cm} - 2 p_{11} \varepsilon_a \nabla_{\tau} r(\tau, p, q) \big ( \varepsilon_{\tau} + \varepsilon_a r(\tau, p, q) \big ) - 2 p_{12} \varepsilon_a^2 \nabla_{\tau} r(\tau, p, q) \\
& = 2 p_{11} \varepsilon_a \big ( \varepsilon_{\tau} + \varepsilon_a r(\tau, p, q) \big ) + 2 p_{12} \varepsilon_a^2  \\ 
& \hspace{1.5cm} - 2 p_{11} \varepsilon_a \nabla_{\tau} r(\tau, p, q) \big ( \varepsilon_{\tau} + \varepsilon_a r(\tau, p, q) \big ) - 2 p_{12} \varepsilon_a^2 \nabla_{\tau} r(\tau, p, q) \\
& = 2 p_{11} \varepsilon_a \big ( \varepsilon_{\tau} + \varepsilon_a r(\tau, p, q) \big ) (1 - \nabla_{\tau} r(\tau, p, q)) + 2 p_{12} \varepsilon_a^2 (1 - \nabla_{\tau} r(\tau, p, q))  \\ 
& = \big ( 2 p_{11} \varepsilon_a \big ( \varepsilon_{\tau} + \varepsilon_a r(\tau, p, q) \big ) + 2 p_{12} \varepsilon_a^2 \big ) \big (1 - \nabla_{\tau} r(\tau, p, q) \big )  \\ 
& = \big ( 2 p_{11} \big ( \varepsilon_a \varepsilon_{\tau} + \varepsilon_a^2 r(\tau, p, q) \big ) + 2 p_{12} \varepsilon_a^2 \big ) \big (1 - \nabla_{\tau} r(\tau, p, q) \big )  \\
\end{align*}
\noindent
 for each $x_{\varepsilon} \in C_{\varepsilon}$.   Now, with $\nabla_{\tau} r(\tau, p, q) = \frac{(c - d)(q - 1)}{c} + 1$,  when $x_{\varepsilon} \in C_{\varepsilon_2}$ with $q=1$, $\nabla_{\tau} r(\tau, p, q) = 1$, thus $$\langle \nabla V(x_{\varepsilon}), F_{\varepsilon}(x_{\varepsilon}) \rangle = 0$$ 
\noindent
When $x_{\varepsilon} \in C_{\varepsilon_2}$ with $q=0$, $r(\tau,p,q) = \tau \big ( \frac{d-c}{c} + 1 \big ) + d(5-p)$ and  $\nabla_{\tau} r(\tau, p, q) = \frac{d-c}{c} + 1$  one then has 
\begin{align*}
\langle \nabla V(x_{\varepsilon}), F_{\varepsilon}(x_{\varepsilon}) \rangle & = \Big ( 2 p_{11} \varepsilon_a \varepsilon_{\tau} + 2 p_{11} \varepsilon_a^2 \Big ( \tau \Big ( \frac{d-c}{c} + 1 \Big ) + d(5-p) \Big ) + 2 p_{12} \varepsilon_a^2 \Big ) \Big ( \frac{c-d}{c} \Big )  \\
& = \Big ( 2 p_{11} \varepsilon_a \varepsilon_{\tau} + 2 p_{11} \varepsilon_a^2 \Big ( \tau \Big ( \frac{d-c}{c} \Big ) + \tau + d(5-p) \Big ) + 2 p_{12} \varepsilon_a^2 \Big ) \Big ( \frac{c-d}{c} \Big ) \\
& = \Big ( \frac{c-d}{c} \Big ) 2 p_{11} \varepsilon_a \varepsilon_{\tau} + \Big ( \frac{c-d}{c} \Big ) 2 p_{11} \varepsilon_a^2 \Big ( \tau \Big ( \frac{d-c}{c} \Big ) + \tau + d(5-p) \Big ) \\
& \hspace{5mm} + \Big ( \frac{c-d}{c} \Big ) 2 p_{12} \varepsilon_a^2  \\
& = \Big ( \frac{2(c-d)}{c} \Big ) p_{11} \varepsilon_a \varepsilon_{\tau} + \Big ( \frac{2(c-d)}{c} \Big ) p_{11} \varepsilon_a^2 \Big ( \tau \Big ( \frac{d-c}{c} \Big ) + \tau + d(5-p) \Big ) \\
& \hspace{5mm} + \Big ( \frac{2(c-d)}{c} \Big ) p_{12} \varepsilon_a^2 
\end{align*}
\noindent
Let $\alpha = \frac{2(c-d)}{c}$, then since, $0 < c \leq d$ we have that $\alpha \leq 0$
\begin{align*}
\langle \nabla V(x_{\varepsilon}), F_{\varepsilon}(x_{\varepsilon}) \rangle & = -|\alpha| p_{11} \varepsilon_a \varepsilon_{\tau} - |\alpha| p_{11} \varepsilon_a^2 \Big ( \tau \Big ( \frac{d-c}{c} \Big ) + \tau + d(5-p) \Big ) - |\alpha| p_{12} \varepsilon_a^2 
\end{align*}
 
\noindent
Then, recognizing that $\tau \in [0,c]$ when $x_{\varepsilon} \in C_{\varepsilon_2}$ then we have that $\tau \leq c$,  which due to the fact that $p_{11} > 0$ and $p \in \mathcal{P} = \{0, 1, 2, 3, 4, 5 \}$ leading to
\begin{align*}
\langle \nabla V(x_{\varepsilon}), F_{\varepsilon}(x_{\varepsilon}) \rangle & \leq - |\alpha| p_{11} \varepsilon_a \varepsilon_{\tau} + |\alpha| p_{11} \Big ( c \Big ( \frac{d-c}{c} \Big ) + c + d(5-p) \Big )  \varepsilon_a^2 - |\alpha| p_{12} \varepsilon_a^2 \\
& \leq - |\alpha| p_{11} \varepsilon_a \varepsilon_{\tau} + |\alpha| p_{11} \Big (  (d-c) + c + d(5-p) \Big ) \varepsilon_a^2 - |\alpha| p_{12} \varepsilon_a^2 \\
& \leq - |\alpha| p_{11} \varepsilon_a \varepsilon_{\tau} + |\alpha| p_{11} d( 6 - p )  \varepsilon_a^2 - |\alpha| p_{12} \varepsilon_a^2 
\end{align*}
\noindent
 for each $x_{\varepsilon} \in C_{\varepsilon_2}$. We can upper bound the quantity $6 - p$ by noting that $p \in \mathcal{P} = \{0, 1, 2, 3, 4, 5 \}$. Thus, we have that $6 - p \leq 6$ for each $p \in \mathcal{P}$, leading to

\begin{align*}
\langle \nabla V(x_{\varepsilon}), F_{\varepsilon}(x_{\varepsilon}) \rangle & \leq - |\alpha| p_{11} \varepsilon_a \varepsilon_{\tau} + |\alpha| \big ( p_{11} 6d  - p_{12} \big ) \varepsilon_a^2 
\end{align*}
Now, with $\beta = p_{11} 6d   - p_{12}$. Then, we obtain 
\begin{equation}
\langle \nabla V(x_{\varepsilon}), F_{\varepsilon}(x_{\varepsilon}) \rangle \leq |\alpha| p_{11} |\eps_a| |\eps_{\tau}| + |\alpha| \beta \eps_a^2 \hspace{1cm}  \forall x_{\varepsilon} \in C_{\varepsilon_2}
\end{equation}
  
\noindent
Then, through an application of Young's inequality one has
\begin{align*}
\langle \nabla V(x_{\varepsilon}), F_{\varepsilon}(x_{\varepsilon}) \rangle
& \leq |\alpha| p_{11} \Big (\frac{1}{2 \epsilon} \eps_a^2 + \frac{\epsilon}{2} \eps_{\tau}^2 \Big ) + |\alpha| \beta \eps_a^2 \\
& \leq \frac{|\alpha| p_{11}}{2 \epsilon} \eps_a^2 + |\alpha| \beta \eps_a^2 + \frac{|\alpha| p_{11} \epsilon}{2}  \eps_{\tau}^2  \\
& \leq \frac{|\alpha| p_{11} \epsilon}{2} \eps_{\tau}^2 + |\alpha| \Big (\beta + \frac{p_{11}}{2 \epsilon}  \Big )\eps_a^2 \\
& \leq \gamma \big ( \eps_{\tau}^2 + \eps_a^2 \big ) \\
& \leq \gamma \varepsilon^{\top} \varepsilon
\end{align*}
 for each $x_{\varepsilon} \in C_{\varepsilon_2}$.  Then, from the definition of $V$ in (\ref{eqn:lyap_fun}) 
\begin{align*}
\langle \nabla V(x_{\varepsilon}), F_{\varepsilon}(x_{\varepsilon}) \rangle & \leq \gamma |x_{\varepsilon} |^2  \\
& \leq \frac{\gamma}{\alpha_2} V(x_{\varepsilon})
\end{align*}
\noindent
 for each $x_{\varepsilon} \in C_{\varepsilon_2}$  where $\epsilon > 0$, $\alpha_2$ and $\gamma$ are positive constants given in (\ref{eqn:thm_alpha2}) and (\ref{eqn:thm_gamma}), respectively.
\ifbool{conf}{\hfill $\square$ \end{pf}}{
\end{proof}}}


\ifbool{conf}{}{
\begin{lemma} \label{lem:V_jumps}
Let the hybrid system $\HS_{\varepsilon}$ in (\ref{eqn:Hy2}) with constants $d \geq c > 0$ be given. If there exist a constant $\mu > 0$ and a positive definite symmetric matrix $P$ such that 
\ifbool{conf}{
\begin{equation} \label{eqn:lyap_cond}
\begin{aligned}
A_g^{\top} e^{6d A_f^{\top} } P e^{6d A_f } A_g  -  P \prec 0
\end{aligned}
\end{equation}}
{\begin{equation} \label{eqn:lyap_cond}
\begin{aligned}
A_g^{\top} \exp \big ( 6d A_f^{\top} \big ) P \exp \big ( 6d A_f \big ) A_g  -  P \prec 0
\end{aligned}
\end{equation}}
\noindent
where $A_{g} = \begin{bmatrix} 0 & \gamma_1 \\ 0 &  1 {\minus} \mu \gamma_2 \end{bmatrix}$  and $A_f$ is as given in (\ref{eqn:HyEps_fx})  with $\gamma_1 = \frac{1}{2}(3c + 4d)$ and $\gamma_2 = 2c + 2d$ then, for each $x_{\varepsilon} \in \mathcal{M} \cap D_{\varepsilon}$, $$V \big (G_{\ell}(x_{\varepsilon}) \big ) - V \big (x_{\varepsilon} \big ) \leq 0$$ for each $\ell \in \{1,2,3,4,5\}$, and \footnote{Observe that $\varepsilon^+ = A_g \varepsilon$ is the matrix representation of the jump map $G_6$ for which $\varepsilon$ is reset to when $x_{\varepsilon} \in \mathcal{M}_6 \cap D_{\varepsilon}$.} $$V \big ( G_{6}(x_{\varepsilon}) \big ) - V \big (x_{\varepsilon} \big ) \leq -\sigma \varepsilon^{\top} \varepsilon$$ where
\begin{equation} \label{eqn:sigma}
\sigma \in \bigg ( 0,  \minus \lambda_{\rm min} \Big ( A_g^{\top} \exp \big ((6d) A_f^{\top} \big ) P \exp \big ((6d) A_f \big ) A_g  -  P \Big ) \bigg )
\end{equation}
\end{lemma}

\ifbool{conf}{\begin{pf}}{
\begin{proof}}
 For every $g \in G_{\varepsilon}(x_{\varepsilon})$, the state $\tau$ is reset to a point in  the set $\{c,d\}$. Moreover, for each $x_{\varepsilon} \in D_{\varepsilon}$, $\tau = 0$. Hence, when $x_{\varepsilon} \in D_{\varepsilon_1} \cap \mathcal{M}_1$, we have that $\tau = 0$, $q = 0$, and $p = 0$, leading to
\begin{align*}
V(G_{\varepsilon_1}(x_{\varepsilon})) - V(x_{\varepsilon}) & = \varepsilon^{\top} \exp \big (A_f^{\top} (d + d(5 \minus 1)) \big ) P \exp \big (A_f (d + d(5 \minus 1)) \big ) \varepsilon \\ 
& \hspace{1cm} - \varepsilon^{\top} \exp \big (A_f^{\top} (0 + d(5-0)) \big ) P \exp \big (A_f (0 + d(5 \minus 0)) \big )  \varepsilon \\
& = \varepsilon^{\top} \exp \big (A_f^{\top} (5d) \big ) P \exp \big (A_f (5d) \big ) \varepsilon \\
& \hspace{5mm} - \varepsilon^{\top} \exp \big (A_f^{\top} (5d) \big ) P \exp \big (A_f (5d) \big ) \varepsilon \\
& = 0
\end{align*}
When $x_{\varepsilon} \in D_{\varepsilon_2} \cap \mathcal{M}_2$, we have that $\tau = 0$, $q = 1$, and $p = 1$, leading to
\begin{align*}
V(G_{\varepsilon_2}(x_{\varepsilon})) - V(x_{\varepsilon}) & = \varepsilon^{\top} \exp \big (A_f^{\top} (c(1 + c^{\minus 1}(d - c)) + d(5 \minus 2)) \big ) P \exp \big (A_f (c(1 \\
& \hspace{5mm} + c^{\minus 1}(d - c)) + d(5 \minus 2)) \big ) \varepsilon \\ 
& \hspace{1cm} - \varepsilon^{\top} \exp \big (A_f^{\top} (0 + d(5 \minus 1)) \big ) P \exp \big (A_f (0 + d(5 \minus 1)) \big )  \varepsilon \\
& = \varepsilon^{\top} \exp \big (A_f^{\top} (d + 3d) \big ) P \exp \big (A_f (d + 3d) \big ) \varepsilon \\
& \hspace{5mm} - \varepsilon^{\top} \exp \big (A_f^{\top} (4d) \big ) P \exp \big (A_f (4d) \big ) \varepsilon \\
& = 0
\end{align*}
When $x_{\varepsilon} \in D_{\varepsilon_3} \cap \mathcal{M}_3$, we have that $\tau = 0$, $q = 0$, and $p = 2$, leading to
\begin{align*}
V(G_{\varepsilon_3}(x_{\varepsilon})) - V(x_{\varepsilon}) & = \varepsilon^{\top} \exp \big (A_f^{\top} (d + d(5 \minus 3)) \big ) P \exp \big (A_f (d + d(5 \minus 3)) \big ) \varepsilon \\ 
& \hspace{1cm} - \varepsilon^{\top} \exp \big (A_f^{\top} (0 + d(5 \minus 2)) \big ) P \exp \big (A_f (0 + d(5 \minus 2)) \big )  \varepsilon \\
& = \varepsilon^{\top} \exp \big (A_f^{\top} (3d) \big ) P \exp \big (A_f (3d) \big ) \varepsilon \\
& \hspace{5mm} - \varepsilon^{\top} \exp \big (A_f^{\top} (3d) \big ) P \exp \big (A_f (3d) \big ) \varepsilon \\
& = 0
\end{align*}
When $x_{\varepsilon} \in D_{\varepsilon_4} \cap \mathcal{M}_4$, we have that $\tau = 0$, $q = 1$, and $p = 3$, leading to
\begin{align*}
V(G_{\varepsilon_4}(x_{\varepsilon})) - V(x_{\varepsilon}) & = \varepsilon^{\top} \exp \big (A_f^{\top} (c(1 + c^{\minus 1}(d - c)) + d(5 \minus 4)) \big ) P \exp \big (A_f (c(1 \\
& \hspace{5mm} + c^{\minus 1}(d - c)) + d(5 \minus 4)) \big ) \varepsilon \\ 
& \hspace{1cm} - \varepsilon^{\top} \exp \big (A_f^{\top} (0 + d(5 \minus 3)) \big ) P \exp \big (A_f (0 + d(5 \minus 3)) \big )  \varepsilon \\
& = \varepsilon^{\top} \exp \big (A_f^{\top} (d + d) \big ) P \exp \big (A_f (d + d) \big ) \varepsilon \\
& \hspace{5mm} - \varepsilon^{\top} \exp \big (A_f^{\top} (2d) \big ) P \exp \big (A_f (2d) \big ) \varepsilon \\
& = 0
\end{align*}
When $x_{\varepsilon} \in D_{\varepsilon_5} \cap \mathcal{M}_5$, we have that $\tau = 0$, $q = 0$, and $p = 4$, leading to
\begin{align*}
V(G_{\varepsilon_5}(x_{\varepsilon})) - V(x_{\varepsilon}) & = \varepsilon^{\top} \exp \big (A_f^{\top} (d + d(5 \minus 5)) \big ) P \exp \big (A_f (d + d(5 \minus 5)) \big ) \varepsilon \\ 
& \hspace{1cm} - \varepsilon^{\top} \exp \big (A_f^{\top} (0 + d(5 \minus 4)) \big ) P \exp \big (A_f (0 + d(5 \minus 4)) \big )  \varepsilon \\
& = \varepsilon^{\top} \exp \big (A_f^{\top} (d) \big ) P \exp \big (A_f (d) \big ) \varepsilon \\
& \hspace{5mm} - \varepsilon^{\top} \exp \big (A_f^{\top} (d) \big ) P \exp \big (A_f (d) \big ) \varepsilon \\
& = 0
\end{align*}
\noindent
When $x_{\varepsilon} \in D_{\varepsilon_6} \cap \mathcal{M}_6$, we have that $\tau = 0$, $q = 1$, and $p = 5$. For resets according to $G_{\varepsilon_6}$, one has 
\ifbool{conf}{
\begin{align*}
V(G_{\varepsilon_6}(x_{\varepsilon})) \minus V(x_{\varepsilon}) & = \\
& \hspace{-2.25cm} {\tiny \mequals} \begin{bmatrix} \varepsilon {\tiny \mplus} \begin{bmatrix}
{\tiny \minus} K_{\tilde{o}}(\m^i) \\ K_{a}(\m^i,\tau_k)
\end{bmatrix} \end{bmatrix}^{\top} e^{A_f^{\top} (c {\tiny \mplus} d(5))} P e^{A_f (c {\tiny \mplus} d(5))} \begin{bmatrix} \varepsilon {\tiny \mplus} \begin{bmatrix} {\tiny \minus} K_{\tilde{o}}(\m^i) \\ K_{a}(\m^i,\tau_k)
\end{bmatrix} \end{bmatrix} \\ 
& - \varepsilon^{\top} e^{A_f^{\top} (0 + d(0))} P e^{A_f^{\top} (0 + d(0))}  \varepsilon \\
& \hspace{-2.25cm} {\tiny \mequals} \begin{bmatrix} \varepsilon {\tiny \mplus} \begin{bmatrix}
{\tiny \minus} K_{\tilde{o}}(\m^i) \\ K_{a}(\m^i,\tau_k)
\end{bmatrix} \end{bmatrix}^{\top} e^{A_f^{\top} (c {\tiny \mplus} 5d)} P e^{A_f (c {\tiny \mplus} 5d) } \begin{bmatrix} \varepsilon {\tiny \mplus} \begin{bmatrix} {\tiny \minus} K_{\tilde{o}}(\m^i) \\ K_{a}(\m^i,\tau_k)
\end{bmatrix} \end{bmatrix} \\ & - \varepsilon^{\top} P  \varepsilon
\end{align*}}{
\begin{align*}
V(G_{\varepsilon_6}(x_{\varepsilon})) - V(x_{\varepsilon}) & = \\
& \hspace{-3cm} \begin{bmatrix} \varepsilon \mplus \begin{bmatrix}
K_{\tilde{o}}(\m^i) \\ \minus K_{a}(\m^i,\tau_k)
\end{bmatrix} \end{bmatrix}^{\top} \exp \big (A_f^{\top} (c(1 \mplus c^{\minus 1}(d \minus c)) \mplus d(5 \minus 0)) \big ) P \exp \big (A_f (c(1 \\
& \hspace{5mm} \mplus c^{\minus 1}(d \minus c)) \mplus d(5 \minus 0)) \big ) \begin{bmatrix} \varepsilon \mplus \begin{bmatrix} K_{\tilde{o}}(\m^i) \\ \minus K_{a}(\m^i,\tau_k)
\end{bmatrix} \end{bmatrix} \\ 
& \hspace{1cm} - \varepsilon^{\top} \exp \big (A_f^{\top} (0 + d(0)) \big ) P \big (A_f^{\top} (0 + d(0)) \big )  \varepsilon \\
& \hspace{-25mm} = \begin{bmatrix} \varepsilon + \begin{bmatrix}
K_{\tilde{o}}(\m^i) \\ \minus K_{a}(\m^i,\tau_k)
\end{bmatrix} \end{bmatrix}^{\top} \exp \big (A_f^{\top} (6d) \big ) P \exp \big ( A_f (6d) \big ) \begin{bmatrix} \varepsilon + \begin{bmatrix} K_{\tilde{o}}(\m^i) \\ \minus K_{a}(\m^i,\tau_k)
\end{bmatrix} \end{bmatrix} \\
& \hspace{5mm} - \varepsilon^{\top} P  \varepsilon
\end{align*}}
\noindent
 Now, with  $x_{\varepsilon} \in D_{\varepsilon_6} \cap \mathcal{M}_6$, which implies that $p=5$, $q=1$, and $\tau = 0$,  one has that for jumps with resets according to $G_{\varepsilon_6}(x_{\varepsilon})$, the  feedback laws  $K_{\tilde{o}}$ and $K_{a}$ applied to $\tau_{k}$ and $a_{k}$, respectively, give 
\begin{align*}
K_{\tilde{o}}(\m^i) & = \frac{1}{2} (\m^i_4- \m^i_5 - \m^i_2 + \m^i_3) \\
& = \frac{1}{2} \bigg ( \Big ( (\tau_{k} - a_{k}(2c + 2d)) - (\tau_{i} - a_{i}(2c + 3d)) \Big ) \\ 
& \hspace{1cm} - \Big ( (\tau_{i} - a_{i}(c + d)) - (\tau_{k} - a_{k}(c + 2d)) \Big ) \bigg ) \\
& = \frac{1}{2} \Big ( 2 (\tau_{k} - \tau_{i} ) + a_{i}(3c + 4d) - a_{k} (3c + 4d) \Big ) \\
& = (\tau_{k} - \tau_{i} ) + \frac{1}{2} (a_{i} - a_{k}) (3c + 4d) \\
& = - \varepsilon_{\tau} + \gamma_1  \varepsilon_a \\
K_{a}(\m^i, \tau_k) & =  \mu \big ( (\m^i_1 - \m^i_5 ) - (\tau_{k} - \m^i_4) \big ) \\
& = \mu \big ( (\tau_{i} - a_{i}(d) - (\tau_{i} - a_{i}(2c + 3d))) \\
& \hspace{1cm} - (\tau_{k} - (\tau_{k} - a_{k}(2c + 2d))) \big ) \\
& = \mu \big ( (a_{i}(2c + 3d) + a_{i}(d) ) - (a_{k}(2c + 2d)) \big ) \\
& = \mu (a_{i} - a_{k})(2c + 2d) \\
& = \mu \gamma_2 \varepsilon_a
\end{align*}
\noindent
where $\gamma_1 = \frac{3c + 4d}{2}$ and $\gamma_2 = 2(c + d)$.   Using the expressions for $K_{\tilde{o}}(\m^i) $ and $K_{a}(\m^i, \tau_k)$, it follows that 
\ifbool{conf}{
\small
\begin{align*}
V(G_{\varepsilon_6}(x_{\varepsilon})) {\minus} V(x_{\varepsilon}) & = \\ 
& \hspace{-25mm} = \begin{bmatrix} \varepsilon {+} \begin{bmatrix} K_{\tilde{o}}(\m^i) \\ \minus K_{a}(\m^i,\tau_k)
\end{bmatrix} \end{bmatrix}^{\top} e^{6d A_f^{\top} } P e^{ 6d A_f} \begin{bmatrix} \varepsilon {+} \begin{bmatrix} K_{\tilde{o}}(\m^i) \\ \minus K_{a}(\m^i,\tau_k)
\end{bmatrix} \end{bmatrix} {\minus} \varepsilon^{\top} P  \varepsilon \\
& \hspace{-25mm} = \begin{bmatrix} \varepsilon_{\tau} {\minus} \varepsilon_{\tau} {+} \gamma_1  \varepsilon_a \\ \varepsilon_a {\minus} \mu \gamma_2 \varepsilon_a \end{bmatrix}^{\top} e^{ 6d A_f^{\top}} P e^{ 6d A_f} \begin{bmatrix} \varepsilon_{\tau} {\minus} \varepsilon_{\tau} {+} \gamma_1  \varepsilon_a \\ \varepsilon_a {\minus} \mu \gamma_2 \varepsilon_a \end{bmatrix} {\minus} \varepsilon^{\top} P \varepsilon \\
& \hspace{-25mm} = \begin{bmatrix} I \varepsilon {+} \begin{bmatrix}
\minus 1 & \gamma_1 \\ 0 & \minus \mu \gamma_2
\end{bmatrix} \varepsilon \end{bmatrix}^{\top} e^{ 6d A_f^{\top}} P e^{ 6d A_f} \begin{bmatrix} I \varepsilon {+} \begin{bmatrix}
\minus 1 & \gamma_1 \\ 0 & \minus \mu \gamma_2
\end{bmatrix} \varepsilon \end{bmatrix} \\ 
& \hspace{-15mm} {\minus} \varepsilon^{\top} P  \varepsilon \\
& \hspace{-25mm} = \varepsilon \begin{bmatrix}
0 & \gamma_1 \\ 0 & 1 \minus \mu \gamma_2 \end{bmatrix}^{\top} e^{ 6d A_f^{\top}} P e^{ 6d A_f} \begin{bmatrix}
0 & \gamma_1 \\ 0 & 1 \minus \mu \gamma_2 \end{bmatrix} \varepsilon - \varepsilon^{\top} P  \varepsilon \\
& \hspace{-25mm} = \varepsilon^{\top} A_g^{\top} e^{ 6d A_f^{\top}} P e^{ 6d A_f} A_g  \varepsilon - \varepsilon^{\top} P  \varepsilon \\
& \hspace{-25mm} = \varepsilon^{\top} \big ( A_g^{\top} e^{ 6d A_f^{\top}} P e^{ 6d A_f} A_g  -  P \big )  \varepsilon
\end{align*}
\normalsize}{
\begin{align*}
V(G_{\varepsilon_6}(x_{\varepsilon})) {-} V(x_{\varepsilon}) & = \begin{bmatrix} \varepsilon + \begin{bmatrix} K_{\tilde{o}}(\m^i) \\ \minus K_{a}(\m^i,\tau_k)
\end{bmatrix} \end{bmatrix}^{\top} \exp \big ( 6d A_f^{\top} \big ) P \exp \big ( 6d A_f \big ) \begin{bmatrix} \varepsilon + \begin{bmatrix} K_{\tilde{o}}(\m^i) \\ \minus K_{a}(\m^i,\tau_k)
\end{bmatrix} \end{bmatrix} \\
& \hspace{5mm} - \varepsilon^{\top} P  \varepsilon \\
& \hspace{-20mm} = \begin{bmatrix} \varepsilon_{\tau} - \varepsilon_{\tau} + \gamma_1  \varepsilon_a \\ \varepsilon_a - \mu \gamma_2 \varepsilon_a \end{bmatrix}^{\top} \exp \big ( 6d A_f^{\top} \big ) P \exp \big ( 6d A_f \big ) \begin{bmatrix} \varepsilon_{\tau} - \varepsilon_{\tau} + \gamma_1  \varepsilon_a \\ \varepsilon_a - \mu \gamma_2 \varepsilon_a \end{bmatrix} \\
& \hspace{5mm} - \varepsilon^{\top} P \varepsilon \\
& \hspace{-20mm} = \varepsilon \begin{bmatrix}
0 & \gamma_1 \\ 0 & 1 \minus \mu \gamma_2 \end{bmatrix}^{\top} \exp \big ( 6d A_f^{\top} \big ) P \exp \big ( 6d A_f \big ) \begin{bmatrix}
0 & \gamma_1 \\ 0 & 1 \minus \mu \gamma_2 \end{bmatrix} \varepsilon - \varepsilon^{\top} P  \varepsilon \\
& \hspace{-20mm} = \varepsilon^{\top} A_g^{\top} \exp \big ( 6d A_f^{\top} \big ) P \exp \big ( 6d A_f \big ) A_g  \varepsilon - \varepsilon^{\top} P  \varepsilon \\
& \hspace{-20mm} = \varepsilon^{\top} \Big ( A_g^{\top} \exp \big ( 6d A_f^{\top} \big ) P \exp \big ( 6d A_f \big ) A_g  -  P \Big )  \varepsilon
\end{align*}}
\noindent
 for each $x_{\varepsilon} \in D_{\varepsilon_6} \cap \mathcal{M}_6$.  Then, by continuity of condition (\ref{eqn:lyap_cond}),  there exists $\sigma$ as in (\ref{eqn:sigma}) such that 
$$V(G_{\varepsilon_6}) - V(x_{\varepsilon}) \leq -\sigma \varepsilon^\top \varepsilon$$
 for each $x_{\varepsilon} \in D_{\varepsilon_6} \cap \mathcal{M}_6$.
\ifbool{conf}{\hfill $\square$ \end{pf}}{
\end{proof}}}

\ifbool{conf}{}{\begin{remark}
Observe that condition (\ref{eqn:lyap_cond}) may be difficult
to satisfy numerically as it may not be convex in $\mu$ and $P$. The authors in \cite{97}
utilize a polytopic embedding strategy to arrive at a linear
matrix inequality in which one needs to find some matrices
$X_i$ such that the exponential matrix is an element in
the convex hull of the $X_i$ matrices. Such an algorithm can be adapted to our setting.
\end{remark}}


\ifbool{conf}{}{
\ifbool{conf}{\begin{thm}}{
\begin{theorem}} \label{thm}
Let the hybrid system $\HS_{\varepsilon}$  in (\ref{eqn:Hy2})  with constants $d \geq c > 0$ be given. If there exist a constant $\mu > 0$ and a positive definite symmetric matrix $P$ such that 
 (\ref{eqn:lyap_cond}) holds  with $\gamma_1 = \frac{3c + 4d}{2}$ and $\gamma_2 = 2(c + d)$, and $\sigma$ as in (\ref{eqn:sigma}) such that
\begin{equation} 
\eta^{\frac{1}{6}} \hspace{1mm}  \rho \hspace{1mm} < 1
\end{equation}
\noindent
with $\eta = \big | 1 - \frac{\sigma }{\alpha_2} \big |$ and $\rho = \exp \Big ( \frac{\gamma c}{ 2 \alpha_2} \Big )$ holds, where $\alpha_2$ and $\gamma$ are as given in (\ref{eqn:thm_gamma}) and (\ref{eqn:thm_gamma}), respectively,  then $\A_{\varepsilon}$ is globally attractive for $\HS_{\varepsilon}$. Moreover, every maximal solution $\phi_{\varepsilon}$ to $\HS_{\varepsilon}$  with $\phi(0,0) \in \big ( C_{\varepsilon} \cup D_{\varepsilon} \big ) \cap \mathcal{M}$,   satisfies 
\begin{equation} \label{eqn:bound_phi}
\begin{aligned}
|\phi(t,j)|_{\A_{\varepsilon}} & \leq \sqrt{  \frac{\alpha_2}{\alpha_1} \eta^{\frac{j}{6} }  \rho^{ j } \exp \Big ( \frac{\gamma c}{\alpha_2} \Big ) } |\phi(0,0)|_{\A_{\varepsilon}} & & \forall (t,j) \in \mbox{\rm dom } \phi
\end{aligned}
\end{equation}
\noindent
where  $$\alpha_1 =  \underset{\nu \in \mathcal{Q}, \sigma \in \mathcal{P}}{\lambda_{\min}} \Big ( \exp \big ( (\nu  h(\nu) + d(5 - \sigma) ) A_f^{\top} \big ) P \exp \big ( \big (\nu h(\nu) + d(5 - \sigma) \big ) A_f  \big ) \Big )$$ and, consequently, $\lim_{t+j \to \infty} |\phi(t,j)|_{\A_{\varepsilon}} = 0$.
\ifbool{conf}{\end{thm}}{
\end{theorem}}}

\ifbool{conf}{}{
\begin{proof}
Pick a maximal solution with initial condition $\phi_{\varepsilon}(0,0) \in \big ( C_{\varepsilon} \cup D_{\varepsilon} \big ) \cap \mathcal{M}$. Recall the function $V$ in (\ref{eqn:v_bounds}), from the proof of Lemma \ref{lem:V_jumps} we have that
\begin{equation} \label{eqn:G6_jump}
V(G_{\varepsilon_6}(x_{\varepsilon})) - V(x_{\varepsilon}) \leq -\sigma \varepsilon^\top \varepsilon \hspace{1cm} \forall x_{\varepsilon} \in D_{\varepsilon_6} \cap \mathcal{M}
\end{equation}
and from the definition of $V$ in  (\ref{eqn:v_bounds}), there exists a positive scalar  $\alpha_2$ as in (\ref{eqn:thm_alpha2})  such that
\begin{align*}
V(x_{\varepsilon}) & \leq \alpha_2 |x_{\varepsilon}|^2_{\A_{\varepsilon}} 
\end{align*}
rearranging terms one then has
\begin{align*} 
- |x_{\varepsilon}|^2_{\A_{\varepsilon}} & \leq - \frac{1}{\alpha_2} V(x_{\varepsilon})
\end{align*}
Then, by making the appropriate substitutions in (\ref{eqn:G6_jump}), since $\varepsilon^\top \varepsilon = |x_{\varepsilon}|^2_{\A_{\varepsilon}}$ one has
\begin{align*}
V(G_{\varepsilon_6}(x_{\varepsilon})) - V(x_{\varepsilon}) & \leq - \frac{\sigma}{\alpha_2} V(x_{\varepsilon}) \\
V(G_{\varepsilon_6}(x_{\varepsilon})) & \leq \Big | 1 - \frac{\sigma }{\alpha_2} \Big | V(x_{\varepsilon})
\end{align*}
\noindent
 From Lemma \ref{lem:V_flows} we have that for each $x_{\varepsilon} \in C_{\varepsilon}$, 
\begin{equation}
\langle \nabla V(x_{\varepsilon}), F_{\varepsilon}(x_{\varepsilon}) \rangle \leq \begin{cases} 
0 & \mbox{ if } x_{\varepsilon} \in C_{\varepsilon_2} \\
\frac{\gamma}{\alpha_2} V(x_{\varepsilon}) & \mbox{ if } x_{\varepsilon} \in C_{\varepsilon_1}
\end{cases}
\end{equation}
and from Lemma \ref{lem:V_jumps} we have that for each $x_{\varepsilon} \in D_{\varepsilon} \cap \mathcal{M}$,
\begin{equation}
V(G_{\varepsilon_\ell}(x_{\varepsilon})) \leq \begin{cases} 
V(x_{\varepsilon}) & \mbox{ if } \ell \in \{1, 2, 3, 4, 5 \} \\
\Big ( 1 - \frac{\sigma }{\alpha_2} \Big ) V(x_{\varepsilon}) & \mbox{ if } \ell = 6
\end{cases}
\end{equation}
Pick a solution $\phi$ to $\HS_{\varepsilon}$ with $\phi_{\varepsilon}(0,0) \in C_{\varepsilon} \cap \mathcal{M}_1$. Then for each $(t,j) \in [0, t_1] \times \{0 \}$
$$V  \big ( \phi_{\varepsilon} (t,0) \big ) \leq \exp \Big ( \frac{\gamma}{\alpha_2} (t_1 - 0) \Big )  V  \big ( \phi_{\varepsilon} (0,0) \big )$$
At $(t_1,1)$, following a reset according to $G_{\varepsilon_1}$ one has
$$V  \big ( \phi_{\varepsilon} (t_1,1) \big ) \leq V  \big ( \phi_{\varepsilon} (t_1,0) \big )$$
Then, since $\phi_q(t_1,1) = 1$ for each $(t,j) \in [t_1, t_2] \times \{1 \}$, we obtain
$$V  \big ( \phi_{\varepsilon} (t,1) \big ) \leq  V  \big ( \phi_{\varepsilon} (t_1,1) \big )$$
At $(t_2,2)$, following a reset according to $G_{\varepsilon_2}$ one has
$$V  \big ( \phi_{\varepsilon} (t_2,2) \big ) \leq V  \big ( \phi_{\varepsilon} (t_2,1) \big )$$
Then since $\phi_q(t_2,2) = 0$ for each $(t,j) \in [t_2, t_3] \times \{2 \}$, we obtain
$$V  \big ( \phi_{\varepsilon} (t,2) \big ) \leq \exp \Big ( \frac{\gamma}{\alpha_2} (t_3-t_2) \Big ) V  \big ( \phi_{\varepsilon} (t_2,2) \big )$$
At $(t_3,3)$, following a reset according to $G_{\varepsilon_3}$ one has
$$V  \big ( \phi_{\varepsilon} (t_3,3) \big ) \leq V  \big ( \phi_{\varepsilon} (t_3,3) \big )$$
Then since $\phi_q(t_3,3) = 1$ for each $(t,j) \in [t_3, t_4] \times \{3 \}$, we obtain
$$V  \big ( \phi_{\varepsilon} (t,3) \big ) \leq V  \big ( \phi_{\varepsilon} (t_3,2) \big )$$
At $(t_4,4)$ following a reset according to $G_{\varepsilon_4}$ one has
$$V  \big ( \phi_{\varepsilon} (t_4,4) \big ) \leq V  \big ( \phi_{\varepsilon} (t_4,3) \big )$$
Then since $\phi_q(t_4,4) = 0$ for each $(t,j) \in [t_4, t_5] \times \{4 \}$, we obtain
$$V  \big ( \phi_{\varepsilon} (t,4) \big ) \leq \exp \Big ( \frac{\gamma}{\alpha_2} (t_5 - t_4) \Big ) V  \big ( \phi_{\varepsilon} (t_4,4) \big )$$
At $(t_5,5)$, following a reset according to $G_{\varepsilon_5}$ one has
$$V  \big ( \phi_{\varepsilon} (t_5,5) \big ) \leq V  \big ( \phi_{\varepsilon} (t_5,4) \big )$$
then since $\phi_q(t_5,5) = 1$ for each $(t,j) \in [t_5, t_6] \times \{5 \}$, we obtain
$$V  \big ( \phi_{\varepsilon} (t,5) \big ) \leq V  \big ( \phi_{\varepsilon} (t_5,5) \big )$$
At $(t_6,6)$, following a reset according to $G_{\varepsilon_6}$ one has
$$V  \big ( \phi_{\varepsilon} (t_6,6) \big ) \leq \Big | 1 - \frac{\sigma }{\alpha_2} \Big | V  \big ( \phi_{\varepsilon} (t_6,5) \big )$$
\noindent
Making the appropriate substitutions one has
$$V  \big ( \phi_{\varepsilon} (t_6,6) \big ) \leq \Big | 1 - \frac{\sigma }{\alpha_2} \Big | \exp \Big ( \frac{\gamma}{\alpha_2} (t_5 - t_4) \Big ) \exp \Big ( \frac{\gamma}{\alpha_2} (t_3 - t_2) \Big ) \exp \Big ( \frac{\gamma}{\alpha_2} (t_1 - 0) \Big ) V \big ( \phi_{\varepsilon} (0,0) \big )$$
\noindent
leading to a general bound of the form
\begin{equation} \label{eqn:bound1} 
V  \big ( \phi_{\varepsilon} (t,j) \big ) \leq \Big | 1 - \frac{\sigma }{\alpha_2} \Big |^{\lfloor\frac{j}{6} \rfloor} \bigg ( \prod_{k = 0}^{ \lfloor \frac{j-1}{2} \rfloor} \exp \Big ( \frac{\gamma}{\alpha_2} (t_{2k+1} - t_{2k}) \Big ) \bigg )  V \big ( \phi_{\varepsilon}(0,0) \big )
\end{equation}
\noindent
However, by noting the bounds in (\ref{eqn:t_bounds}) one has that $t_{j+1} - t_j \leq c(j+1)$ for each $j \in \{2i : i \in \mathbb{N} \}, j > 0$,  then assuming $\gamma > 0$,  the bound in (\ref{eqn:bound1}) reduces to
\begin{equation*} 
V  \big ( \phi_{\varepsilon} (t,j) \big ) \leq \Big | 1 - \frac{\sigma }{\alpha_2} \Big |^{\lfloor\frac{j}{6} \rfloor} \bigg ( \prod_{k = 0}^{ \lfloor \frac{j-1}{2} \rfloor} \exp \Big ( \frac{\gamma}{\alpha_2} (c(2k + 1)) \Big ) \bigg ) V \big ( \phi_{\varepsilon} (0,0) \big ) 
\end{equation*}
\begin{equation*} 
V  \big ( \phi_{\varepsilon} (t,j) \big ) \leq \Big | 1 - \frac{\sigma }{\alpha_2} \Big |^{\lfloor\frac{j}{6} \rfloor}  \bigg ( \exp \Big ( \frac{\gamma c}{\alpha_2} \Big ) \bigg )^{ \lceil \frac{j}{2} \rceil} V \big ( \phi_{\varepsilon} (0,0) \big )
\end{equation*}
\noindent 
Using the relation $\lceil \frac{j}{2} \rceil = \frac{j}{2} + 1$ we then have 
\begin{equation*}
\begin{aligned}
V  \big ( \phi_{\varepsilon} (t,j) \big ) & \leq \Big | 1 - \frac{\sigma }{\alpha_2} \Big |^{\lfloor\frac{j}{6} \rfloor}  \bigg ( \exp \Big ( \frac{\gamma c}{\alpha_2} \Big ) \bigg )^{ \frac{j}{2} } \exp \Big ( \frac{\gamma c}{\alpha_2} \Big ) V \big ( \phi_{\varepsilon} (0,0)  \big ) \\
\end{aligned}
\end{equation*}
Then noting that $\lfloor \frac{j}{6} \rfloor \leq \frac{j}{6}$ 
\begin{equation*}
\begin{aligned}
V(t,j) & \leq \Big | 1 - \frac{\sigma }{\alpha_2} \Big |^{\frac{j}{6}}  \bigg ( \exp \Big ( \frac{\gamma c}{ 2 \alpha_2} \Big ) \bigg )^{ j } \exp \Big ( \frac{\gamma c}{\alpha_2} \Big ) V(0,0)
\end{aligned}
\end{equation*}
Then given the definition of $V$ in (\ref{eqn:v_bounds}) we have that
\begin{equation} \label{eqn:bound2}
\begin{aligned}
\alpha_1 |x_{\varepsilon}|^2_{\A_{\varepsilon}} \leq V \big ( \phi_{\varepsilon} (t,j) \big ) & \leq \Big | 1 - \frac{\sigma }{\alpha_2} \Big |^{\frac{j}{6}}  \bigg ( \exp \Big ( \frac{\gamma c}{ 2 \alpha_2} \Big ) \bigg )^{ j } \exp \Big ( \frac{\gamma c}{\alpha_2} \Big ) V \big ( \phi_{\varepsilon} (0,0)  \big )
\end{aligned}
\end{equation}
 Finally, by leveraging $V(\phi(0,0)) \leq \alpha_2 |\phi(0,0)|^2_{\A_{\varepsilon}}$, we arrive at (\ref{eqn:bound_phi}).
\end{proof}}


\ifbool{conf}{
\ifbool{conf}{\begin{thm}}{
\begin{theorem}} \label{thm}
Let the hybrid system $\HS_{\varepsilon}$ with constants $d = c > 0$ be given. If there exist a constant $\mu > 0$ and positive definite symmetric matrix $P$ such that 
\begin{equation} \label{eqn:lyap_cond}
\begin{aligned}
A_g^{\top} e^{ 6d A_f^{\top} } P e^{6d A_f } A_g  -  P \prec 0
\end{aligned}
\end{equation}
is satisfied where $A_{g} = \begin{bmatrix} 0 & \gamma_1 \\ 0 &  1 {\minus} \mu \gamma_2 \end{bmatrix}$ with $\gamma_1 = \frac{7}{2}c$ and $\gamma_2 = 4c$, then $\A_{\varepsilon}$ is globally attractive for $\HS_{\varepsilon}$.
\ifbool{conf}{\end{thm}}{
\end{theorem}}

\ifbool{conf}{

To prove this result, we first show the existence of a forward invariant and finite time attractive set that enforces valid initialization values of the logic variables $p$, $q$ and memory state vectors $\m^i$ and $\m^k$ such that the update laws $K_{\tilde{o}}$ and $K_a$ give the input for the convergence of $\varepsilon$. 
We then show that for $x_{\varepsilon} \in C_{\varepsilon}$, $V$ has the infinitesimal property of being constant during flows, namely  $$\langle \nabla V(x_{\varepsilon}), f_{\varepsilon}(x_{\varepsilon}) \rangle = 0$$ By continuity of the condition in (\ref{eqn:lyap_cond}), there exists $\sigma > 0$ such that, within the initialization, set $V$ is constant or strictly decreasing during jumps. Namely for each $x_{\varepsilon} \in D_{\varepsilon}$, $$V \big (G_{\ell}(x_{\varepsilon}) \big ) - V \big (x_{\varepsilon} \big ) \leq 0$$ for each $\ell \in \{1,2,3,4,5\}$, and $$V \big ( G_{6}(x_{\varepsilon}) \big ) - V \big (x_{\varepsilon} \big ) \leq -\sigma \varepsilon^{\top} \varepsilon$$
Then by picking a solution $\phi$ from the set of solutions to $\HS_{\varepsilon}$ and evaluating $V$ along the solution $\phi$. We show that, following attractivity to the invariant initialization set, the infinitesimal properties of $V$ gives that $\phi$ converges to $\A_{\varepsilon}$ thus, Problem \ref{prob:1} is solved. Due to space constraints, the complete proof has been omitted and will be published elsewhere.}{
\ifbool{conf}{\begin{pf}}{
\begin{proof}}
Pick a maximal solution $\phi \in \mathcal{S}_{\HS_{\varepsilon}}$. By Lemma \ref{lem:finite_M}, there exists $T^*$ such that $\phi(t,j) \in \mathcal{M}$ for all $(t,j) \in \mbox{dom } \phi$ with $t + j \geq T^*$. Now observe, by Lemma \ref{lem:V_flows}, when $\phi(t,j) \in C_{\varepsilon}$ the derivative of $V(\phi(t,j))$ along flows for the chosen solution is given by $$\langle \nabla V(\phi(t,j)), f(\phi(t,j)) \rangle = 0$$ for all $(t,j) \in \mbox{dom } \phi$. Moreover, Lemma \ref{lem:V_jumps} gives that when $\phi(t,j) \in \mathcal{M} \cap D_{\varepsilon}$ and $(t,j+1) \in \mbox{dom } \phi$, $V(x_{\varepsilon})$ evaluated at jumps for the chosen solution has that $$V(G_{\ell}(\phi(t,j))) - V(\phi(t,j)) \leq 0$$ for each $\ell \in \{1,2,3,4,5\}$ and $$V(G_{\varepsilon_6}(\phi(t,j))) - V(\phi(t,j)) \leq \minus \sigma \varepsilon^{\top} \varepsilon$$ Then by noting that the length of each flow interval $[t_j,t_{j+1}] \times \{j\}$ is finite due to the flow map $C_{\varepsilon}$ and its dependence on $\tau$, the solution will always periodically jump according to $G_{\varepsilon}(x_{\varepsilon})$. In particular, the solution will always periodically jump according to $G_{\varepsilon_6}(x_{\varepsilon})$ due to the monotonic behavior of $p$ following resets according to $G_{\varepsilon_{\ell}}(x_{\varepsilon})$ for each $\ell \in \{1,2,3,4,5\}$, thus the solution converges to $\A_{\varepsilon}$.
\ifbool{conf}{\hfill $\square$ \end{pf}}{
\end{proof}}}}{}

\section{About the Multi-Agent Case} \label{sec:multi_agent}

 In this section, we present an extension to the proposed algorithm model to capture the scenario of synchronizing multiple networked agents. For such a setting, we consider  a network system of $n$ nodes in a  leader-follower scenario where there exists a single designated reference node to which all the connected $N = n - 1$ nodes synchronize. To this end,  let $\tau_R \in \reals$ define the clock of the designated reference node and $\tau_S := (\tau_{S_1}, \tau_{S_2}, \ldots, \tau_{S_{N}}) \in \reals^{N}$ define the clocks of the synchronizing child nodes,  where $\tau_{S_i}$ is the clock of the $i$-th child node.  Moreover, we let $a_R \in \reals$ and $a := (a_1, a_2, \ldots, a_{N}) \in \reals^{N}$ define the skews of the reference clock and synchronizing clocks, respectively,  where $a_{i}$ is the clock skew of the $i$-th child node.  Given the leader-follower architecture to synchronize the nodes, the algorithm in  $\HS$ is modified such that the algorithm modeled by $\HS$ is executed for each synchronizing node. In particular, the algorithm executes the synchronization process given by \ref{itm:one}-\ref{itm:six} for the reference node $\tau_R$ and the $i$-th child node $\tau_{S_i}$. Upon completion, the algorithm then executes the same synchronization steps \ref{itm:one}-\ref{itm:six} for the reference node and the  $\mbox{mod } (i+1, N)$-th child node.  This procedure is repeated recurrently and cyclically for each pair reference-child node in the network. To enable the modeling of such an algorithm, we define:
\begin{itemize}
\item A discrete variable $\ell \in \{1, 2, \ldots, N \} =: \mathcal{S}$ that indexes the node to be synchronized. The variable remains constant during flows, namely, $\dot{s} = 0$, and resets to either $s + 1$ upon the completion of the synchronization algorithm for $s \in \{1, 2, \ldots, n-2 \}$ or is reset to $1$ when $s = n -1$. 
\item For each $\ell \in \mathcal{S}$, a timer variable $\tilde{\tau}_{\ell} \in [0, 3c + 3d]$ to track the execution of the synchronization algorithm for the respective $\ell$-th child node, with dynamics
\begin{align*}
& \dot{\tilde{\tau}}_{\ell} = -1 & & \tilde{\tau}_{\ell} \in [0,3c + 3d] \\
& \tilde{\tau}_{\ell}^+ = 3c + 3d  & & \tilde{\tau}_{\ell} = 0
\end{align*}
\noindent
for each $\ell \in \mathcal{S}$. The value $3c + 3d$ reflects the duration of the synchronization algorithm executed between the reference and the synchronizing node capturing the total time elapsed during message transmission and residence delay.
\end{itemize}

The state of this multi-agent system is given by 
\begin{equation*}
\tilde{x} := (\tau_R, \tau_S, a_R, a, \tau, \tilde{\tau}, \m^R, \m^S, \ell, p, q) \in \widetilde{\mathcal{X}}
\end{equation*}
where $\tilde{\tau} := (\tilde{\tau}_1, \tilde{\tau}_2, \ldots, \tilde{\tau}_{N} )$,  $\m^R = [\m^R_1, \m^R_2, \m^R_3, \m^R_4, \m^R_5, \m^R_6]^{\top} \in \reals^6$, $\m^S = [\m^S_1, \m^S_2, \m^S_3, \m^S_4, \m^S_5, \m^S_6]^{\top} \in \reals^6$  and
\begin{equation*}
\widetilde{\mathcal{X}} := \reals \times \reals^{N} \times \reals \times \reals^{N} \times [0,d] \times [0, 3c + 3d]^{N} \times \mathbb{R}^6 \times \mathbb{R}^6 \times \mathcal{S} \times \mathcal{P} \times \mathcal{Q}
\end{equation*}
Then by noting the dynamics of the clocks as given in (\ref{eqn:clocks_dyn}) and those of the timer $\tilde{\tau}$ above, the continuous dynamics of $\tilde{x}$ is given by
\begin{equation*}
\begin{aligned}
\dot{\tilde{x}} = (a_{R}, a ,0, \textbf{0}_{N \times 1}, -1, -\textbf{1}_{(N) \times 1}, \textbf{0}_{6 \times 1},\textbf{0}_{6 \times 1},0,0, 0) & & \tilde{x} \in \widetilde{C} := \widetilde{\mathcal{X}}
\end{aligned}
\end{equation*}
To model the discrete dynamics of the communication and arrival events of the exchanged timing messages, in addition to the subsequent corrections on the clock rate and offset, we consider the jump map $\widetilde{G}(\tilde{x}) := \{\widetilde{G}^{i}(\tilde{x}) : \tilde{x} \in \widetilde{D}^i, i \in \mathcal{S} \}$ where

\begin{equation*}
\widetilde{G}^i(\tilde{x}) =
\widetilde{G}^i_{k}(\tilde{x}) \quad \mbox{if } x \in \widetilde{D}^\ell_{k} 
\end{equation*} 
\noindent
and
\begin{equation*}
\begin{aligned}
& \hspace{-15mm}
\widetilde{G}^i_{1}(\tilde{x}) = \begin{bmatrix}
\renewcommand\arraystretch{0.7}
\tau_R \\
\tau_S \\
a_R \\
a \\
d \\
\tilde{\tau} \\
\begin{bmatrix}
\begin{bmatrix}
\tau_{R}, & \m_1^R, & \cdots, & \m_5^R
\end{bmatrix} & \m^s 
\end{bmatrix}^\top \\
\ell \\
p + 1 \\ 
1
\end{bmatrix},
& &
\widetilde{G}^i_{2}(\tilde{x}) = \begin{bmatrix}
\renewcommand\arraystretch{0.7}
\tau_R \\
\tau \\
a_R \\
a \\
c \\
\tilde{\tau} \\
\begin{bmatrix}
\m^R & \begin{bmatrix}
\tau_{i}, & \m_1^R, & \cdots, & \m_5^R
\end{bmatrix} 
\end{bmatrix}^\top \\
\ell \\
p + 1 \\ 
0
\end{bmatrix}, \\
& \hspace{-15mm}
\widetilde{G}^i_{3}(\tilde{x}) = \begin{bmatrix}
\renewcommand\arraystretch{0.7}
\tau_R \\
\tau_S \\
a_R \\
a \\
d \\
\tilde{\tau} \\
\begin{bmatrix}
\m^R & \begin{bmatrix}
\tau_{i}, & \m_1^s, & \cdots, & \m_5^s
\end{bmatrix} 
\end{bmatrix}^\top \\
\ell \\
p + 1 \\ 
1
\end{bmatrix},
& &
\widetilde{G}^i_{4}(\tilde{x}) = \begin{bmatrix}
\renewcommand\arraystretch{0.7}
\tau_R \\
\tau_S \\
a_R \\
a \\
c \\
\tilde{\tau} \\
\begin{bmatrix}
\begin{bmatrix}
\tau_{R}, & \m_1^s, & \cdots, & \m_5^s
\end{bmatrix} & \m^s 
\end{bmatrix}^\top \\
\ell \\
p + 1 \\ 
0
\end{bmatrix}, \\
& \hspace{-15mm}
\widetilde{G}^i_{5}(\tilde{x}) = \begin{bmatrix}
\renewcommand\arraystretch{0.7}
\tau_R \\
\tau_S \\
a_R \\
a \\
d \\
\tilde{\tau} \\
\begin{bmatrix}
\begin{bmatrix}
\tau_{R}, & \m_1^R, & \cdots, & \m_5^R
\end{bmatrix} & \m^s
\end{bmatrix}^\top \\
\ell \\
p + 1 \\ 
1
\end{bmatrix},
& &
\widetilde{G}^i_{6}(\tilde{x}) = \begin{bmatrix}
\renewcommand\arraystretch{0.7}
\tau_R \\
\tau - \begin{bmatrix} \textbf{0}_{i-1}, & K_{\tilde{o}}(\m^R), & \textbf{0}_{n - 1 - i} \end{bmatrix}^\top \\ 
a_R \\
a + \begin{bmatrix} \textbf{0}_{i-1}, & K_{a}(\m^R, \tau_{S_i}), & \textbf{0}_{n - 1 - i} \end{bmatrix}^\top \\
c \\
\begin{bmatrix} \tilde{\tau}_{1}, \ldots, \tilde{\tau}_{i},  3c + 3d, \tilde{\tau}_{i+2}, \ldots, \tilde{\tau}_{n {-} 1} \end{bmatrix}^{\top} \\
\begin{bmatrix}
\m^R & \begin{bmatrix}
\tau_{i}, & \m_1^R, & {\cdots}, & \m_5^R
\end{bmatrix}
\end{bmatrix}^\top \\
\ell + 1 \\
0 \\ 
0
\end{bmatrix}
\end{aligned}
\end{equation*}
\noindent
 The main idea behind the construction of the map $\widetilde{G}(\tilde{x}) := \{\widetilde{G}^{i}(\tilde{x}) : \tilde{x} \in \widetilde{D}^i, i \in \mathcal{S} \}$ is to capture protocol events of the synchronization algorithm for each child node.  To handle the condition where $\ell = n - 1$ such that the protocol cycles back to synchronizing the first node, we have the following jump map for $\widetilde{G}^{N}_{6}$,
\begin{equation*}
\begin{aligned}
\widetilde{G}^{N}_{6}(\tilde{x}) = \begin{bmatrix}
\renewcommand\arraystretch{0.7}
\tau_R \\
\tau - \begin{bmatrix} \textbf{0}_{i-1}, & K_{\tilde{o}}(\m^R), & \textbf{0}_{n - 1 - i} \end{bmatrix}^\top \\ 
a_R \\
a + \begin{bmatrix} \textbf{0}_{i-1}, & K_{a}(\m^R, \tau_i), & \textbf{0}_{n - 1 - i} \end{bmatrix}^\top \\
c \\
\begin{bmatrix} \tilde{\tau}_{1}, & \ldots, & \tilde{\tau}_{n-2}, &  3c + 3d \end{bmatrix}^{\top} \\
\begin{bmatrix}
\m^R & \begin{bmatrix}
\tau_{i}, & \m_1^R, & {\cdots}, & \m_5^R
\end{bmatrix}
\end{bmatrix}^\top \\
1 \\
0 \\ 
0
\end{bmatrix}
\end{aligned}
\end{equation*}
\noindent
To trigger the jump map corresponding to the particular protocol event  for each child node,  we define the jump set as $\widetilde{D} := \widetilde{D}^1 \cup \widetilde{D}^2 \cup \cdots \cup \widetilde{D}^i \cup \cdots \cup \widetilde{D}^{N}$ where $\widetilde{D}^i := \widetilde{D}^i_{1} \cup \widetilde{D}^i_{2} \cup \widetilde{D}^i_{3} \cup \widetilde{D}^i_{4} \cup \widetilde{D}^i_{5} \cup \widetilde{D}^i_{6}$ and
\begin{equation*}
\begin{aligned}
\widetilde{D}^i_{1} & := \{\tilde{x} \in \widetilde{\mathcal{X}} : \tau = 0, p = 0, \ell = i \}, \hspace{1cm}
\widetilde{D}^i_{2} := \{\tilde{x} \in \widetilde{\mathcal{X}} : \tau = 0, p = 1, \ell = i \} \\
\widetilde{D}^i_{3} & := \{\tilde{x} \in \widetilde{\mathcal{X}} : \tau = 0, p = 2, \ell = i \}, \hspace{1cm}
\widetilde{D}^i_{4} := \{\tilde{x} \in \widetilde{\mathcal{X}} : \tau = 0, p = 3, \ell = i \} \\
\widetilde{D}^i_{5} & := \{\tilde{x} \in \widetilde{\mathcal{X}} : \tau = 0, p = 4, \ell = i \}, \hspace{1cm}
\widetilde{D}^i_{6} := \{\tilde{x} \in \widetilde{\mathcal{X}} : \tau = 0, \tilde{\tau}_{\ell} = 0,\\ 
& \hspace{10.2cm} p = 5, \ell = i \}
\end{aligned}
\end{equation*}
\noindent
This hybrid system is denoted 
\begin{equation} \label{eqn:Hy2_multi}
\widetilde{\HS} = (\widetilde{C}, \widetilde{F}, \widetilde{D}, \widetilde{G})
\end{equation}

\subsubsection{Error Model}

With an abuse of notation, let $\varepsilon := (\varepsilon_1, \ldots, \varepsilon_{N}) \in \reals^{2(N)}$, where $\varepsilon_i = \begin{bmatrix} \tau_R  - \tau_i \\ a_R - a_i \end{bmatrix}$ for each $i \in \mathcal{S}$. Then, define $$\tilde{x}_{\varepsilon} := (\varepsilon, \tilde{x}) \in \widetilde{\mathcal{X}}_{\varepsilon} := \reals^{2(N)} \times \tilde{\mathcal{X}}$$ For each $\tilde{x}_{\varepsilon} \in \widetilde{C}_{\varepsilon} := \widetilde{\mathcal{X}}_{\varepsilon}$, the flow map is given by $$\widetilde{F}_{\varepsilon} (\tilde{x}_{\varepsilon}) = \big ( A_F \varepsilon, \widetilde{F}(\tilde{x}) \big )$$ where $$A_F = \begin{bmatrix}
A_f & \cdots & 0 \\
\vdots & \ddots & \vdots \\
0 & \cdots & A_f
\end{bmatrix}$$
\noindent
 is a blaock diagonal matrix with $N$ entries equal to  $A_f = \begin{bmatrix}
0 & 1 \\ 0 & 0
\end{bmatrix}$.

The discrete dynamics of the protocol are modeled through the jump map $\widetilde{G}_{\varepsilon}(\tilde{x}) := \{\widetilde{G}_{\varepsilon}^{i}(\tilde{x}) : \tilde{x}_{\varepsilon} \in \widetilde{D}_{\varepsilon}^i, i \in \mathcal{S} \}$ where 
\begin{equation*}
\widetilde{G}_{\varepsilon}^i(\tilde{x}_{\varepsilon}) {\mequals} \begin{cases} 
\widetilde{G}_{\varepsilon_1}^i(\tilde{x}_{\varepsilon}) \hspace{1mm} \mbox{if } \tilde{x}_{\varepsilon} \in \widetilde{D}_{\varepsilon_1}^i {\setminus} (\widetilde{D}_{\varepsilon_2}^i {\cup} \widetilde{D}_{\varepsilon_3}^i {\cup} \widetilde{D}_{\varepsilon_4}^i {\cup} \widetilde{D}_{\varepsilon_5}^i {\cup} \widetilde{D}_{\varepsilon_6}^i)  \\ 
\widetilde{G}_{\varepsilon_2}^i(\tilde{x}_{\varepsilon}) \hspace{1mm} \mbox{if } \tilde{x}_{\varepsilon} \in \widetilde{D}_{\varepsilon_2}^i {\setminus} (\widetilde{D}_{\varepsilon_1}^i {\cup} \widetilde{D}_{\varepsilon_3}^i {\cup} \widetilde{D}_{\varepsilon_4}^i {\cup} \widetilde{D}_{\varepsilon_5}^i {\cup} \widetilde{D}_{\varepsilon_6}^i) \\ 
\widetilde{G}_{\varepsilon_3}^i(\tilde{x}_{\varepsilon}) \hspace{1mm} \mbox{if } \tilde{x}_{\varepsilon} \in \widetilde{D}_{\varepsilon_3}^i {\setminus} (\widetilde{D}_{\varepsilon_1}^i {\cup} \widetilde{D}_{\varepsilon_2}^i {\cup} \widetilde{D}_{\varepsilon_4}^i {\cup} \widetilde{D}_{\varepsilon_5}^i {\cup} \widetilde{D}_{\varepsilon_6}^i) \\ 
\widetilde{G}_{\varepsilon_4}^i(\tilde{x}_{\varepsilon}) \hspace{1mm} \mbox{if } \tilde{x}_{\varepsilon} \in \widetilde{D}_{\varepsilon_4}^i {\setminus} (\widetilde{D}_{\varepsilon_1}^i {\cup} \widetilde{D}_{\varepsilon_2}^i {\cup} \widetilde{D}_{\varepsilon_3}^i {\cup} \widetilde{D}_{\varepsilon_5}^i {\cup} \widetilde{D}_{\varepsilon_6}^i) \\ 
\widetilde{G}_{\varepsilon_5}^i(\tilde{x}_{\varepsilon}) \hspace{1mm} \mbox{if } \tilde{x}_{\varepsilon} \in \widetilde{D}_{\varepsilon_5}^i {\setminus} (\widetilde{D}_{\varepsilon_1}^i {\cup} \widetilde{D}_{\varepsilon_2}^i {\cup} \widetilde{D}_{\varepsilon_3}^i {\cup} \widetilde{D}_{\varepsilon_4}^i {\cup} \widetilde{D}_{\varepsilon_6}^i) \\ 
\widetilde{G}_{\varepsilon_6}^i(\tilde{x}_{\varepsilon}) \hspace{1mm} \mbox{if } \tilde{x}_{\varepsilon} \in \widetilde{D}_{\varepsilon_6}^i {\setminus} (\widetilde{D}_{\varepsilon_1}^i {\cup} \widetilde{D}_{\varepsilon_2}^i {\cup} \widetilde{D}_{\varepsilon_3}^i {\cup} \widetilde{D}_{\varepsilon_4}^i {\cup} \widetilde{D}_{\varepsilon_5}^i)
\end{cases}
\end{equation*}
\noindent
where
\begin{equation*}
\begin{aligned}
& \widetilde{G}_{\varepsilon_1}^i(\tilde{x}_{\varepsilon}) {=} \begin{bmatrix}
\renewcommand\arraystretch{0.7}
\varepsilon \\
\widetilde{G}_1^i(\tilde{x})
\end{bmatrix},
\hspace{1mm}
\widetilde{G}_{\varepsilon_2}^i(\tilde{x}_{\varepsilon}) {=} \begin{bmatrix}
\renewcommand\arraystretch{0.7}
\varepsilon \\
\widetilde{G}_2^i(\tilde{x})
\end{bmatrix},
\hspace{1mm}
\widetilde{G}_{\varepsilon_3}^i(\tilde{x}_{\varepsilon}) {=} \begin{bmatrix}
\renewcommand\arraystretch{0.7}
\varepsilon \\
\widetilde{G}_3^i(\tilde{x})
\end{bmatrix}, \\
& \widetilde{G}_{\varepsilon_4}^i(\tilde{x}_{\varepsilon}) {=} \begin{bmatrix}
\renewcommand\arraystretch{0.7}
\varepsilon \\
\widetilde{G}_4^i(\tilde{x})
\end{bmatrix},
\hspace{1mm}
\widetilde{G}_{\varepsilon_5}^i(\tilde{x}_{\varepsilon}) {=} \begin{bmatrix}
\renewcommand\arraystretch{0.7}
\varepsilon \\
\widetilde{G}_5^i(\tilde{x})
\end{bmatrix},
\hspace{1mm}
\widetilde{G}_{\varepsilon_6}^i(\tilde{x}_{\varepsilon}) {=} \begin{bmatrix}
\renewcommand\arraystretch{0.7}
\begin{bmatrix} \varepsilon_1, \ldots, \varepsilon_i^+, \ldots, \varepsilon_{N}
\end{bmatrix}^\top \\
\widetilde{G}_6^i(\tilde{x})
\end{bmatrix}
\end{aligned}
\end{equation*}
\normalsize
\noindent
where $$\varepsilon_i^+ = \begin{bmatrix}
\tau_R - \big ( \tau_i - K_{\tilde{o}}(\tilde{x}) \big ) \\
a_R - \big ( a_i + K_a(\tilde{x}) \big )
\end{bmatrix}$$
\noindent
These discrete dynamics apply for $x$ in $\widetilde{D}_{\varepsilon} := \widetilde{D}_{\varepsilon}^1 \cup \widetilde{D}_{\varepsilon}^2 \cup \cdots \cup \widetilde{D}_{\varepsilon}^i \cup \cdots \cup \widetilde{D}_{\varepsilon}^{N}$, where $\widetilde{D}_{\varepsilon}^i := \widetilde{D}_{\varepsilon_1}^i \cup \widetilde{D}_{\varepsilon_2}^i \cup \widetilde{D}_{\varepsilon_3}^i \cup \widetilde{D}_{\varepsilon_4}^i \cup \widetilde{D}_{\varepsilon_5}^i \cup \widetilde{D}_{\varepsilon_6}^i$ and
\begin{equation*}
\begin{aligned}
\widetilde{D}_{\varepsilon_1}^i & := \{\tilde{x}_{\varepsilon} \in \widetilde{\mathcal{X}}_{\varepsilon} : \tau = 0, p = 0, \ell = i \}, \hspace{1cm}
\widetilde{D}_{\varepsilon_2}^i := \{\tilde{x}_{\varepsilon} \in \widetilde{\mathcal{X}}_{\varepsilon} : \tau = 0, p = 1, \ell = i \} \\
\widetilde{D}_{\varepsilon_3}^i & := \{\tilde{x}_{\varepsilon} \in \widetilde{\mathcal{X}}_{\varepsilon} : \tau = 0, p = 2, \ell = i \}, \hspace{1cm}
\widetilde{D}_{\varepsilon_4}^i := \{\tilde{x}_{\varepsilon} \in \widetilde{\mathcal{X}}_{\varepsilon} : \tau = 0, p = 3, \ell = i \} \\
\widetilde{D}_{\varepsilon_5}^i & := \{\tilde{x}_{\varepsilon} \in \widetilde{\mathcal{X}}_{\varepsilon} : \tau = 0, p = 4, \ell = i \}, \hspace{1cm}
\widetilde{D}_{\varepsilon_6}^i := \{\tilde{x}_{\varepsilon} \in \widetilde{\mathcal{X}}_{\varepsilon} : \tau = 0, \tilde{\tau}_{\ell} = 0,\\ 
& \hspace{10.8cm} p = 5, \ell = i \}
\end{aligned}
\end{equation*}
\noindent
This hybrid system is denoted 
\begin{equation} \label{eqn:Hy2_multi}
\widetilde{\HS}_{\varepsilon} = (\widetilde{C}_{\varepsilon}, \widetilde{F}_{\varepsilon}, \widetilde{D}_{\varepsilon}, \widetilde{G}_{\varepsilon})
\end{equation}
and the set to render attractive for the multi-agent model is given by
\begin{equation} \label{set:A_multi}
\widetilde{\A}_{\varepsilon} := \{ \tilde{x}_{\varepsilon} \in \widetilde{\mathcal{X}}_{\varepsilon} : \varepsilon_i = 0 \hspace{1mm} \forall i \in \mathcal{S} \}
\end{equation}

\ifbool{HSLreport}{

With the system defined in this manner, we certify attractivity of $\widetilde{\HS}_{\varepsilon}$ to $\widetilde{\A}_{\varepsilon}$ through an extension of the results for the two-agent model given in Theorem \ref{thm}. In the same order as the two-agent model, we proceed as follows:
\begin{itemize}
\item Define and give a result for a forward invariant set $\mathcal{M}$ that gives the correct values of  $\m^R$ and $\m^i$ such that the update laws $K_{\tilde{o}}$ and $K_{a}$ give the correct values.
\item Show that the forward invariant set $\mathcal{M}$ is finite time attractive.
\item Utilizing the same Lyapunov-based approach for the two-agent system, a function $\widetilde{V} : \widetilde{\mathcal{X}}_{\varepsilon} \to \reals_{\geq 0}$ is defined as an extension of the function given in (\ref{eqn:lyap_fun}).
\item Infinitesimal properties of $\widetilde{V}$ are established across flows and jumps of the system $\widetilde{\HS}_{\varepsilon}$ via separate lemmas.
\end{itemize}
\noindent
Following these results, the main result is presented showing attractivity of the set of interest. To demonstrate the feasibility of the algorithm, a numerical example featuring a three-agent system illustrating the convergence properties of $\widetilde{\HS}_{\varepsilon}$ is simulated.

}{

With the system defined in this manner, we certify attractivity of $\widetilde{\HS}_{\varepsilon}$ to $\widetilde{\A}_{\varepsilon}$ through an extension of the results for the two-agent model. In the same order as the two-agent model, we proceed as follows:
\begin{itemize}
\item Define and give a result for a forward invariant set $\mathcal{M}$ that gives the correct values of  $\m^R$ and $\m^i$ such that the update laws $K_{\tilde{o}}$ and $K_{a}$ give the correct values.
\item Show that the forward invariant set $\mathcal{M}$ is finite time attractive.
\item Utilizing the same Lyapunov-based approach for the two-agent system, a function $\widetilde{V} : \widetilde{\mathcal{X}}_{\varepsilon} \to \reals_{\geq 0}$ is defined as an extension of the function given in (\ref{eqn:lyap_fun}).
\item Infinitesimal properties of $\widetilde{V}$ are established across flows and jumps of the system $\widetilde{\HS}_{\varepsilon}$ via separate lemmas.
\end{itemize}
\noindent
Due to the similar nature of these incremental results, we have postponed the details of the results and their proofs to ... Instead, the main result is presented and to  demonstrate the feasibility of the model, a numerical example illustrating the convergence properties of $\widetilde{\HS}_{\varepsilon}$ is included in the following section where a three-agent system is simulated.}

\subsubsection{Properties of the Multi-agent model}

Given that the multi-agent model is an extension of the two-agent model, we remind the reader that the accuracy of the update laws $K_{\tilde{o}}$ and $K_{a}$ given in (\ref{eqn:offset_law2}) and (\ref{eqn:skew_law2}) depend on the assigned values of the memory buffers, namely $\m^R$ and $\m^i$ for the multi-agent model. Therefore, to facilitate the analysis of $\widetilde{\HS}_{\varepsilon}$ in rendering the set $\A_{\varepsilon}$ asymptotically attractive, we define a new set modified from the set $\mathcal{M}$ given in (\ref{set:M}) that gives the state space of values of $\m^R$ and $\m^i$ for which the update laws $K_{\tilde{o}}$ and $K_{a}$ give the correct values. In particular, consider the set \begin{equation}
\mathcal{M} := \mathcal{M}_1 \cup \mathcal{M}_2 \cup \mathcal{M}_3 \cup \mathcal{M}_4 \cup \mathcal{M}_5 \cup \mathcal{M}_6
\end{equation}
\noindent
where
\ifbool{conf}{
\begin{equation*} 
\begin{aligned}
\mathcal{M}_1 & := \{ \tilde{x}_{\varepsilon} \in \mathcal{X}_{\varepsilon} : p {=} 0, q {=} 0 \} \\
\mathcal{M}_2 & := \{ \tilde{x}_{\varepsilon} \in \mathcal{X}_{\varepsilon} : p {=} 1, q {=} 1, \m_1^R {\minus} \rho_R(\tilde{x}_{\varepsilon}, 0)  = 0 \} \\
\mathcal{M}_3 & := \{ \tilde{x}_{\varepsilon} \in \mathcal{X}_{\varepsilon} : p {=} 2, q {=} 0, \m_1^i {\minus} \rho_i(\tilde{x}_{\varepsilon}, 0)  = 0, \\ 
& \hspace{7.5mm} \m_2^i {\minus} \rho_R(\tilde{x}_{\varepsilon}, d) = 0 \} \\
\mathcal{M}_4 & := \{ \tilde{x}_{\varepsilon} \in \mathcal{X}_{\varepsilon} : p {=} 3, q {=} 1, \m_1^i {\minus} \rho_i(\tilde{x}_{\varepsilon}, 0) = 0, \\
& \hspace{7.5mm} \m_2^i {\minus} \rho_i(\tilde{x}_{\varepsilon}, d) = 0, \m_3^i {\minus} \rho_R(\tilde{x}_{\varepsilon}, c{+}d) = 0 \} \\
\mathcal{M}_5 & := \{ \tilde{x}_{\varepsilon} \in \mathcal{X}_{\varepsilon} : p {=} 4, q {=} 0, \m_1^R {\minus} \rho_R(\tilde{x}_{\varepsilon}, 0) = 0, \\
& \hspace{7.5mm} \m_2^R {\minus} \rho_i(\tilde{x}_{\varepsilon}, d) = 0, \m_3^R {\minus} \rho_i(\tilde{x}_{\varepsilon},  c{+}d) = 0, \\
& \hspace{7.5mm}  \m_4^R {\minus} \rho_R(\tilde{x}_{\varepsilon}, 2d{+}c) {=} 0 \} \\
\mathcal{M}_6 & := \{ \tilde{x}_{\varepsilon} \in \mathcal{X}_{\varepsilon} : p {=} 5, q {=} 1, \m_1^R {\minus} \rho_R(\tilde{x}_{\varepsilon}, 0)  = 0,  \\
& \hspace{7.5mm} \m_2^R {\minus} \rho_R(\tilde{x}_{\varepsilon}, c) = 0, \m_3^R {-} \rho_i(\tilde{x}_{\varepsilon},c{+}d) = 0, \\
& \hspace{7.5mm}  \m_4^R {\minus} \rho_i(\tilde{x}_{\varepsilon},2c{+}d) = 0, \m_5^R {\minus} \rho_R(\tilde{x}_{\varepsilon}, 2c{+}2d) = 0 \}
\end{aligned}
\end{equation*}
}{
\begin{equation*} 
\begin{aligned}
\mathcal{M}_1 & := \{ \tilde{x}_{\varepsilon} \in \mathcal{X}_{\varepsilon} : p {=} 0, q {=} 0 \} \\
\mathcal{M}_2 & := \{ \tilde{x}_{\varepsilon} \in \mathcal{X}_{\varepsilon} : p {=} 1, q {=} 1, \m_1^R {\minus} \rho_R(\tilde{x}_{\varepsilon}, 0)  = 0 \} \\
\mathcal{M}_3 & := \{ \tilde{x}_{\varepsilon} \in \mathcal{X}_{\varepsilon} : p {=} 2, q {=} 0, \m_1^i {\minus} \rho_i(\tilde{x}_{\varepsilon}, 0)  = 0, \\ 
& \hspace{7.5mm} \m_2^i {\minus} \rho_R(\tilde{x}_{\varepsilon}, d) = 0 \} \\
\mathcal{M}_4 & := \{ \tilde{x}_{\varepsilon} \in \mathcal{X}_{\varepsilon} : p {=} 3, q {=} 1, \m_1^i {\minus} \rho_i(\tilde{x}_{\varepsilon}, 0) = 0, \\
& \hspace{7.5mm} \m_2^i {\minus} \rho_i(\tilde{x}_{\varepsilon}, d) = 0, \m_3^i {\minus} \rho_R(\tilde{x}_{\varepsilon}, c{+}d) = 0 \} \\
\mathcal{M}_5 & := \{ \tilde{x}_{\varepsilon} \in \mathcal{X}_{\varepsilon} : p {=} 4, q {=} 0, \m_1^R {\minus} \rho_R(\tilde{x}_{\varepsilon}, 0) = 0, \\
& \hspace{7.5mm} \m_2^R {\minus} \rho_i(\tilde{x}_{\varepsilon}, d) = 0, \m_3^R {\minus} \rho_i(\tilde{x}_{\varepsilon},  c{+}d) = 0, \\
& \hspace{7.5mm}  \m_4^R {\minus} \rho_R(\tilde{x}_{\varepsilon}, 2d{+}c) {=} 0 \} \\
\mathcal{M}_6 & := \{ \tilde{x}_{\varepsilon} \in \mathcal{X}_{\varepsilon} : p {=} 5, q {=} 1, \m_1^R {\minus} \rho_R(\tilde{x}_{\varepsilon}, 0)  = 0,  \\
& \hspace{7.5mm} \m_2^R {\minus} \rho_R(\tilde{x}_{\varepsilon}, c) = 0, \m_3^R {-} \rho_i(\tilde{x}_{\varepsilon},c{+}d) = 0, \\
& \hspace{7.5mm}  \m_4^R {\minus} \rho_i(\tilde{x}_{\varepsilon},2c{+}d) = 0, \m_5^R {\minus} \rho_R(\tilde{x}_{\varepsilon}, 2c{+}2d) = 0 \}
\end{aligned}
\end{equation*}
}
\noindent
and
\begin{equation}
\begin{aligned}
\rho_R(\tilde{x}_{\varepsilon}, \beta) & = \Big (\tau_{R} - a_{R} \big (\big ((1{-}q)c {+} qd \big ) - \tau \big ) \Big ) {-} a_{R} (\beta) \\
\rho_i(\tilde{x}_{\varepsilon}, \beta) & = \Big (\tau_{i} - a_{i} \big (((1{-}q)c {+} qd)- \tau \big ) \Big )  {-} a_{i} (\beta)
\end{aligned}
\end{equation}
\noindent
for some $\beta > 0$, i.e., $\beta = 0, c + d, 2c+d, 2c + 2d$.

\begin{lemma} \label{lem:fwd_inv_M_2}
The set $\mathcal{M}$ is forward invariant for the hybrid system $\widetilde{\HS}_{\varepsilon}$.
\end{lemma}

\begin{proof}
This result follows from the proof of Lemma \ref{lem:fwd_inv_M} since the variables $\tau$, $p$, $q$, $\m^R$, and $\m^i$ have identical dynamics to the corresponding variables in the two-agent model.
\end{proof}

\begin{lemma} \label{lem:finite_M_2}
Let constants $d \geq c > 0$ be given. For each maximal solution $\phi$ to $\widetilde{\HS}_{\varepsilon}$, there exists a $T^* \geq 0$ such that for any $(t,j) \in \mbox{dom } \phi$ with $t + j \geq T^*$, $\phi(t,j) \in \mathcal{M}$.
\end{lemma}

\ifbool{conf}{}{
\ifbool{conf}{\begin{pf}}{
\begin{proof}}
Pick a solution $\phi$ with initial condition $\phi(0,0) \in  C_{\varepsilon} \cup \widetilde{D}_{\varepsilon}$. From the initial condition, the solution flows and jumps according to the dynamics of $\HS$. Observe that when $\phi \in C_{\varepsilon}$, $\dot{p} = 0$ and when $\phi \in \widetilde{D}_{\varepsilon}$, $p^+ = p + 1$ following jumps according to $G^i_{\varepsilon_{\ell}}(\tilde{x}_{\varepsilon})$, $\ell \in \{1,2,3,4,5\}$. However, for jumps according to $G^i_{\varepsilon_{6}}(\tilde{x}_{\varepsilon})$, $p^+ = 0$ then by also noting that $q^+ = 0$, $G^i_{\varepsilon_{6}}(\tilde{x}_{\varepsilon}) \subset \mathcal{M}_6$. Therefore, due to the monotonic behavior of $p$ following jumps according to $G^i_{\varepsilon_{\ell}}(\tilde{x}_{\varepsilon})$, $\ell \in \{1,2,3,4,5\}$, at $t + j \geq T^*$ the solution jumps according to $G^i_{\varepsilon_6}(\tilde{x}_{\varepsilon})$, then following the notions on forward invariance for $\mathcal{M}$ in Lemma \ref{lem:fwd_inv_M_2}, the solution remains in $\mathcal{M}$ as $t + j \to \infty$. 
\ifbool{conf}{\hfill $\square$ \end{pf}}{
\end{proof}}}

\subsubsection{Results}

To show attractivity of $\widetilde{\A}_{\varepsilon}$ in (\ref{set:A_multi}) for $\widetilde{\HS}_{\varepsilon}$ in (\ref{eqn:Hy2_multi}) , we utilize a similar Lyapunov-based approach to the two-node model as presented in (\ref{eqn:Hy}). Define the function $\widetilde{V} : \widetilde{\mathcal{X}}_{\varepsilon} \to \reals_{\geq 0}$ as
\begin{equation} \label{eqn:lyap_fun_2}
\begin{aligned}
\widetilde{V}(\tilde{x}_{\varepsilon}) & = \varepsilon_1^\top \exp(A_f^\top \tilde{\tau}_{1}) P \exp(A_f \tilde{\tau}_{1}) \varepsilon_1 + \ldots + \varepsilon_i^\top \exp(A_f^\top \tilde{\tau}_{i}) P \exp(A_f \tilde{\tau}_{i}) \varepsilon_i + \ldots \\
& \hspace{5mm} + \varepsilon_{N}^\top \exp(A_f^\top \tilde{\tau}_{N}) P \exp(A_f \tilde{\tau}_{N}) \varepsilon_{N} \\
& = \varepsilon^\top Q(\tilde{\tau})^\top \mathcal{P} \hspace{0.5mm} Q(\tilde{\tau}) \varepsilon 
\end{aligned}
\end{equation}
\noindent
defined for each $\tilde{x}_{\varepsilon} \in C_{\varepsilon} \cup \widetilde{D}_{\varepsilon}$ where $Q(\tau_\HS) = \mbox{diag} \big ( \exp(A_F \tau_{\HS_1}), \ldots, \exp(A_F \tau_{\HS_{N}}) \big )$ and $\mathcal{P} = \mbox{diag} \big ( P, \ldots, P \big )$  where $P \succ 0$. Furthermore, note the existence of two positive scalars $\alpha_1$ and $\alpha_2$ such that
\begin{equation}
\begin{aligned}
\alpha_1|\tilde{x}_{\varepsilon}|^2_{\widetilde{\A}_{\varepsilon}} \leq V(\tilde{x}_{\varepsilon}) \leq \alpha_2|\tilde{x}_{\varepsilon}|^2_{\widetilde{\A}_{\varepsilon}}
\end{aligned}
\end{equation}
\noindent
The function $\widetilde{V}$ satisfies the following infinitesimal properties.

\begin{lemma} \label{lem:V2_flows}
Let the hybrid system $\widetilde{\HS}_{\varepsilon}$ with constants $d \geq c > 0$ be given. For each $\tilde{x}_{\varepsilon} \in \widetilde{C}$, $$\langle \nabla V(\tilde{x}_{\varepsilon}), \widetilde{F}_{\varepsilon} (\tilde{x}_{\varepsilon}) \rangle = 0$$
\end{lemma}

\begin{proof}
This result follows through direct calculation of $\langle \nabla V(\tilde{x}_{\varepsilon}), \widetilde{F}_{\varepsilon} (\tilde{x}_{\varepsilon}) \rangle$.
\end{proof}

\begin{lemma} \label{lem:V2_jumps}
Let the hybrid system $\widetilde{\HS}_{\varepsilon}$ with constants $d \geq c > 0$ be given. If there exist a constant $\mu > 0$ and a positive definite symmetric matrix  $P$  such that 
\begin{equation} \label{eqn:lyap_cond_2}
\begin{aligned}
A_g^{\top} \exp \big ( (3c + 3d) A_f^{\top} \big ) P \exp \big ( (3c + 3d) A_f \big ) A_g  -  P \prec 0
\end{aligned}
\end{equation}
is satisfied where $A_{g} = \begin{bmatrix} 0 & \gamma_1 \\ 0 &  1 {\minus} \mu \gamma_2 \end{bmatrix}$ with $\gamma_1 = \frac{3c + 4d}{2}$ and $\gamma_2 = 2c + 2d$, then for each $\tilde{x}_{\varepsilon} \in \mathcal{M} \cap \widetilde{D}_{\varepsilon}$
\begin{equation}
\begin{aligned}
V(g) - V(\tilde{x}_{\varepsilon}) \leq \begin{cases} 0 & \mbox{\rm if } g \in \widetilde{G}_{\varepsilon}(\tilde{x}_{\varepsilon}) \setminus \widetilde{G}_6^i(\tilde{x}_{\varepsilon} ), i \in \nodes \\
\minus \sigma \varepsilon_i^{\top} \varepsilon_i & \mbox{\rm if } g \in \widetilde{G}_6^i(\tilde{x}_{\varepsilon} ), i \in \nodes
\end{cases}
\end{aligned}
\end{equation}
\noindent
where  $$\sigma \in \bigg ( 0, \lambda_{\rm min} \Big ( A_g^{\top} \exp \big ((3c + 3d) A_f^{\top} \big ) P \exp \big ((3c + 3d) A_f \big ) A_g  -  P \prec 0 \Big ) \bigg )$$ 
\end{lemma}

\begin{proof}
Consider the Lyapunov function candidate $\widetilde{V} : \widetilde{\mathcal{X}}_{\varepsilon} \to \reals_{\geq 0}$ in (\ref{eqn:lyap_fun_2}). For each $\tilde{x}_{\varepsilon} \in \widetilde{D}_{\varepsilon} \setminus \widetilde{D}_{\varepsilon_6}^i$, $i \in \nodes$ and each $g \in \widetilde{G}_{\varepsilon}(\tilde{x}_{\varepsilon}) \setminus \widetilde{G}^i_6(\tilde{x}_{\varepsilon})$, $i \in \nodes$, we have that $V(g) - V(\tilde{x}_{\varepsilon}) = 0$. Now for each  $\tilde{x}_{\varepsilon} \in \widetilde{D}_{\varepsilon_6}^i \cap \mathcal{M}_6$, $i \in \nodes$ we have that $\tilde{\tau}_{i} = 0$ then for each $g \in \widetilde{G}^i_6(\tilde{x}_{\varepsilon})$, $i \in \nodes$, 
\begin{align*}
V(\widetilde{G}^i_{\varepsilon_6}(\tilde{x}_{\varepsilon})) - V(\tilde{x}_{\varepsilon}) & = \\
& \hspace{-3cm} \begin{bmatrix} \varepsilon_i \mplus \begin{bmatrix}
\minus K_{\tilde{o}}(\tilde{x}_{\varepsilon}) \\ K_{a}(\tilde{x}_{\varepsilon})
\end{bmatrix} \end{bmatrix}^{\top} \exp \big ((3c + 3d) A_f^{\top} \big ) P \exp \big ( (3c + 3d) A_f \big ) \begin{bmatrix} \varepsilon_i \mplus \begin{bmatrix} \minus K_{\tilde{o}}(\tilde{x}_{\varepsilon}) \\ K_{a}(\tilde{x}_{\varepsilon})
\end{bmatrix} \end{bmatrix} \\ 
& \hspace{1cm} - \varepsilon_i^{\top} \exp \big ( (0) A_f^{\top} \big ) P \big ( (0) A_f^{\top} \big )  \varepsilon_i \\
& \hspace{-3cm} = \begin{bmatrix} \varepsilon_i {+} \begin{bmatrix}
\minus K_{\tilde{o}}(\tilde{x}_{\varepsilon}) \\ K_{a}(\tilde{x}_{\varepsilon})
\end{bmatrix} \end{bmatrix}^{\top} \exp \big ( (3c + 3d) A_f^{\top} \big ) P \exp \big ( (3c + 3d) A_f \big ) \begin{bmatrix} \varepsilon_i {+} \begin{bmatrix} \minus K_{\tilde{o}}(\tilde{x}_{\varepsilon}) \\ K_{a}(\tilde{x}_{\varepsilon})
\end{bmatrix} \end{bmatrix} \\ 
& \hspace{1cm} - \varepsilon_i^{\top} P  \varepsilon_i
\end{align*}
\noindent
Recall from the proof of Lemma \ref{lem:V_jumps}, that for jumps with reset $g = \widetilde{G}_{\varepsilon_6}^i(x)$ for $x_{\varepsilon} \in \mathcal{M}$, the corrections $K_{\tilde{o}}(p, \m)$, $K_{a}(\m, \tau_{k})$ applied to $\tau_{i}$, $a_{i}$, respectively, give $
K_{\tilde{o}}(\tilde{x}_{\varepsilon}) = - \varepsilon_{\tau} + \gamma_1  \varepsilon_a$ and $K_{a}(\tilde{x}_{\varepsilon}) = \mu \gamma_2 \varepsilon_a$ where $\gamma_1 = \frac{(3c + 4d)}{2}$ and $\gamma_2 = 2c + 2d$.  One then has,
\begin{align*}
\hspace{-5mm} V(\widetilde{G}^i_{\varepsilon_6}(\tilde{x}_{\varepsilon})) - V(\tilde{x}_{\varepsilon}) & = \begin{bmatrix} \varepsilon_i + \begin{bmatrix}
\minus K_{\tilde{o}}(\tilde{x}_{\varepsilon}) \\ K_{a}(\tilde{x}_{\varepsilon})
\end{bmatrix} \end{bmatrix}^{\top} \exp \big ( (3c + 3d) A_f^{\top} \big ) P \exp \big ( (3c + 3d) A_f \big ) \begin{bmatrix} \varepsilon_i + \begin{bmatrix} \minus K_{\tilde{o}}(\tilde{x}_{\varepsilon}) \\ K_{a}(\tilde{x}_{\varepsilon})
\end{bmatrix} \end{bmatrix} \\
& \hspace{5mm} - \varepsilon_i^{\top} P  \varepsilon_i \\
& \hspace{-25mm} = \begin{bmatrix} \varepsilon_{\tau_i} - \varepsilon_{\tau_i} + \gamma_1  \varepsilon_{a_i} \\ \varepsilon_{a_i} - \mu \gamma_2 \varepsilon_{a_i} \end{bmatrix}^{\top} \exp \big ( (3c + 3d) A_f^{\top} \big ) P \exp \big ( (3c + 3d) A_f \big ) \begin{bmatrix} \varepsilon_{\tau_i} - \varepsilon_{\tau_i} + \gamma_1  \varepsilon_{a_i} \\ \varepsilon_{a_i} - \mu \gamma_2 \varepsilon_{a_i} \end{bmatrix} \\
& \hspace{5mm} - \varepsilon_i^{\top} P \varepsilon_i \\
& \hspace{-25mm} = \begin{bmatrix} I \varepsilon_i + \begin{bmatrix}
\minus 1 & \gamma_1 \\ 0 & \minus \mu \gamma_2
\end{bmatrix} \varepsilon_i \end{bmatrix}^{\top} \exp \big ( (3c + 3d) A_f^{\top} \big ) P \exp \big ( (3c + 3d) A_f \big ) \begin{bmatrix} I \varepsilon_i + \begin{bmatrix}
\minus 1 & \gamma_1 \\ 0 & \minus \mu \gamma_2
\end{bmatrix} \varepsilon_i \end{bmatrix} \\
& \hspace{5mm} - \varepsilon_i^{\top} P  \varepsilon_i \\
& \hspace{-25mm} = \begin{bmatrix} \bigg ( I + \begin{bmatrix}
\minus 1 & \gamma_1 \\ 0 & \minus \mu \gamma_2
\end{bmatrix} \bigg ) \varepsilon_i \end{bmatrix}^{\top} \exp \big ( (3c + 3d) A_f^{\top} \big ) P \exp \big ( (3c + 3d) A_f \big ) \begin{bmatrix} \bigg ( I + \begin{bmatrix}
\minus 1 & \gamma_1 \\ 0 & \minus \mu \gamma_2
\end{bmatrix} \bigg ) \varepsilon_i \end{bmatrix} \\
& \hspace{5mm} - \varepsilon_i^{\top} P \varepsilon_i \\
& \hspace{-25mm} = \varepsilon_i \begin{bmatrix}
0 & \gamma_1 \\ 0 & 1 \minus \mu \gamma_2 \end{bmatrix}^{\top} \exp \big ( (3c + 3d) A_f^{\top} \big ) P \exp \big ( (3c + 3d) A_f \big ) \begin{bmatrix}
0 & \gamma_1 \\ 0 & 1 \minus \mu \gamma_2 \end{bmatrix} \varepsilon_i \\
& \hspace{5mm} - \varepsilon_i^{\top} P  \varepsilon \\
& \hspace{-25mm} = \varepsilon_i^{\top} A_g^{\top} \exp \big ( (3c + 3d) A_f^{\top} \big ) P \exp \big ( (3c + 3d) A_f \big ) A_g  \varepsilon_i - \varepsilon_i^{\top} P  \varepsilon_i \\
& \hspace{-25mm} = \varepsilon_i^{\top} \Big ( A_g^{\top} \exp \big ( (3c + 3d) A_f^{\top} \big ) P \exp \big ( (3c + 3d) A_f \big ) A_g  -  P \Big )  \varepsilon_i
\end{align*}
\noindent
where $A_g = \begin{bmatrix} 0 & \gamma_1 \\ 0 &  1 - \mu \gamma_2 \end{bmatrix}$. Then, by continuity of condition (\ref{eqn:lyap_cond_2}),
\begin{equation}
V(\widetilde{G}^i_{\varepsilon_6}) - V(\tilde{x}_{\varepsilon}) \leq -\sigma \varepsilon_i^\top \varepsilon_i
\end{equation}
where $$\sigma \in \bigg ( 0, \lambda_{\rm min} \Big ( A_g^{\top} \exp \big ((3c + 3d) A_f^{\top} \big ) P \exp \big ((3c + 3d) A_f \big ) A_g  -  P \prec 0 \Big ) \bigg )$$
\end{proof}

\begin{theorem}
Let the hybrid system $\widetilde{\HS}_{\varepsilon}$ with constants $d \geq c > 0$ be given. If there exist a constant $\mu > 0$ and positive definite symmetric matrix $P$ such that 
\begin{equation} \label{eqn:lyap_cond_2}
\begin{aligned}
A_g^{\top} \exp \big ( (3c + 3d) A_f^{\top} \big ) P \exp \big ( (3c + 3d) A_f \big ) A_g  -  P \prec 0
\end{aligned}
\end{equation}
is satisfied where $A_{g} = \begin{bmatrix} 0 & \gamma_1 \\ 0 &  1 {\minus} \mu \gamma_2 \end{bmatrix}$ with $\gamma_1 = \frac{3c + 4d}{2}$ and $\gamma_2 = 2c + 2d$, then $\widetilde{\A}_{\varepsilon}$ is globally attractive for $\widetilde{\HS}_{\varepsilon}$. 
\end{theorem}

\begin{proof}
Pick a maximal solution $\phi \in \mathcal{S}_{\widetilde{\HS}_{\varepsilon}}$ with initial condition $\phi \in \big ( \widetilde{C}_{\varepsilon} \cup \widetilde{D}_{\varepsilon} \big ) \cap \mathcal{M}$ and consider the function $V$ in (\ref{eqn:lyap_fun_2}). Recall the dynamics of $V$ from Lemmas  \ref{lem:V2_flows} and \ref{lem:V2_jumps}  for $\tilde{x}_{\varepsilon} \in \widetilde{C}_{\varepsilon} \cup \widetilde{D}_{\varepsilon}$,
\begin{align*}
\langle \nabla \widetilde{V}(\tilde{x}_{\varepsilon}), \widetilde{F}_{\varepsilon} (\tilde{x}_{\varepsilon}) \rangle = 0
\end{align*}
\noindent
and
\begin{align*}
\widetilde{V}(g) - \widetilde{V}(\tilde{x}_{\varepsilon}) \leq \begin{cases} 0 & \mbox{\rm if } g \in \widetilde{G}_{\varepsilon}(\tilde{x}_{\varepsilon}) \setminus \widetilde{G}_6^i(\tilde{x}_{\varepsilon} ), i \in \nodes \\
\minus \sigma_i \varepsilon_i^{\top} \varepsilon_i & \mbox{\rm if } g = \widetilde{G}_6^i(\tilde{x}_{\varepsilon} ), i \in \nodes
\end{cases}
\end{align*}
\noindent
Observe that $V$ remains constant during flows and following any jump $g \in \widetilde{G}_{\varepsilon}(\tilde{x}_{\varepsilon}) \setminus \widetilde{G}_6^i(\tilde{x}_{\varepsilon} ), i \in \nodes$. It is only following jumps according to $g = \widetilde{G}_6^i(\tilde{x}_{\varepsilon} ), i \in \nodes$ that $V$ observes a decrease. 

Now, by noting the dynamics of the variable $s$, namely that it increments by one following each jump $g = \widetilde{G}_6^i(\tilde{x}_{\varepsilon} ), i \in \nodes \setminus \{n \minus 1 \}$ and resets to one following a jump according to $g = \widetilde{G}_6^{n \minus 1}(\tilde{x}_{\varepsilon} )$, we have that the trajectory of $\phi_{s}$ is cyclic. Moreover, we have that the algorithm iterates through each node $i \in \nodes$ applying the corrections $K_{\tilde{o}}(p, \m)$, $K_{a}(\m, \tau_{k})$ to $\varepsilon_i$. Therefore, due to the deterministic nature of the timers $\tau^*$ and $\tilde{\tau}$ that govern the flow and jumps of the system, there exists a hybrid time $(t,j)$ such that $t + j \geq T^* \geq 0$ where the algorithm has applied the corrections $K_{\tilde{o}}(p, \m)$, $K_{a}(\m, \tau_{k})$ to $\varepsilon_i$ for each $i \in \nodes$. Then we have that 
\begin{align*}
V \big ( \phi (t,j) \big ) = - \sum_{i = 1}^{n \minus 1} \sigma_i \varepsilon_i^{\top}(t^i,j^i) \varepsilon_i (t^i,j^i) \hspace{5mm} \forall (t,j) \in \{ \mbox{dom } \phi : t + j \geq T^* \}
\end{align*}
\noindent
where $(t^i,j^i) \in \{ \mbox{dom } \phi : \phi(t,j) \in \widetilde{D}_{\varepsilon_6}^i, i \in \nodes \}$ denotes the time at which the corrections $K_{\tilde{o}}(p, \m)$, $K_{a}(\m, \tau_{k})$ are applied to $\varepsilon_{i}$. Then, taking the limit of $V \big ( \phi (t,j) \big )$ we have that
\begin{align*}
\lim_{t + j \to \infty} V \big ( \phi (t,j) \big ) = 0
\end{align*}
\noindent
allowing us to conclude attractivity to the set $\widetilde{\A}_{\varepsilon}$ for $\widetilde{\HS}_{\varepsilon}$ and complete the proof.
\end{proof}

\section{Numerical Results} \label{sec:num}

\subsection{Two-agent system}

\subsubsection{Nominal Setting}

\ifbool{conf}{}{In this first example, we present a numerical simulation of the two-agent system for the nominal setting that validates our theoretical results, namely we show that with the conditions in (\ref{eqn:lyap_cond}) satisfied, the trajectories of the simulation converge to the desired set.}

\ifbool{conf}{
Consider Nodes $i$ and $k$ with dynamics as in (\ref{eqn:clocks_dyn}) with data $a_i = 1$, $a_k = 0.8$ and $c = d=0.5$ to the system $\HS$. Setting $\mu = 0.25$, condition (\ref{eqn:lyap_cond}) is satisfied with $P = \begin{bmatrix}
5.429 & -0.134 \\ -0.134 & 35.010
\end{bmatrix}$. Simulating the system, Figure \ref{fig:num_ex1} shows the trajectories of the error in the clocks and error in the clock rates of Nodes $i$ and $k$ along with a plot of  $V$ evaluated along the solution. Notice, that $V$ converges to zero asymptotically following several periodic executions of the algorithm.  \footnote{Code at github.com/HybridSystemsLab/HybridSenRecClockSync}
}{
\begin{example}
Consider Nodes $i$ and $k$ with dynamics as in (\ref{eqn:clocks_dyn}) with data $a_i = 1$, $a_k = 1.8$ and $c = 0.1$, $d=0.2$ to the system $\HS$. Setting $\mu = 0.833$, condition (\ref{eqn:lyap_cond}) is satisfied with $P = \begin{bmatrix}
6.2594 &   -0.5219 \\
   -0.5219 &  11.4302
\end{bmatrix}$. Simulating the system, Figure \ref{fig:num_ex1} shows the trajectories of the error in the clocks and error in the clock rates of Nodes $i$ and $k$  for a solution $\phi$ to the system such that $\phi(0,0) \in (C \cup D) \cap \mathcal{M}$.  Figure \ref{fig:num_ex1} also shows the plot of  $V$ evaluated along the solution. Notice, that $V$ converges to zero asymptotically following several periodic executions of the algorithm. Observe that the behavior of clock error is more stable than the conventional sender-receiver algorithm simulated in Figure \ref{fig:ex1}.  \footnote{Code at github.com/HybridSystemsLab/HybridSenRecClockSync}
\end{example}}

\begin{figure}[h]
\centering
\vspace{-3mm}
\subfigure[\label{fig:num_ex1a}]{\includegraphics[trim={0mm 0mm 0mm 0mm}, width=0.4\textwidth]{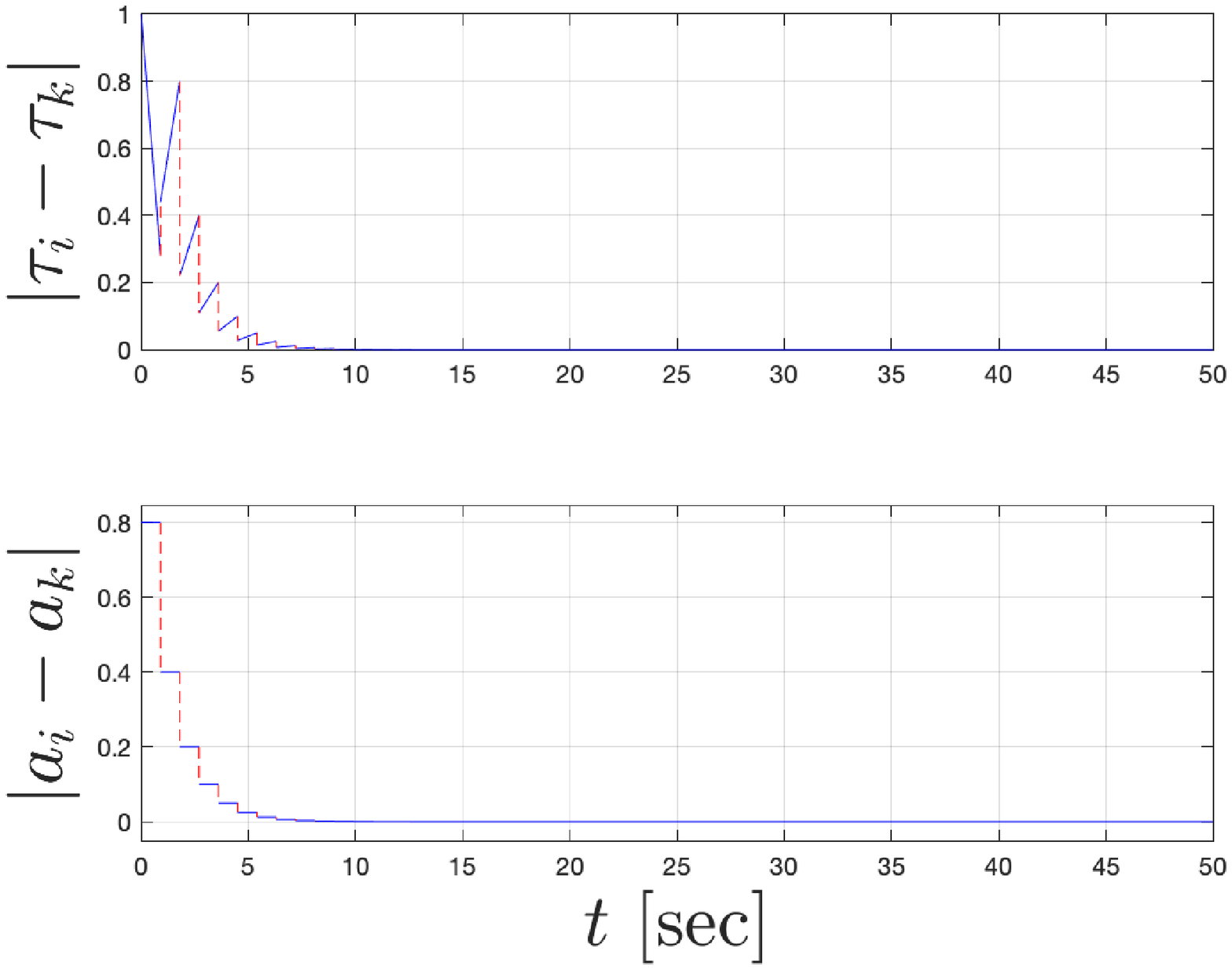}}
\vspace{-3mm}
\subfigure[\label{fig:num_ex1b}]{\includegraphics[trim={0mm 0mm 0mm 10mm},width=0.4\textwidth]{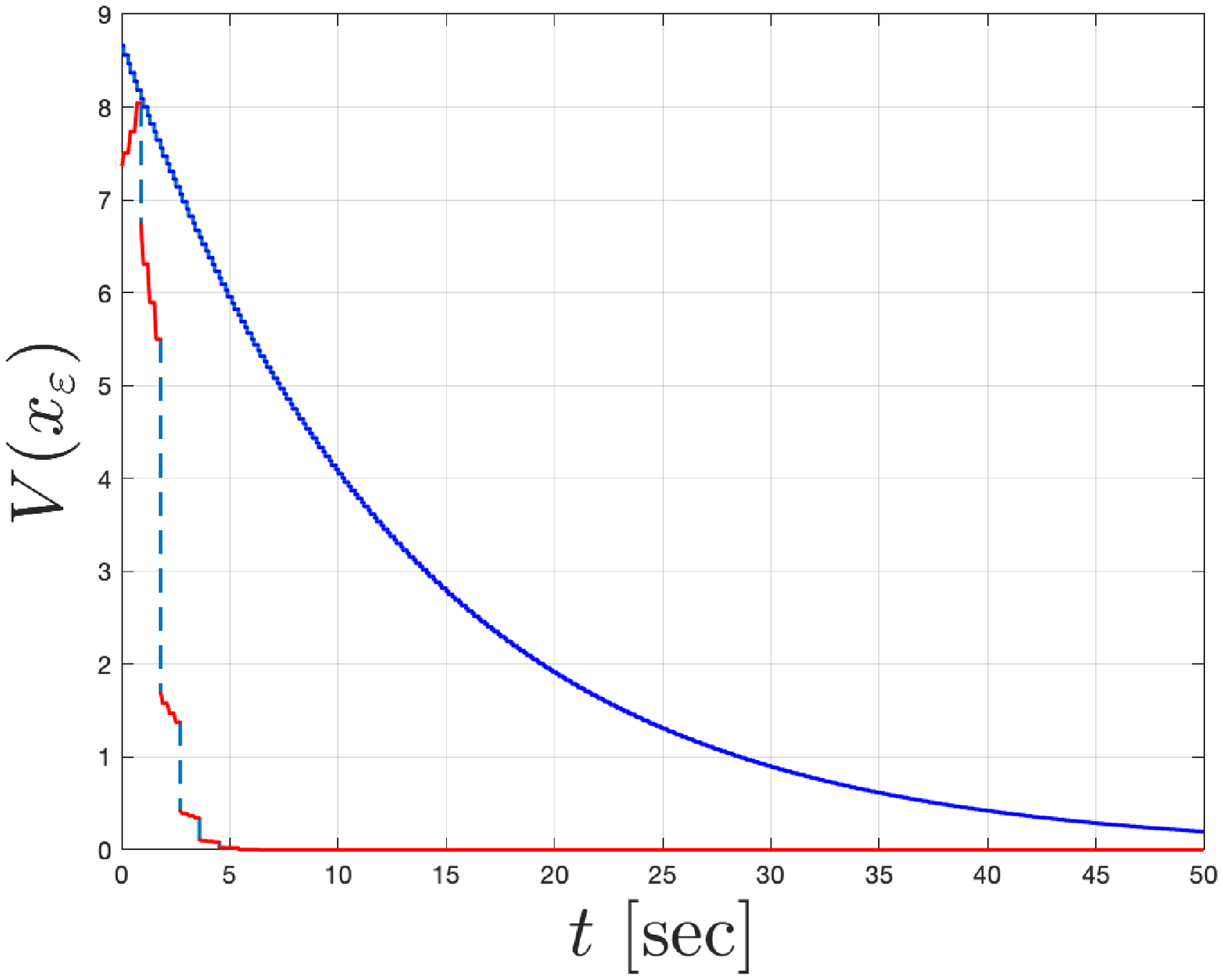}}
\caption{\label{fig:num_ex1} Figure \ref{fig:num_ex1a} gives the evolution of the error in the clocks and clock rates of Nodes $i$ and $k$. Figure \ref{fig:num_ex1b} gives $V$ evaluated along the solution.}
\end{figure}

\subsubsection{Variable propagation delay due to communication noise}

In the next example, we simulate the case of noise in the communication channel that contributes to a variable propagation delay $d$. Noise in the communication channel makes the propagation delay between nodes $i$ and $k$ no longer symmetric.  

\begin{example}
Consider the same clock dynamics from the previous example, i.e., $a_i = 1.1$, $a_k = 0.75$, with  $\mu = 0.3571$ and condition (\ref{eqn:lyap_cond}) satisfied for $c = 0.2$, $d = 0.5$, and $P = \begin{bmatrix}
5.435 & 1.041 \\ 1.041 & 16.0982
\end{bmatrix}$. Now, with $[d_1, d_2]$ defining the allowed values of $d$ with $d_1 = 0.49$ and $d_2 = 0.51$, we generate variable propagation delay by replacing the dynamics of $\tau$ in (\ref{eqn:timer_dyn}) by
\begin{equation*}
\begin{aligned}
& \dot{\tau} = -1 & & \tau \in [0,d_2] \\
& \tau^+ \in \cup_{d \in [d_1,d_2]} (1-q) d + qc & & \tau = 0
\end{aligned}
\end{equation*}
Figure \ref{fig:ex3} shows a simulation of the trajectories of the error in the clocks and error in the clock rates of Nodes $i$ and $k$. Observe that absolute error in the clocks converges to zero even in the presence of the perturbation after several periodic executions of the algorithm. The error in clock rates is also able to converge sufficiently close to zero but suffers from some observed variability due to the noise. \footnote{Code at github.com/HybridSystemsLab/HybridSenRecClockSync}
\end{example}

\begin{figure}[h]
\centering
\vspace{-3mm}
\subfigure[\label{fig:num_ex3a}]{\includegraphics[trim={0mm 0mm 0mm 0mm}, width=0.4\textwidth]{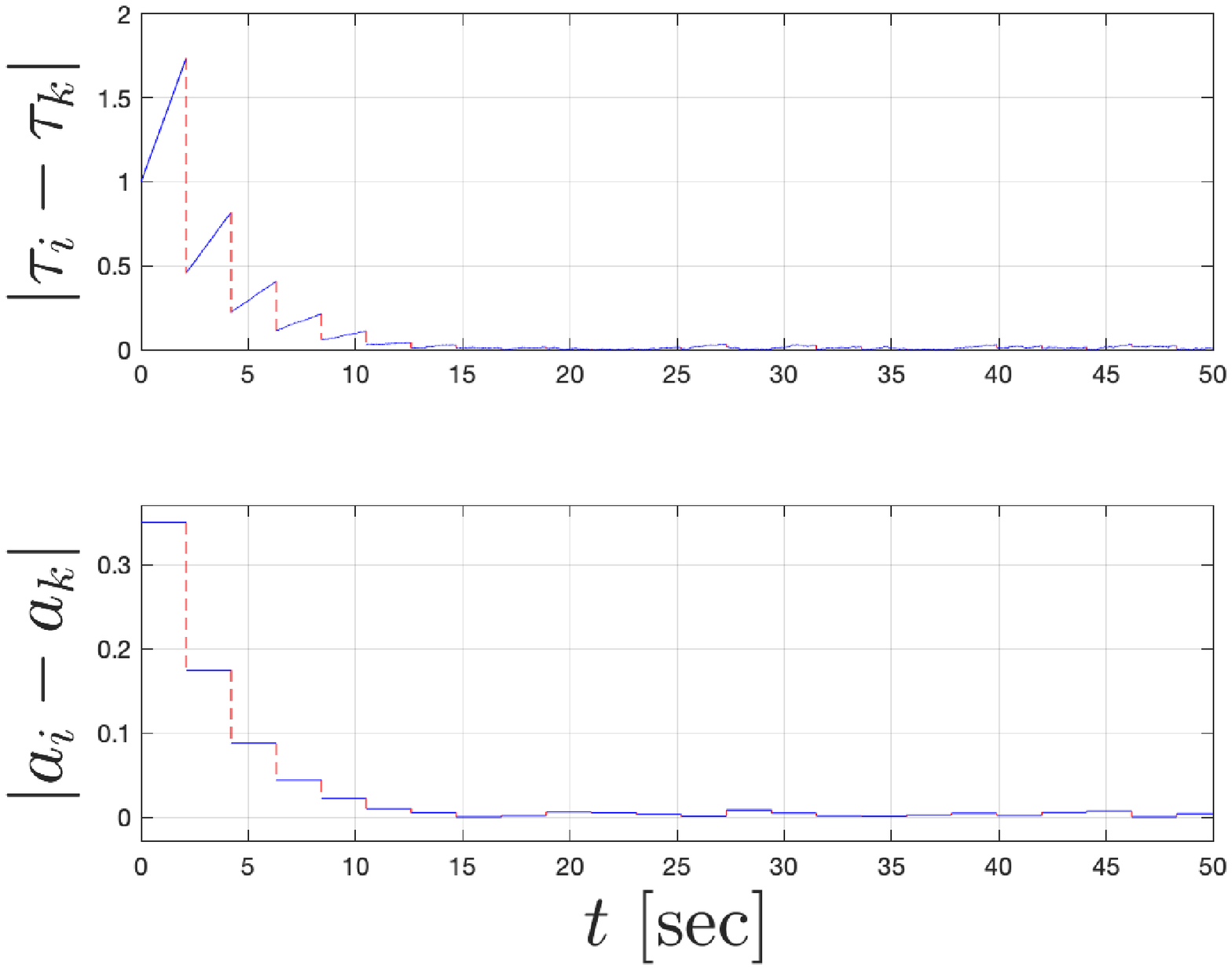}}
\vspace{-3mm}
\subfigure[\label{fig:num_ex3b}]{\includegraphics[trim={0mm 0mm 0mm 10mm},width=0.4\textwidth]{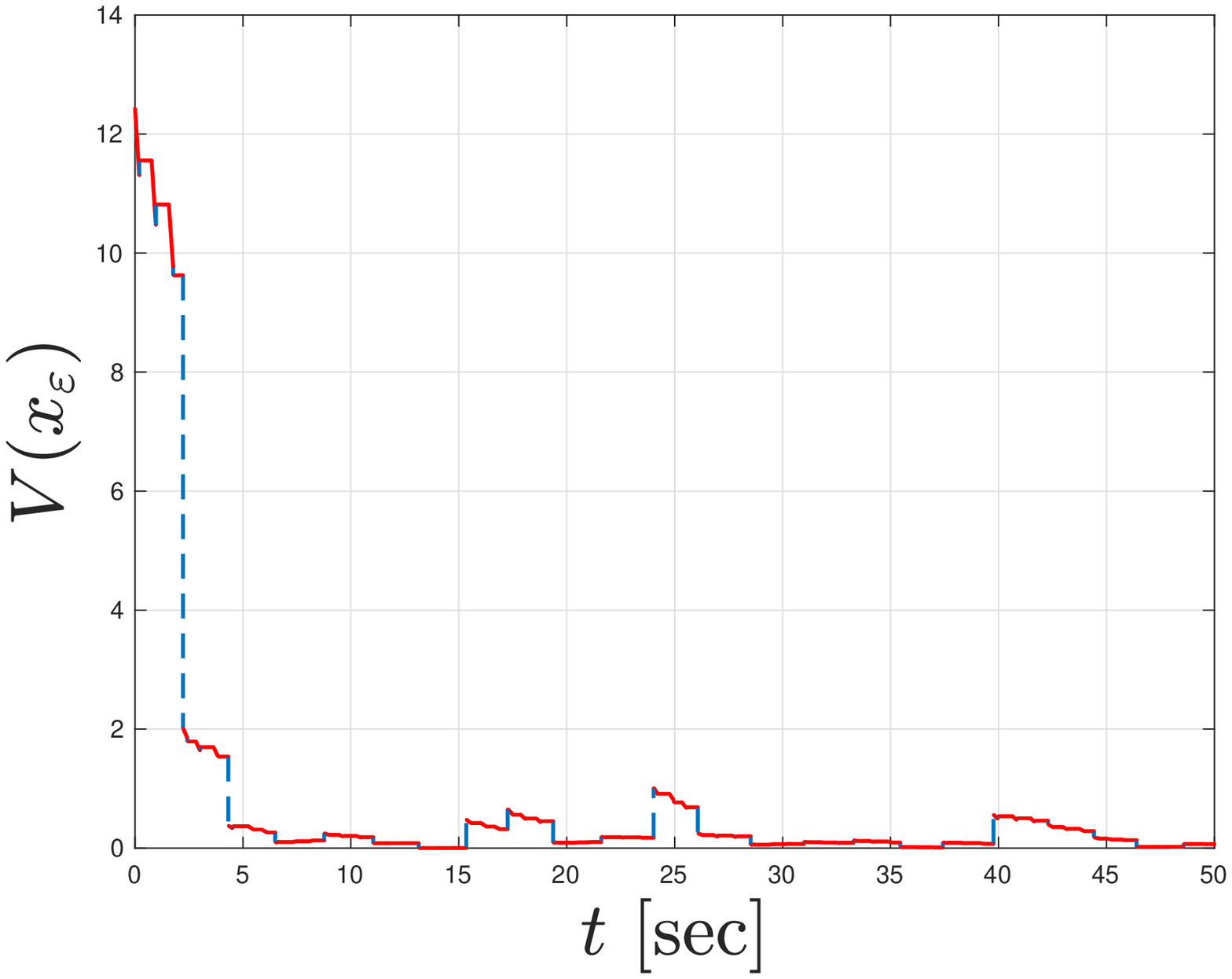}}
\caption{\label{fig:ex3} Figure \ref{fig:num_ex1a} gives the evolution of the error in the clocks and clock rates of Nodes $i$ and $k$ subject to noise on the communication channel. Figure \ref{fig:num_ex1b} gives $V$ evaluated along the solution.}
\end{figure}

\subsubsection{Time-varying clock rates}

In the next example, we consider the common scenario of time-varying clock skews at both nodes $i$ and $k$. This noise is injected at the clock dynamics $\dot{\tau}_i$ and $\dot{\tau}_k$. The system is then simulated with the remaining dynamics left unchanged.

\begin{example}
For $c = 0.2$ and $d = 0.5$, consider nodes $i$ and $k$ with clock dynamics
\begin{align*}
\dot{\tau}_i & = a_i + m_a \\
\dot{\tau}_k & = a_k + m_a 
\end{align*}
where $a_i = 1.1$, $a_k = 0.75$, and $m_a \in (-0.3,0.3)$ is a Gaussian injected noise on the clock dynamics. Letting $\mu = 0.3571$, condition (\ref{eqn:lyap_cond}) is satisfied with $P = \begin{bmatrix}
5.435 & 1.041 \\ 1.041 & 16.0982
\end{bmatrix}$. 
Simulating the system, Figure \ref{fig:ex4} shows the trajectories of the error in the clocks and error in the clock rates of Nodes $i$ and $k$. Again, the system is able to converge after a couple of executions of the algorithm. The error on the clocks observes the most variability due to simulated noise.\footnote{Code at github.com/HybridSystemsLab/HybridSenRecClockSync}
\end{example}

\begin{figure}[h]
\centering
\vspace{-3mm}
\subfigure[\label{fig:num_ex4a}]{\includegraphics[trim={0mm 0mm 0mm 0mm}, width=0.4\textwidth]{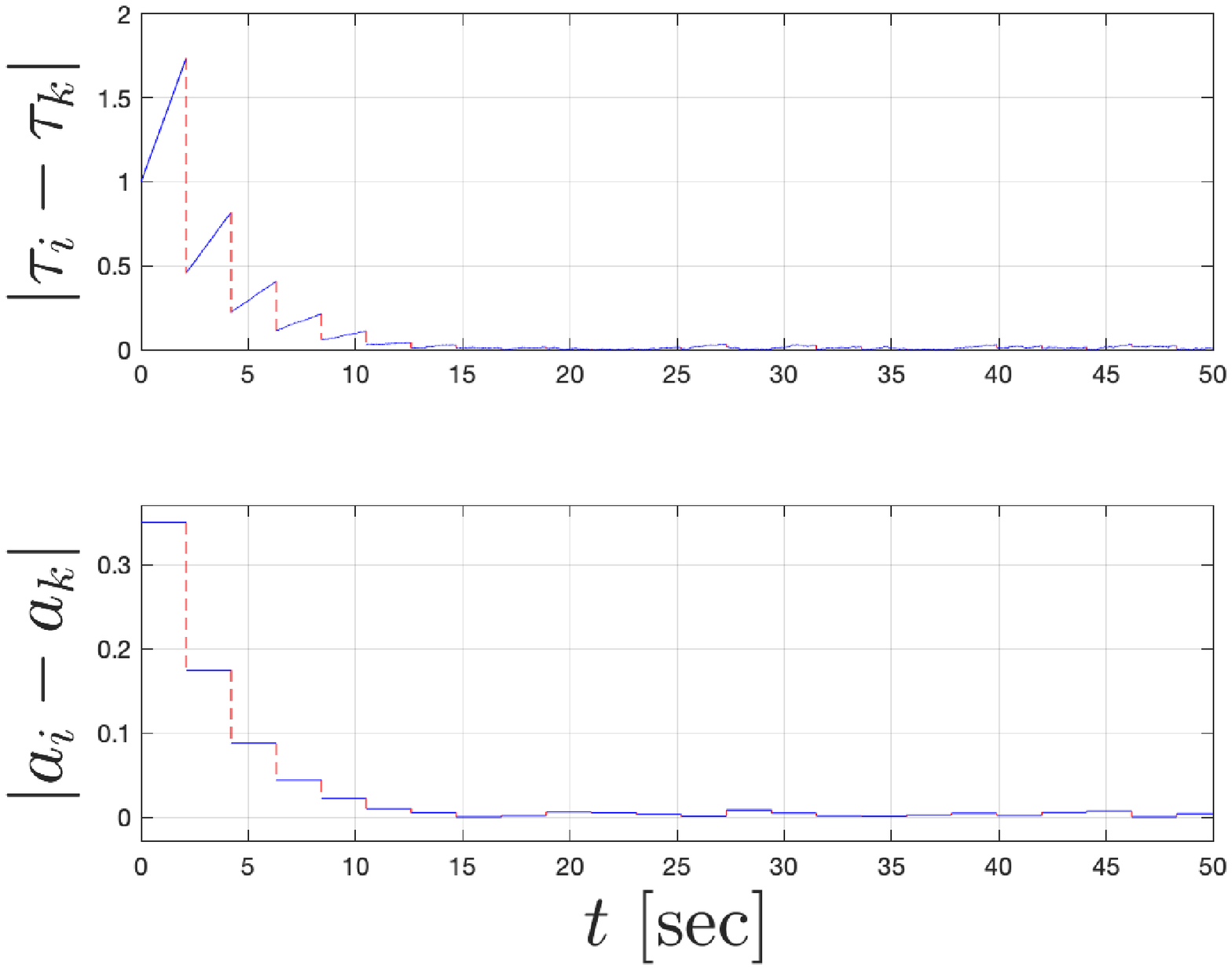}}
\vspace{-3mm}
\subfigure[\label{fig:num_ex4b}]{\includegraphics[trim={0mm 0mm 0mm 10mm},width=0.4\textwidth]{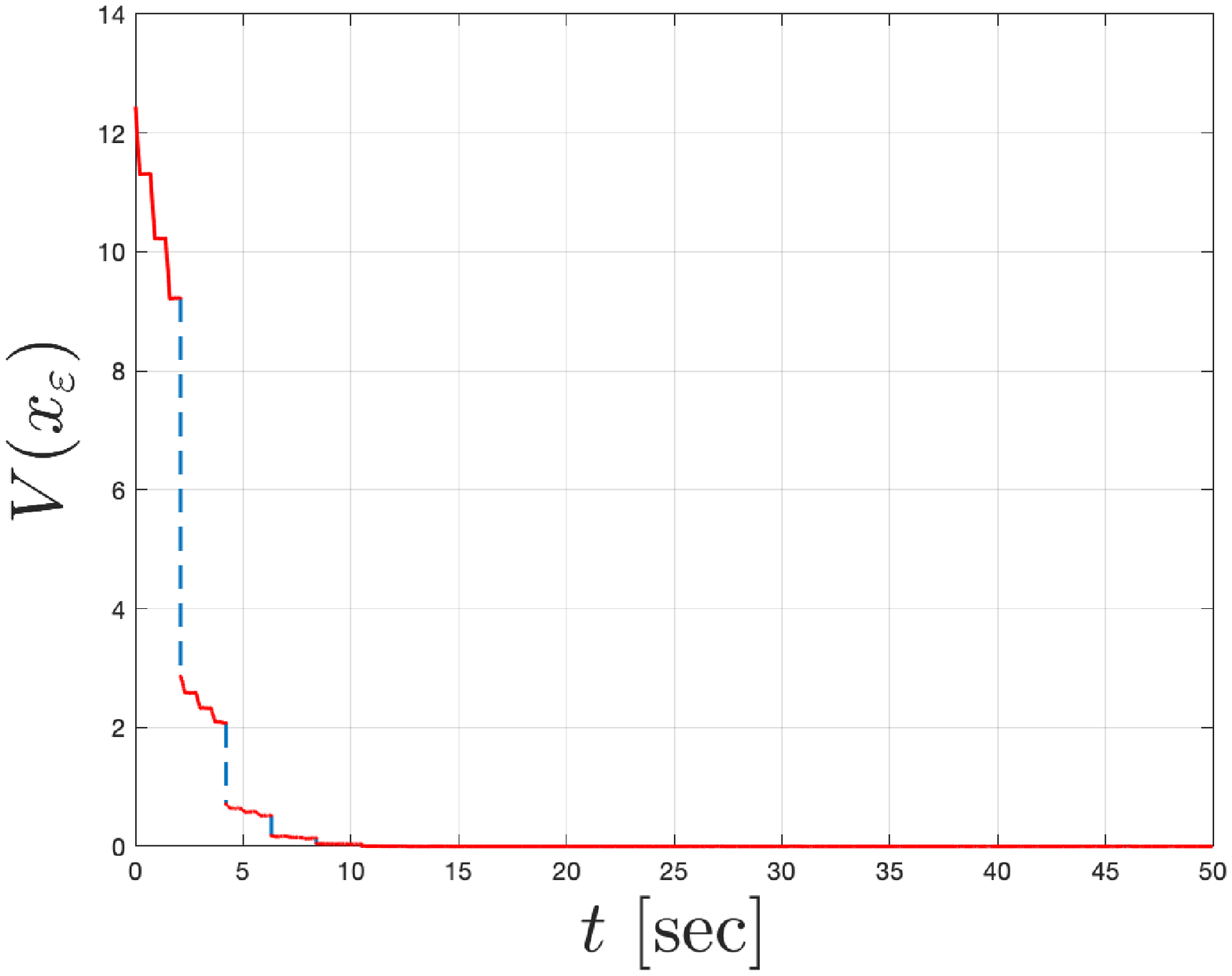}}
\caption{\label{fig:ex4} Figure \ref{fig:num_ex1a} gives the evolution of the error in the clocks and clock rates of Nodes $i$ and $k$ subject to noise $m_a$ on the clock dynamics. Figure \ref{fig:num_ex1b} gives $V$ evaluated along the solution.}
\end{figure}

\vspace{-0mm}

\subsection{Multi-agent model}

In this section we present numerical results for the multi-agent model to validate our theoretical results and draw comparisons with other multi-agent clock synchronization models from the literature.

\begin{example}
Consider a network of three nodes  $\{R, 1, 2 \}$ where $R$ denotes the reference or parent node while nodes $1$ and $2$ denote the synchronizing child nodes. The data of this system is given by $a_R, a_1, a_2 \in [0.5, 1.5 ]$ and  $c = 0.1$, $d = 0.2$ with $\mu = 0.833$. Simulating the multi-agent system $\widetilde{\HS}$, Figure \ref{fig:num_ex4a} shows the trajectories of the error in the clocks and error in the clock rates of Nodes $1$ and $2$ with respect to Node $R$. Note that the errors with respect to each clock converge after several executions of the algorithm on the respective clocks at Nodes $1$ and $2$. \footnote{Code at github.com/HybridSystemsLab/HybridSenRecMultiClockSync}
\end{example}

\begin{figure}[H]
\centering
\includegraphics[width=0.5\textwidth]{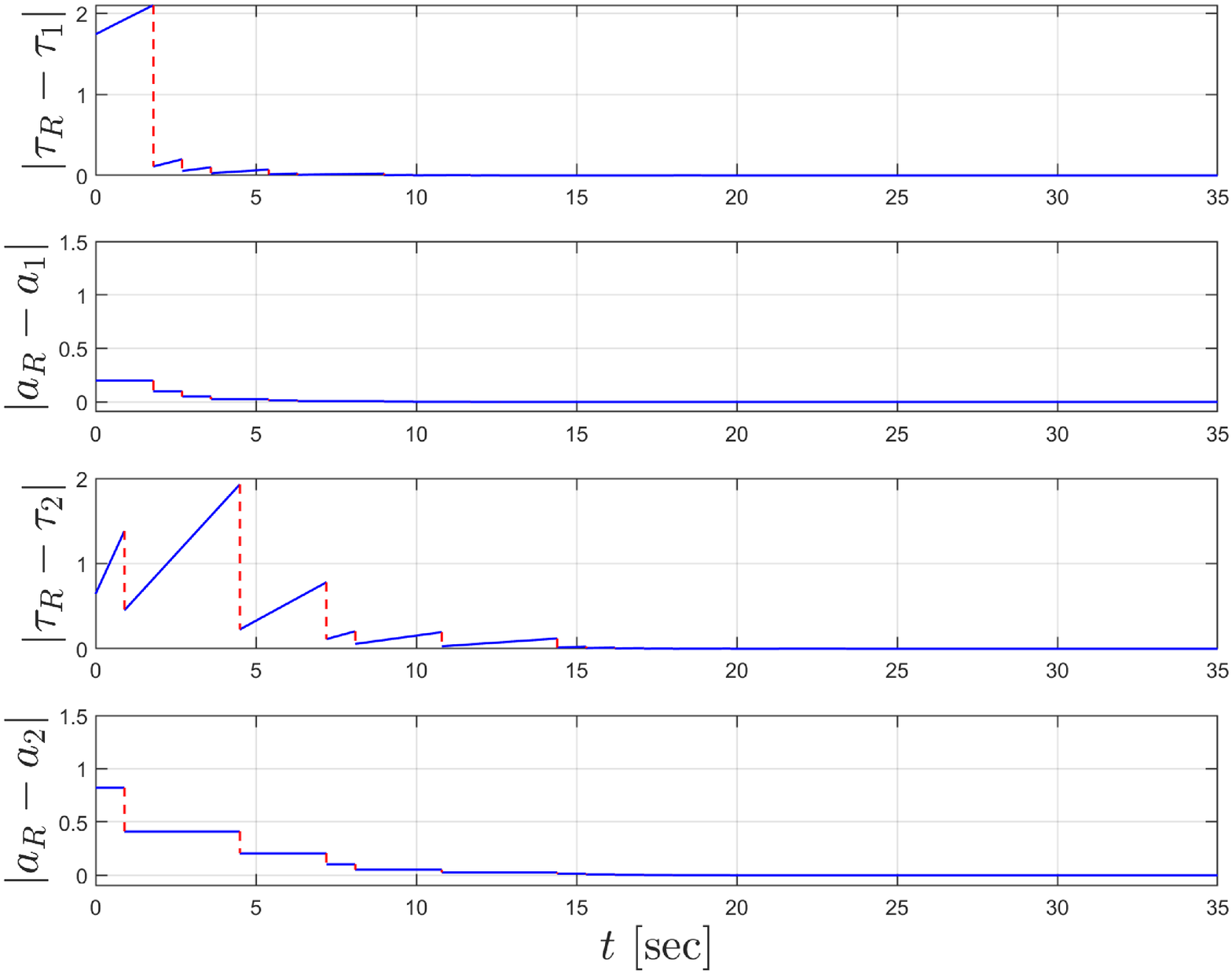}
\caption{\label{fig:ex3}  The evolution of the error in the clocks and clock rates of Nodes $1$ and $2$ with respect to Node $R$.}
\end{figure}


\section{Conclusion}

In this paper, we introduced a sender-receiver clock synchronization algorithm with sufficient design conditions ensuring synchronization. Results were given to show asymptotic attractivity of a set of interest reflecting the desired synchronized setting. Numerical results validating the attractivity of the system to the set of interest were also given. An additional model to capture the multi-agent setting was presented with a numerical example to demonstrate its feasibility. In future work we will study stability of the system and robustness properties to specific perturbations. 

\vspace{-0mm}

\ifbool{conf}{\nocite{guarro}}{
\nocite{*}}
\bibliography{IEEE_PTP}
\bibliographystyle{elsarticle-num}

\end{document}